%% file: icml.tex
\documentclass{article}
\usepackage[accepted]{icml2024}
\usepackage{microtype}
\usepackage{graphicx}
\usepackage{subfigure}
\usepackage{tcolorbox}
\usepackage{enumitem}
\usepackage{booktabs} 
\usepackage{multirow}
\usepackage{changepage}

\usepackage{amsmath}
\usepackage{amssymb}
\usepackage{mathtools}
\usepackage{amsthm}

\newtheorem{theorem}{Theorem}[section]
\newtheorem{lemma}{Lemma}[section]
\newtheorem{proposition}{Proposition}[section]

\newtheorem{assumption}{Assumption}[section]
\newtheorem{remark}{Remark}[section]
\newtheorem{definition}{Definition}[section]

\usepackage[textsize=tiny]{todonotes}

\usepackage{color}
\definecolor{salmon}{rgb}{1.0, 0.55, 0.41}

\usepackage[colorlinks=true,breaklinks=true,bookmarks=true,urlcolor=teal,citecolor=teal,linkcolor=teal,bookmarksopen=false,draft=false]{hyperref}  
\usepackage[nameinlink, capitalize,noabbrev]{cleveref}
\crefname{assumption}{Assumption}{Assumptions}
\crefname{Assumption}{Assumption}{Assumptions}
\crefname{Lemma}{Lemma}{Lemmata}
\crefname{lemma}{Lemma}{Lemmata}
\crefname{Theorem}{Theorem}{Theorems}
\crefname{theorem}{Theorem}{Theorems}
\crefname{corollary}{Corollary}{Corollaries}
\crefname{proposition}{Proposition}{Propositions}
\crefname{claim}{Claim}{Claims}
\crefname{procedure}{Procedure}{Procedures}
\crefname{algorithm}{Algorithm}{Algorithms}
\crefname{figure}{Figure}{Figures}
\crefname{remark}{Remark}{Remarks}
\crefname{section}{Section}{Sections}
\crefname{procedure}{Procedure}{Procedures}
\crefname{definition}{Definition}{Definitions}
\crefname{example}{Example}{Examples}
\crefname{table}{Table}{Tables}
\crefname{equation}{}{}
\crefname{enumi}{}{}
\crefname{item}{Item}{Items}

\usepackage{xcolor}
\usepackage{colortbl}
\usepackage{tabularx, makecell, booktabs}

\usepackage[toc,page,header]{appendix}
\usepackage[]{minitoc}

\usepackage{fancyhdr}
\usepackage{ifthen}
\newboolean{ICML}
\newboolean{Arxiv}
\setboolean{ICML}{true}   
\setboolean{Arxiv}{false} 

\icmltitlerunning{A Single-loop Robust Policy Gradient Method for Robust Markov Decision Processes}

\begin{document}
\include{macros}
\doparttoc[n]
\faketableofcontents 

\twocolumn[
\icmltitle{A Single-Loop Robust Policy Gradient Method for Robust Markov Decision Processes}




\icmlsetsymbol{equal}{*}

\begin{icmlauthorlist}
\icmlauthor{Zhenwei Lin}{equal,sufe}
\icmlauthor{Chenyu Xue}{equal,sufe}
\icmlauthor{Qi Deng}{sjtu}
\icmlauthor{Yinyu Ye}{stanford}
\end{icmlauthorlist}

\icmlaffiliation{sufe}{Shanghai University of Finance and Economics}
\icmlaffiliation{sjtu}{Antai College of Economics and Management, Shanghai Jiao Tong University}
\icmlaffiliation{stanford}{Stanford University}

\icmlcorrespondingauthor{Qi Deng}{qdeng24@sjtu.edu.cn}

\icmlkeywords{Machine Learning, ICML}

\vskip 0.3in
]

\printAffiliationsAndNotice{\icmlEqualContribution}

\input{icml_abstract}
\input{icml_main}

\input{icml_numerical}
\input{icml_conclusion}

\bibliography{ref}
\bibliographystyle{icml2024}


\input{icml_appendix}
\input{icml_experiment_appendix}

\end{document}

%% file: macros.tex
\global\long\def\lrinprod#1#2{\text{\ensuremath{\langle#1,#2\rangle}}}%

\global\long\def\inner#1#2{\langle#1,#2\rangle}%

\global\long\def\binner#1#2{\big\langle#1,#2\big\rangle}%

\global\long\def\bbinprod#1#2{\bigg<#1,#2\bigg>}%

\global\long\def\Binner#1#2{\text{\ensuremath{\Big<}#1,#2\ensuremath{\Big>}}}%

\global\long\def\norm#1{\Vert#1\Vert}%
\global\long\def\bnorm#1{\big\Vert#1\big\Vert}%
\global\long\def\Bnorm#1{\Big\Vert#1\Big\Vert}%
\global\long\def\red#1{\textcolor{red}{#1}}%
\global\long\def\blue#1{\textcolor{blue}{#1}}%
\global\long\def\pink#1{\textcolor{pink}{#1}}%
\global\long\def\yellow#1{\textcolor{yellow}{#1}}%
\global\long\def\green#1{\textcolor{green}{#1}}%
\global\long\def\orange#1{\textcolor{orange}{#1}}%
\global\long\def\purple#1{\textcolor{purple}{#1}}%
\global\long\def\brown#1{\textcolor{brown}{#1}}%
\global\long\def\indigo#1{\textcolor{indigo}{#1}}%
\global\long\def\teal#1{\textcolor{teal}{#1}}%

\global\long\def\brbra#1{\big(#1\big)}%
\global\long\def\Brbra#1{\Big(#1\Big)}%
\global\long\def\rbra#1{(#1)}%

\global\long\def\sbra#1{[#1]}%
\global\long\def\bsbra#1{\big[#1\big]}%
\global\long\def\Bsbra#1{\Big[#1\Big]}%
\global\long\def\abs#1{\vert#1\vert}%
\global\long\def\babs#1{\big\vert#1\big\vert}%
\global\long\def\Babs#1{\Big|#1\Big|}%

\global\long\def\cbra#1{\{#1\}}%
\global\long\def\bcbra#1{\big\{#1\big\}}%
\global\long\def\Bcbra#1{\Big\{#1\Big\}}%

\global\long\def\bbcbra#1{\bigg\{#1\bigg\}}%
\global\long\def\bbrbra#1{\bigg(#1\bigg)}%

\global\long\def\vertiii#1{\left\vert \kern-0.25ex  \left\vert \kern-0.25ex  \left\vert #1\right\vert \kern-0.25ex  \right\vert \kern-0.25ex  \right\vert }%
\global\long\def\matr#1{\bm{#1}}%
\global\long\def\til#1{\tilde{#1}}%
\global\long\def\wtil#1{\widetilde{#1}}%
\global\long\def\wh#1{\widehat{#1}}%
\global\long\def\mcal#1{\mathcal{#1}}%
\global\long\def\mbb#1{\mathbb{#1}}%
\global\long\def\mtt#1{\mathtt{#1}}%
\global\long\def\ttt#1{\texttt{#1}}%
\global\long\def\dtxt{\textrm{d}}%
\global\long\def\bignorm#1{\bigl\Vert#1\bigr\Vert}%
\global\long\def\Bignorm#1{\Bigl\Vert#1\Bigr\Vert}%
\global\long\def\rmn#1#2{\mathbb{R}^{#1\times#2}}%
\global\long\def\deri#1#2{\frac{d#1}{d#2}}%
\global\long\def\pderi#1#2{\frac{\partial#1}{\partial#2}}%
\global\long\def\limk{\lim_{k\rightarrow\infty}}%
\global\long\def\trans{\textrm{T}}%
\global\long\def\onebf{\mathbf{1}}%
\global\long\def\Bbb{\mathbb{B}}%
\global\long\def\hbf{\mathbf{h}}%
\global\long\def\zerobf{\mathbf{0}}%
\global\long\def\zero{\bm{0}}%

\global\long\def\Euc{\mathrm{E}}%
\global\long\def\Expe{\mathbb{E}}%
\global\long\def\rank{\mathrm{rank}}%
\global\long\def\range{\mathrm{range}}%
\global\long\def\diam{\mathrm{diam}}%
\global\long\def\epi{\mathrm{epi} }%
\global\long\def\inte{\operatornamewithlimits{int}}%
\global\long\def\dist{\operatornamewithlimits{dist}}%
\global\long\def\proj{\operatorname{Proj}}%
\global\long\def\cov{\mathrm{Cov}}%
\global\long\def\argmin{\operatornamewithlimits{argmin}}%
\global\long\def\argmax{\operatornamewithlimits{argmax}}%
\global\long\def\tr{\operatornamewithlimits{tr}}%
\global\long\def\dis{\operatornamewithlimits{dist}}%
\global\long\def\sign{\operatornamewithlimits{sign}}%
\global\long\def\prob{\mathrm{Prob}}%
\global\long\def\st{\operatornamewithlimits{s.t.}}%
\global\long\def\dom{\mathrm{dom}}%
\global\long\def\prox{\mathrm{prox}}%
\global\long\def\diag{\mathrm{diag}}%
\global\long\def\and{\mathrm{and}}%
\global\long\def\aleq{\overset{(a)}{\leq}}%
\global\long\def\aeq{\overset{(a)}{=}}%
\global\long\def\ageq{\overset{(a)}{\geq}}%
\global\long\def\bleq{\overset{(b)}{\leq}}%
\global\long\def\beq{\overset{(b)}{=}}%
\global\long\def\bgeq{\overset{(b)}{\geq}}%
\global\long\def\cleq{\overset{(c)}{\leq}}%
\global\long\def\ceq{\overset{(c)}{=}}%
\global\long\def\cgeq{\overset{(c)}{\geq}}%
\global\long\def\dleq{\overset{(d)}{\leq}}%
\global\long\def\deq{\overset{(d)}{=}}%
\global\long\def\dgeq{\overset{(d)}{\geq}}%
\global\long\def\eleq{\overset{(e)}{\leq}}%
\global\long\def\eeq{\overset{(e)}{=}}%
\global\long\def\egeq{\overset{(e)}{\geq}}%
\global\long\def\fleq{\overset{(f)}{\leq}}%
\global\long\def\feq{\overset{(f)}{=}}%
\global\long\def\fgeq{\overset{(f)}{\geq}}%
\global\long\def\gleq{\overset{(g)}{\leq}}%
\global\long\def\as{\textup{a.s.}}%
\global\long\def\ae{\textup{a.e.}}%
\global\long\def\Var{\operatornamewithlimits{Var}}%
\global\long\def\clip{\operatorname{clip}}%
\global\long\def\conv{\operatorname{conv}}%
\global\long\def\Cov{\operatornamewithlimits{Cov}}%
\global\long\def\raw{\rightarrow}%
\global\long\def\law{\leftarrow}%
\global\long\def\Raw{\Rightarrow}%
\global\long\def\Law{\Leftarrow}%
\global\long\def\vep{\varepsilon}%
\global\long\def\dom{\operatornamewithlimits{dom}}%
\global\long\def\tsum{{\textstyle {\sum}}}%
\global\long\def\Cbb{\mathbb{C}}%
\global\long\def\Ebb{\mathbb{E}}%
\global\long\def\Fbb{\mathbb{F}}%
\global\long\def\Nbb{\mathbb{N}}%
\global\long\def\Rbb{\mathbb{R}}%
\global\long\def\extR{\widebar{\mathbb{R}}}%
\global\long\def\Pbb{\mathbb{P}}%
\global\long\def\Mrm{\mathrm{M}}%
\global\long\def\Acal{\mathcal{A}}%
\global\long\def\Bcal{\mathcal{B}}%
\global\long\def\Ccal{\mathcal{C}}%
\global\long\def\Dcal{\mathcal{D}}%
\global\long\def\Ecal{\mathcal{E}}%
\global\long\def\Fcal{\mathcal{F}}%
\global\long\def\Gcal{\mathcal{G}}%
\global\long\def\Hcal{\mathcal{H}}%
\global\long\def\Ical{\mathcal{I}}%
\global\long\def\Kcal{\mathcal{K}}%
\global\long\def\Lcal{\mathcal{L}}%
\global\long\def\Mcal{\mathcal{M}}%
\global\long\def\Ncal{\mathcal{N}}%
\global\long\def\Ocal{\mathcal{O}}%
\global\long\def\Pcal{\mathcal{P}}%
\global\long\def\Scal{\mathcal{S}}%
\global\long\def\Tcal{\mathcal{T}}%
\global\long\def\Xcal{\mathcal{X}}%
\global\long\def\Ycal{\mathcal{Y}}%
\global\long\def\Zcal{\mathcal{Z}}%
\global\long\def\i{i}%

\global\long\def\abf{\mathbf{a}}%
\global\long\def\bbf{\mathbf{b}}%
\global\long\def\cbf{\mathbf{c}}%
\global\long\def\fbf{\mathbf{f}}%
\global\long\def\qbf{\mathbf{q}}%
\global\long\def\gbf{\mathbf{g}}%
\global\long\def\ebf{\mathbf{e}}%
\global\long\def\lambf{\bm{\lambda}}%
\global\long\def\alphabf{\bm{\alpha}}%
\global\long\def\sigmabf{\bm{\sigma}}%
\global\long\def\thetabf{\bm{\theta}}%
\global\long\def\deltabf{\bm{\delta}}%
\global\long\def\lbf{\mathbf{l}}%
\global\long\def\ubf{\mathbf{u}}%
\global\long\def\pbf{\mathbf{\mathbf{p}}}%
\global\long\def\vbf{\mathbf{v}}%
\global\long\def\wbf{\mathbf{w}}%
\global\long\def\xbf{\mathbf{x}}%
\global\long\def\ybf{\mathbf{y}}%
\global\long\def\zbf{\mathbf{z}}%
\global\long\def\dbf{\mathbf{d}}%
\global\long\def\Wbf{\mathbf{W}}%
\global\long\def\Abf{\mathbf{A}}%
\global\long\def\Ubf{\mathbf{U}}%
\global\long\def\Pbf{\mathbf{P}}%
\global\long\def\Ibf{\mathbf{I}}%
\global\long\def\Ebf{\mathbf{E}}%
\global\long\def\sbf{\mathbf{s}}%
\global\long\def\Mbf{\mathbf{M}}%
\global\long\def\Dbf{\mathbf{D}}%
\global\long\def\Qbf{\mathbf{Q}}%
\global\long\def\Lbf{\mathbf{L}}%
\global\long\def\Pbf{\mathbf{P}}%
\global\long\def\Xbf{\mathbf{X}}%

\global\long\def\abm{\bm{a}}%
\global\long\def\bbm{\bm{b}}%
\global\long\def\cbm{\bm{c}}%
\global\long\def\dbm{\bm{d}}%
\global\long\def\ebm{\bm{e}}%
\global\long\def\fbm{\bm{f}}%
\global\long\def\gbm{\bm{g}}%
\global\long\def\hbm{\bm{h}}%
\global\long\def\pbm{\bm{p}}%
\global\long\def\qbm{\bm{q}}%
\global\long\def\rbm{\bm{r}}%
\global\long\def\sbm{\bm{s}}%
\global\long\def\tbm{\bm{t}}%
\global\long\def\ubm{\bm{u}}%
\global\long\def\vbm{\bm{v}}%
\global\long\def\wbm{\bm{w}}%
\global\long\def\xbm{\bm{x}}%
\global\long\def\ybm{\bm{y}}%
\global\long\def\zbm{\bm{z}}%
\global\long\def\Abm{\bm{A}}%
\global\long\def\Bbm{\bm{B}}%
\global\long\def\Cbm{\bm{C}}%
\global\long\def\Dbm{\bm{D}}%
\global\long\def\Ebm{\bm{E}}%
\global\long\def\Fbm{\bm{F}}%
\global\long\def\Gbm{\bm{G}}%
\global\long\def\Hbm{\bm{H}}%
\global\long\def\Ibm{\bm{I}}%
\global\long\def\Jbm{\bm{J}}%
\global\long\def\Lbm{\bm{L}}%
\global\long\def\Obm{\bm{O}}%
\global\long\def\Pbm{\bm{P}}%
\global\long\def\Qbm{\bm{Q}}%
\global\long\def\Rbm{\bm{R}}%
\global\long\def\Ubm{\bm{U}}%
\global\long\def\Vbm{\bm{V}}%
\global\long\def\Wbm{\bm{W}}%
\global\long\def\Xbm{\bm{X}}%
\global\long\def\Ybm{\bm{Y}}%
\global\long\def\Zbm{\bm{Z}}%
\global\long\def\lambm{\bm{\lambda}}%
\global\long\def\alphabm{\bm{\alpha}}%
\global\long\def\albm{\bm{\alpha}}%
\global\long\def\taubm{\bm{\tau}}%
\global\long\def\mubm{\bm{\mu}}%
\global\long\def\yrm{\mathrm{y}}%
\global\long\def\iid{i.i.d}%
\global\long\def\vec{\operatorname{vec}}%
\global\long\def\rone{\textrm{I}}%
\global\long\def\rtwo{\textrm{II}}%
\global\long\def\rthree{\textrm{III}}%
\global\long\def\rfour{\textrm{IV}}%
\global\long\def\rfive{\textrm{V}}%
\global\long\def\rsix{\textrm{VI}}%
\global\long\def\rseven{\textrm{VII}}%
\global\long\def\reight{\textrm{VIII}}%
\global\long\def\rnine{\textrm{IX}}%
\global\long\def\rten{\textrm{X}}%
\global\long\def\hess{\operatorname{Hess}}%
\global\long\def\grad{\operatorname{grad}}%
\global\long\def\brho{\boldsymbol{\rho}}
\global\long\def\srpg{\text{SRPG}}
\global\long\def\garnet{\text{GARNET}}

%% file: icml_abstract.tex
\begin{abstract}
Robust Markov Decision Processes (RMDPs) have recently been recognized as a valuable and promising approach to discovering a policy with creditable performance, particularly in the presence of a dynamic environment and estimation errors in the transition matrix due to limited data. 
Despite extensive exploration of dynamic programming algorithms for solving RMDPs, there has been a notable upswing in interest in developing efficient algorithms using the policy gradient method.
In this paper, we propose the first single-loop robust policy gradient (SRPG) method with the global optimality guarantee for solving RMDPs through its minimax formulation. 
Moreover, we complement the convergence analysis of the nonconvex-nonconcave min-max optimization problem with the objective function's gradient dominance property, which is not explored in the prior literature.
Numerical experiments validate the efficacy of SRPG, demonstrating its faster and more robust convergence behavior compared to its nested-loop counterpart.
\end{abstract}

%% file: icml_main.tex
\section{Introduction}

Markov decision processes (MDPs) serve as an important model for sequential decision-making under uncertainty and enjoys wide applications in finance \cite{deng2016deep, jiang2017deep}, autonomous driving \cite{kiran2021deep, sallab2017deep}, and revenue management \cite{den2015dynamic}, etc. However, in most applications, the decision-maker can only estimate model parameters, especially the transition kernel, from noisy and scarce observation data. Consequently, a policy that exhibits poor performance with respect to the true parameters may be employed since the optimal policy obtained from the estimated parameters could be highly sensitive to the small changes in the problem parameters, resulting in suboptimal outcomes \cite{goyal2023beyond}. 

Motivated by recent developments in robust optimization and its practical performance \cite{ben2002robust, bertsimas2004price}, robust Markov Decision Processes (RMDPs) have emerged as a valuable and promising approach to overcome this obstacle. It assumes that the transition kernel lies in a pre-determined ambiguity set and then seeks a policy with the best performance under the worst-case transition kernel. Hence, it involves solving a min-max optimization problem. Moreover, it has been demonstrated that the optimal policies of RMDPs display an advantageous performance in out-of-sample scenarios when the transition kernel needs to be estimated from limited data or undergoes changes over time \cite{xu2009parametric, ghavamzadeh2016safe, mannor2016robust}.

With general uncertainty sets, \citet{wiesemann2013robust} prove that it is NP-hard to find the optimal policy for RMDPs. However, under certain rectangular assumptions of the ambiguity set, the problem becomes tractable. For example, when the ambiguity set is $(s, a)$-rectangular, the dynamic programming techniques apply, and the value iteration is known to achieve linear convergence to optimal robust value~\cite{iyengar2005robust, nilim2005robust}. Here, the $(s, a)$-rectangular ambiguity set allows the adversarial nature to choose the worst-case transition probability vector of each state and action pair independently. Since the $(s, a)$-rectangular assumption is too restrictive and leads to conservative policies, we consider the more general $s$-rectangular ambiguity set, which allows the nature to choose the transition kernel for each state without observing the action and also preserves tractability \cite{wiesemann2013robust}.

Nowadays, the policy gradient (PG) method has become the workhorse for solving a special case of RMDPs, where there is no ambiguity in the transition kernel. It is scalable, easy to implement, and versatile across various settings, including model-free and continuous state-action spaces \cite{kakade2002approximately, schulman2017equivalence}. However, the policy gradient approach to solving general RMDPs has been much less investigated. 
Recently, \citet{wang2023policy} propose a double-loop robust policy gradient (DRPG) for solving $s$-rectangular RMDPs. 
The outer loop of DRPG is designed for updating policies, which resembles policy gradient updates in non-robust MDPs, while the inner loop is to solve the maximization problem with a given policy over the ambiguity set and update the worst-case transition matrices. They prove that DRPG is guaranteed to converge to a globally optimal policy within $\mathcal{O}(1/\vep^{4})$ outer iterations. However, the critical drawback of DRPG is that the inner loop needs to solve a series of nontrivial nonconcave subproblems with increasing accuracy, which brings an unacceptable computation burden and a worse convergence complexity than $\mathcal{O}(1/\vep^{4})$.

In this paper, our goal is to develop a more efficient algorithm for RMDPs that significantly improves the convergence complexity of the existing work.
Our main contributions are summarized as follows.

First, we present a single-loop algorithm, Single-loop Robust Policy Gradient ({\srpg}), for RMDPs with an $s$-rectangular ambiguity set.  
Our algorithm is adapted from the primal-dual algorithm for min-max optimization~\citep{zhang2020single, zheng2023universal}, which employs the Moreau-Yosida smoothing technique to enhance the balance between the primal and dual updates. 
To our knowledge, \srpg{} is the first policy gradient method for RMDPs without involving a more complicated nested loop. 
Consequently, our algorithm not only obtains the best iteration complexity of policy gradient methods for RMDPs with $s$-rectangular ambiguity sets, but also has a much faster update per iteration than the double-loop algorithm proposed in \citet{wang2023policy}. Our numerical experiments further show the superior performance of SRPG.
    
Second, we prove the {\srpg} converges at rate $\mcal O(1/\vep^{4})$ under gradient dominance property. This result appears to be new even for general nonconvex-nonconcave min-max optimization. 
To our knowledge, such convergence rate for min-max optimization with gradient dominance property was only obtained when the max component is a concave function. While nonconvex-concave min-max optimization has a wide range of applications, unfortunately, it is not satisfied in the RMDP setting.

Finally, we provide a detailed analysis of the varying convergence rates of {\srpg} when applied to RMDPs. Leveraging the gradient-dominance-like property for nonsmooth weakly convex functions, we demonstrate that the objective, as a function of $\pi$, exhibits exponential growth within a specific region. This observation implies an accelerated convergence rate when the iterative sequence is distant from the optimal solution set. While this does not improve the convergence rate in terms of $\vep$, it helps explain the initially faster convergence observed in our experiments.

The paper is organized in the following manner. In the remainder of this section, we review the literature related to our work. In \cref{sec:pre}, we introduce some preliminaries and notations and make some assumptions for the later analysis. \cref{sec:properties_rmdp} provides some important properties of RMDPs under $s$-rectangular ambiguity set. \cref{sec:dsgda} presents the specific details of {\srpg} for RMDPs and the comprehensive convergence analysis for {\srpg}.
In \cref{sec:num}, we conduct two numerical experiments to examine the convergence performance of {\srpg}. Finally, we draw the conclusion in \cref{sec:num}. All omitted proofs and details of numerical experiments can be found in the Appendix.

\subsection{Literature Review}

Our paper is related to two streams of the literature, including the policy gradient method applied to RMDPs and min-max optimization.

\textbf{Policy gradient for RMDPs.} Recently, there has been a surge of interest in the policy gradient method for RMDPs with different ambiguity sets. 
\citet{wang2022policy} propose a policy gradient method for solving RMDPs under a $r$-contamination ambiguity set, which is restrictive compared with $s$-rectangular ambiguity set, and the problem can be reduced to an ordinary MDP~\cite{wang2023policy}. 
For $(s, a)$-rectangular ambiguity set, \citet{li2022first} develop an algorithm with a linear convergence rate. 
However, their analysis depends on the $(s, a)$-rectangular assumption and can not be applied to our $s$-rectangular set. 
\citet{wang2024bring} focus on the KL ambiguity set and approximate the worst transition kernel at each iteration instead of fully solving the inner problem.
There are also some papers considering different structured $s$-rectangular ambiguity sets \cite{grand2021scalable, kumar2023policy, li2023first}. 
\citet{grand2021scalable} propose an algorithm interleaving primal-dual first-order updates with approximate value iteration updates and prove its ergodic convergence for ellipsoidal and Kullback-Leibler $s$-rectangular ambiguity sets. 
\citet{kumar2023policy} consider the $L_p$-ball $s$-rectangular ambiguity set, give a closed-form expression of the worst-case transition kernel, and propose the robust policy gradient method. 
However, they mainly discuss the time complexity involved in computing the robust policy gradient without providing the convergence analysis of the overall algorithm. 
\citet{li2023first} introduce a new type of $s$-rectangular ambiguity set: a convex combination of a fixed (possibly unknown) transition kernel and another transition kernel belonging to a pre-specified convex set. 
They focus on the policy evaluation problem and formulate it as an MDP from the view of nature. 
For the general $s$-rectangular ambiguity set, \citet{kumar2023towards} propose an algorithm achieving the $\mathcal{O}(1/\vep)$ convergence rate. However, it relies on a strong assumption that the objective function of the minimization problem is smooth, which does not necessarily hold for many ambiguity sets. 
\citet{wang2023policy} and \citet{li2024policy} propose a nested-loop algorithm for solving $s$-rectangular and non-rectangular RMDPs, respectively, and they both obtain the $\mathcal{O}(1/\vep^{4})$ convergence rate. However, the inner loop needs to optimize over the ambiguity set of transition matrices, which requires high computational cost.  
To avoid such computation burden, we propose a single-loop robust policy gradient method for the general $s$-rectangular ambiguity set.

\textbf{Min-max optimization.}
Minimax optimization has garnered significant attention across various fields, including robust optimization~\cite{duchi2019variance}, game theory~\cite{bailey2020finite}, and adversarial machine learning~\cite{goodfellow2014explaining}. Although there is a comprehensive range of literature on minimax optimization, most prior research \cite{hamedani2021primal,ouyang2021lower,zhang2017stochastic} predominantly concentrates on the convex-convex setting. 
Recently, there has been a shift in focus towards nonconvex or nonconcave minimax problems \cite{zheng2023universal,xu2023unified,lin2020gradient,zhang2020single}. This emerging interest is primarily due to the prevalent occurrence of such problems in practical applications, as discussed in~\citet{razaviyayn2020nonconvex}.
The convergence guarantee of most current algorithms in this field depends on additional information such as one-side convexity, PL condition, weak Minty Variational Inequality~\cite{diakonikolas2021efficient}, positive interaction dominance~\cite{hajizadeh2023linear} and K\L{} condition~\cite{zheng2023universal}. However, these properties do not apply in the specific context of RMDPs. The gradient dominance property associated with RMDPs bears similarities to the K\L{} property yet extends beyond its scope.

\section{\label{sec:pre}Preliminaries and Notations}
Throughout the paper, we use $\Delta^d$ to denote the $d$-dimension probability simplex. With slight abuse of notation, we use $\norm{\cdot}$ to represent the Frobenius norm for matrix and $L_2$ norm for vector. We also use $[n]$ to denote the set $\{1, 2, \ldots, n\}$.

We first introduce the notations of MDPs. An ordinary MDP is specified by a tuple $(\mcal S,\mcal A,p,c,\gamma,\rho )$, where $\mcal S :=\bcbra{1,2,\cdots, S}$ is the finite state set of size $S$, $\mcal A := \bcbra{1,2,\cdots, A}$ is the action set of size $A$, $p = (p_{sa})_{s\in\mcal S,a\in\mcal A}\in(\Delta^{S})^{S\times A}$ is the transition kernel with $p_{sa}\in\Delta^{S}$ being the transition probability vector from a current state $s$ to a subsequent state $s^{\prime}$ after taking an action $a$, $ c = ( c_{cas^{\prime}} )_{s \in \mathcal{S}, a \in \mathcal{A}, s^{\prime} \in \mathcal{S}} $ is the cost of the aforementioned transition, $\gamma\in(0,1)$ is the discount factor, and $\rho\in\Delta^{S}$ is the distribution of the initial state. Moreover, a policy maps from state $s\in \mcal S$ to a distribution over action $a\in \mcal A$ is denoted by $\pi:=(\pi_s)_{s\in \mcal S}$. 

We further define the policy space $\Pi:=(\Delta^{A})^{S}$ and transition kernel space $\mathcal{P}:=(\Delta^{S})^{S\times A}$. We use $\mcal N_{\Pi}(\pi):=\bcbra{\tilde{\pi}\mid \inner{\tilde{\pi}}{\bar{\pi} - \pi}\leq 0,\forall \bar{\pi}\in \Pi}$ and $\mcal N_{\mathcal{P}}(p):=\bcbra{\tilde{p}\mid \inner{\tilde{p}}{\bar{p} - p}\leq 0,\forall \bar{p}\in \mathcal{P}}$ to denote the normal cones of $\Pi$ at policy $\pi$ and $\mathcal{P}$ at transition kernel $p$, respectively. We also denote proximal operator of $\phi(\cdot)$ as $\prox_{\phi,r_1}(\pi):=\argmin_{\tilde{\pi} \in \Pi}\phi(\tilde{\pi}) + {r_1}\norm{\tilde{\pi} - \pi}^2 / 2$. We use $\dist(v,\mcal V):=\min_{\bar{v}\in\mcal V}\norm{\bar{v}-v}$ to denote distance between vector or matrix $v$ and set $\mcal V$. We also define some notations in Table~\ref{tab:Notations1}, which will be used frequently in \cref{sec:dsgda}.
\begin{table}[htb]
\begin{centering}
\renewcommand{\arraystretch}{1.2}
\begin{tabular}{|l|l|}
\hline 
Notation & Meaning \\
\hline
$\phi(\pi)$ & $\max_{p\in \mcal P}J_{\rho}(\pi, p)$ \\
\hline
$\chi(\pi,p,\bar{\pi},\bar{p})$ & $J_{\rho}(\pi,p)+ \frac{r_{1}}{2} \norm{\pi-\bar{\pi}}^{2}  - \frac{r_{2}}{2} \norm{p-\bar{p}}^{2} $ \\
\hline 
$\varphi_{\pi}(p,\bar{\pi},\bar{p})$ & $\min_{\pi\in\Pi}\chi(\pi,p,\bar{\pi},\bar{p})$ \\
\hline  
$\varphi_{\pi,p}(\bar{\pi},\bar{p})$ & $\max_{p\in\mcal P}\varphi_{\pi}(p,\bar{\pi},\bar{p})$\\
\hline 
$\varphi_{\pi,p,\bar{\pi}}(\bar{p})$ & $\min_{\bar{\pi}}\varphi_{\pi,p}(\bar{\pi},\bar{p})$ \\
\hline 
$\varphi_{\pi,p,\bar{p}}(\bar{\pi})$ & $\max_{\bar{p}}\varphi_{\pi,p}(\bar{\pi},\bar{p})$\\
\hline 
$\underline{\chi}$ & $\min_{\bar{\pi}}\max_{\bar{p}}\varphi_{\pi,p}(\bar{\pi},\bar{p})$ \\
\hline 
$p^{+}(\bar{\pi},\bar{p})$ & $\proj_{\mcal P}(p+\sigma\nabla_{p}\chi(\pi(p,\bar{\pi},\bar{p}),p,\bar{\pi},\bar{p}))$  \\
\hline 
$\bar{p}^{+}(\bar{\pi})$ & $\bar{p}+\mu\brbra{p(\pi(\bar{\pi},\bar{p}),\bar{\pi},\bar{p})-\bar{p}}$ \\
\hline 
$p^{+}(\bar{p})$ & $\proj_{\mcal P}(p+\sigma\nabla_{p}J_{\rho}(\pi(p,\bar{\pi};r_{1}),p))$ \\
\hline 
$\bar{p}(\bar{\pi})$ & $\argmax_{\bar{p}}\min_{\pi\in \Pi}\max_{p\in \mcal P}\chi (\pi, p, \bar{\pi}, \bar{p})$  \\
\hline 
$\pi(\bar{\pi}, \bar{p})$ & $\argmin_{\pi \in \Pi}\max_{p\in \mcal P}\chi (\pi, p, \bar{\pi}, \bar{p})$ \\
\hline 
\end{tabular}
\par\end{centering}
\caption{\label{tab:Notations1}Useful Notations.}
\end{table}

We now give some useful definitions in MDP. The discounted state occupancy measure $d_{\rho}^{\pi,p}:\mcal S\to[0,1]$
for an initial distribution $\rho$, a policy $\pi$, and a transition kernel $p$ is defined as $d_{\rho}^{\pi,p}\left(s^{\prime}\right)=(1-\gamma)\sum_{s\in\mathcal{S}}\sum_{t=0}^{\infty}\gamma^{t}\rho(s)p_{ss^{\prime}}^{\pi}(t)$, where $p_{ss^{\prime}}^{\pi}(t)$ is the probability of arriving in a state $s^{\prime}$ after transiting $t$ time steps from state $s$ according to policy $\pi$ and transition kernel $p$. The performance of a policy $\pi$ is measured by the value function $v^{\pi, p}:=[v_1^{\pi, p}, \ldots, v_S^{\pi, p}]^\top \in \mbb R^{S}$, where $v_s^{\pi, p}$ is defined by
$v_{s}^{\pi,p}=\mbb E_{\pi,p}\bsbra{\sum_{t=0}^{\infty}\gamma^{t} c_{s_{t}a_{t}s_{t+1}}\mid s_{0}=s}$. The value function of taking action $a$ at state $s$, namely the action value function, is defined by $q_{sa}^{\pi,p}=\mbb E_{\pi,p}\bsbra{\sum_{t=0}^{\infty}\gamma^{t}c_{s_{t}a_{t}s_{t+1}}\mid s_{0}=s,a_{0}=a}$. The objective of an ordinary MDP is to find an optimal policy $\pi^{*}\in\mbb R^{S\times A}$ to minimize the total expected cost: $\pi^{*}=\argmin_{\pi\in\Pi}\mbb E_{\pi,p,s_{0}\sim\rho}\bsbra{\sum_{t=0}^{\infty}\gamma^{t}c_{s_{t}a_{t}s_{t+1}}}$.

In real applications, the transition kernel is typically unknown, but the decision-maker may have some partial knowledge of it and can build a corresponding ambiguity set. This can be modeled by an RMDP, characterized by $( \mcal S, \mcal A, \mcal P, c, \gamma, \rho )$, where the transition kernel $p$ is replaced by the ambiguity set $\mcal P$. The goal of an RMDP is to find an optimal policy under the worst-case transition kernel: 
\begin{equation}\label{eq:J_obj}
\min_{\pi\in\Pi}\max_{p\in\mcal P}J_{\rho}(\pi,p):=\sum_{s\in\mcal S}\rho_{s}v_{s}^{\pi,p}.
\end{equation}

In the rest of this section, we introduce some assumptions used throughout the paper.
\begin{assumption}
[Bounded Cost]\label{assu:bounded_cost}For any $(s,a,s^{\prime})\in\mcal S\times\mcal A\times\mcal S$,
the cost $c_{sas^{\prime}}\in[0,1]$.
\end{assumption}
\begin{assumption}[$s$-Rectangular Ambiguity Set]
\label{assu:s-rectangular}
The ambiguity set $\mcal P$ is convex and $s$-rectangular, namely the transition probabilities are independent at different
states, i.e., $\mcal P$ can be decomposed as $\mcal P=\mcal P_1{\times}\cdots \times \mcal P_{S}$,
where $\mcal P_{s}\subseteq\mbb R_{+}^{S\times A}$ for all $s = 1, 2, \ldots, S$.
\end{assumption}
\begin{assumption}[Bounded Distribution Mismatch Coefficient]
\label{assu:irreduc}
Define the distribution mismatch coefficient by $D = \|d_{\rho}^{\pi,p} / \rho \|_{\infty}$. It is finite for the worst-case transition kernel $p^*(\pi) := \arg\max_{p \in \mathcal{P}} J_{\rho}(\pi, p)$ for all $\pi\in \Pi$, and for the policy $\pi^*(p) := \arg\min_{\pi \in \Pi} J_{\rho}(\pi, p)$ for all $p \in \mathcal{P}$.
\end{assumption}

The first assumption is standard in the existing literature \cite{wang2023policy}, and does not change the optimal policy set. The second assumption suggests that the ambiguity set is representable as a Cartesian product of separate ambiguity sets for the transition probability matrix associated with distinct states, and is widely used. The last assumption implies the bounded distribution mismatch coefficient, which is crucial for \cref{thm:Moreau_smooth_grad_dom} in \cref{sec:properties_rmdp}.

\section{\label{sec:properties_rmdp}Properties of RMDPs}
In this section, we discuss several key properties of RMDPs for developing our algorithm design and convergence analysis. 
\begin{lemma}
\label{lem:smoothmdp}
For general RMDPs, we have:
\begin{enumerate}[label=\roman*)]
\item \label{item:i)}$J_{\rho}(\pi,p)$ is $L_{\pi}$-Lipschitz
and $\ell_{\pi}$-smooth in $\pi$ with $L_{\pi}:=\sqrt{A}/(1-\gamma)^{2}$
and $\ell_{\pi}:=2\gamma A/(1-\gamma)^{3}$.
\item \label{item:iii)}$\phi(\pi):=\max_{p\in\mcal P}J_{\rho}(\pi,p)$
is $L_\pi$-Lipschitz and $\ell_{\pi}$-weakly convex, i.e., $\phi(\pi) + \ell_{\pi} \norm{\pi}^{2} / 2$ is convex.
\end{enumerate}
Furthermore, when the ambiguity set satisfies \cref{assu:s-rectangular}, the following statements also hold:
\begin{enumerate}[label=\roman*)]
\setcounter{enumi}{2}
\item \label{item:ii)}$ J_{\rho}(\pi,p)$ is $L_{p}$-Lipschitz and $\ell_{p}$-smooth in $p$ with $L_{p}:=\sqrt{SA}/(1-\gamma)^{2}$
and $\ell_{p}:=2\gamma S/(1-\gamma)^{3}$.
\item \label{item:iv)} There exist two positive constants $\ell_{\pi p},\ell_{p\pi}>0$ such that for $\forall \pi,\pi'\in\Pi$ and $\forall p,p'\in\mcal P$, we have
{\footnotesize
\begin{align*}
\hspace*{-2em}\|\nabla_{\pi}J_{\rho}(\pi,p)-\nabla_{\pi}J_{\rho}(\pi^{\prime},p^{\prime})\| & \leq\ell_{\pi p}(\|\pi-\pi'\|+\|p-p'\|),\\
\hspace*{-2em}\|\nabla_{p}J_{\rho}(\pi,p)-\nabla_{p}J_{\rho}(\pi^{\prime},p^{\prime})\| & \leq\ell_{p\pi}(\|\pi-\pi'\|+\|p-p'\|).
\end{align*}
}
\end{enumerate}
\end{lemma}
The above lemma establishes that the objective function $J_{\rho}(\pi,p)$ is two-sided Lipschitz and smooth, and the induced function $\phi(\pi)$ is weakly-convex. Moreover, we discover a stronger Lipschitz condition satisfied by the gradient of $J_{\rho}(\pi,p)$ than prior results.
\begin{remark}
While statements i)-iii) have been established in existing literature \cite{agarwal2021theory, wang2023policy}, the Lipschitz gradient condition iv) is stronger than the results in \citet{wang2023policy}, which only proves $J_{\rho}(\pi,\rho)$ to be partially smooth for $\pi$ and $p$ when fixing the other variable. To interpret statement iv), it indicates the bounded matrix norm of the Hessian of $J_{\rho}(\pi,p)$, which is critical for the convergence analysis of our single-loop algorithm. 
\end{remark}

\subsection{Gradient dominance property}

In this subsection, we focus on the gradient dominance property satisfied by RMDPs.

For general RMDPs, the objective function exhibits a gradient dominance property with respect to $\pi$. Specifically, there exists a constant $\bar{D}_\pi>0$ such that for all $ p \in \mcal P, \pi \in \Pi$:
\begin{equation*}
    J_{\rho}(\pi,p)-\Psi(p) \leq\bar{D}_{\pi}\dist\brbra{0,\nabla_{\pi}J_{\rho}(\pi,p)+\mcal N_{\Pi}(\pi)}, 
\end{equation*}
where $\Psi(p) = \min_{\pi \in \Pi} J_{\rho}(\pi, p) $.
This is the key property to ensure global convergence of policy gradient method for non-robust MDP \cite{agarwal2021theory}, and it is also important for our further analysis. Meanwhile, when the ambiguity set in RMDPs satisfies \cref{assu:s-rectangular}, the gradient dominance property also holds for $J_\rho(\pi, p)$ with respect to $p$, i.e., for all $ p \in \mcal P, \pi \in \Pi$ with $\bar{D}_{p}>0$:
\begin{equation}\label{eq:partial_grad_dom_p}
    \phi(\pi)-J_{\rho}(\pi,p)  \leq\bar{D}_{p}\dist\brbra{0,\nabla_{p}J_{\rho}(\pi,p)-\mcal N_{\mcal P}(p)}.
\end{equation}

Although the function $\phi(\pi)$ in RMDPs may not have the desired gradient dominance property, its weak convexity, combined with the property of $J_\rho(\pi,p)$, make the Moreau envelope function $\phi_{2\ell_\pi}(\pi)$ satisfy a gradient-dominance-like property \cite{wang2023policy}. 

\begin{theorem}[Informal]
\label{thm:Moreau_smooth_grad_dom}
Let $\pi^{*}$ be a global optimal policy for RMDPs. Then there exists a constant $C_{\ell_{\pi}}>0$ such that for any policy $\pi\in\Pi$:
\begin{equation}\label{eq:gradient-like-property}
\phi_{2\ell_{\pi}}(\pi) \leq \phi(\pi)\leq C_{\ell_{\pi}} \norm{\nabla\phi_{2\ell_{\pi}}(\pi)}+\phi(\pi^{*}),
\end{equation}
where $\phi_{2\ell_{\pi}}(\pi):=\inf_{\pi'\in\Pi}\{\phi(\pi')+\ell_{\pi}\norm{\pi-\pi'}^{2}\}$.
\end{theorem}
\begin{remark}
For weakly convex optimization,
it has been shown by \citet{davis2019stochastic} that iterative algorithms, such as the proximal subgradient method, implicitly optimize the smooth approximation of the weakly convex function given by the Moreau envelope.
The gradient-dominance-like property implies that any first-order stationary point of the Moreau envelope function is globally optimal. 
Hence, it is the key prerequisite for the global optimality established in \cref{thm:global_opt} later. 
\end{remark}

\begin{remark}
    We remark that $\phi(\pi)\ge \phi_{2\ell_{\pi}}(\pi)$  holds for any $\pi\in \Pi$ and $\phi(\pi^*) = \phi_{2\ell_{\pi}}(\pi^*)$~\cite{davis2019stochastic}, which means that $\phi_{2\ell_{\pi}}(\pi)$ is also gradient-dominated. This inspires the fine-grained characterization of region-dependent convergence rate discussed in \cref{sec:reg_exp}.
\end{remark}

Additionally, a crucial and general conclusion arising from the gradient dominance property and the Moreau envelope is presented as follows.
\begin{theorem}
\label{thm:moreau_grad_dom}
Suppose that a $L_{f}$-Lipschitz continuous, $\ell_f$-weakly
convex function $f(\cdot)$ also satisfies the gradient dominance property with a constant $C>0$ on set $\mcal X\subseteq \dom f$, i.e., $f(x)-f(x^{*})\leq C\dist\brbra{\partial f(x)+{\cal N}_{\mcal X}(x)}$. Then we have the Moreau envelope function of $f$
with parameter $2r$, i.e., $f_{2r}(x)=\inf_{z\in\mcal X}f(z)+r\norm{z-x}^{2}$, also satisfies the gradient dominance property, where $r>\ell_f$.
\end{theorem}

The above theorem indicates that the Moreau envelope of a weakly convex and gradient-dominated function is also gradient-dominated. 
This extends the results in \citet{yu2022kurdyka}, which shows that the Bregman envelope of a K\L{} function is also a K\L{} function. We leave more discussion in \cref{sec:app_notations}. 
\cref{thm:moreau_grad_dom} plays an important role in the following convergence analysis and may be of independent interest for weakly convex optimization.

\section{\label{sec:dsgda}Single-loop Robust Policy Gradient Method}
In this section, we introduce our single-loop algorithm {\srpg} designed for RMDPs and present its convergence analysis. The motivation and details of {\srpg} are provided in \cref{sec:srpg}. \cref{sec:general_con} gives a standard convergence analysis of {\srpg}, showing a convergence rate of $\mcal O(1/\vep^{4})$ with the gradient dominance condition. In \cref{sec:reg_exp}, we reveal the interesting regional exponential growth condition induced by the gradient dominance property and provide a more precise analysis of the convergence rate. 
\subsection{The detailed algorithm\label{sec:srpg}}
A simple single-loop algorithm designed for min-max optimization is the well-known gradient descent ascent (GDA), which alternatively performs gradient descent on the minimization problem and gradient ascent on the maximization problem. It has been proved that GDA can converge to an $\vep$-stationary point for nonconvex-strongly-concave problem with an iteration complexity of $O(1/\vep^2)$~\citep{lin2020gradient}. However, since the objective function of RMDPs is generally nonconvex-nonconcave, GDA may suffer from oscillation and even diverge~\citep{zhang2020single}.

To address this problem, some recent works employ the Moreau-Yosida smoothing technique to make the iteration sequence stable and converge \cite{zhang2020single, yang2022faster, zheng2023universal}.
Inspired by these works, we consider the following regularized function,
$$
\chi(\pi,p,\bar{\pi},\bar{p}):=J_{\rho}(\pi,p)+\frac{r_{1}}{2}\norm{\pi-\bar{\pi}}^{2}-\frac{r_{2}}{2}\norm{p-\bar{p}}^{2}.
$$
This function introduces two ancillary variables $\bar{\pi}$ and $\bar{p}$ for the primal variable $\pi$ and dual variable $p$, respectively, and smooths the primal update and dual update simultaneously by incorporating two quadratic terms. It is crucial to choose appropriate $r_1$ and $r_2$ to ensure that this regularized function is convex in $\pi$ and concave in $p$. To balance the primal and dual updates, $r_1$ and $r_2$ are not necessarily equal.

We provide the details of our {\srpg} in \cref{alg:dsgda}. At each iteration, it performs gradient descent on $\pi$ and gradient ascent on $p$ based on the gradient of $\chi$ function. Then, the two auxiliary variables are exponentially averaged to make sure that they do not deviate far from $\pi$ and $p$, which contributes to the sequence stability.

\begin{algorithm}[tb]
   \caption{Single-loop Robust Policy Gradient Method}
   \label{alg:dsgda}
    \begin{algorithmic}
       \STATE {\bfseries Input:} $\pi_{0}=\bar{\pi}_{0}\in \Pi,p_{0}=\bar{p}_{0} \in \mcal P$,~stepsize~$\tau,\sigma>0$~and $0<\beta,\mu<1$
       \FOR{$k=0$ {\bfseries to} $K-1$}
       \STATE $\pi_{k+1}=\proj_{\Pi}\Brbra{\pi_{k}-\tau\nabla_{\pi}\chi(\pi_{k},p_{k},\bar{\pi}_{k},\bar{p}_{k})}$
       \STATE $p_{k+1}=\proj_{\mcal P}\Brbra{p_{k}+\sigma\nabla_{p}\chi(\pi_{k+1},p_{k},\bar{\pi}_{k},\bar{p}_{k})}$
       \STATE $\bar{\pi}_{k+1}=\bar{\pi}_{k}+\beta(\pi_{k+1}-\bar{\pi}_{k})$
       \STATE $\bar{p}_{k+1}=\bar{p}_{k}+\mu(p_{k+1}-\bar{p}_{k})$
       \ENDFOR
       \STATE {\bfseries Output:} return a policy from $\cbra{\bar{\pi}_k}_{k=1}^{K}$ uniformly
    \end{algorithmic}
\end{algorithm}

\subsection{\label{sec:general_con} Convergence analysis: an overview}

We provide a convergence analysis of {\srpg}, and a more precise analysis will be presented in \cref{sec:reg_exp}. For clarity, we begin with a proof sketch. Firstly, we inherit the Lyapunov function introduced in \citet{zheng2023universal} and derive the sufficient descent property between two consecutive iteration of {\srpg} in \cref{prop:prop1_un}. Then, we utilize the gradient dominance property of the dual function in \cref{prop:upperbound_neg} to give an upper bound of the negative term presented in the descent property. Finally, we present the convergence rate of SRPG in \cref{thm:stationary_result}, which further implies the global convergence result of SRPG due to the gradient dominance property with respect to $\pi$.

Before delving into the analysis, we define the stationary measure discussed in this paper as follows.
\begin{definition}
    We say point $(\pi, p) \in \Pi \times \mathcal{P}$ is a $\vep$-game stationary point if $\dist(\zerobf,\nabla_{\pi} J_{\rho}(\pi, p)+\mcal N_{\Pi}(\pi))\leq \vep$ and $\dist(\zerobf,-\nabla_p J_{\rho}(\pi, p)+\mcal N_{\mcal P}(p))\leq \vep$. Moreover, we say point $\pi$ is a $\vep$-optimization stationary point if $\norm{\prox_{\phi(\cdot), r_1} (\pi) - \pi }\leq \vep$ for $r_1>\ell_{\pi}$ is a constant.
\end{definition}
The definition above is adapted from the one in \citet{zheng2023universal}. The $\vep$-game stationary point is extended from the first-order stationary point in the minimization-only problem, and widely used in nonconvex-nonconcave optimization~\cite{diakonikolas2021efficient, lee2021fast}. The $\vep$-optimization stationary point is well-studied for weakly convex minimization problem \cite{davis2019stochastic}. Since we can regard our minimax problems as weakly convex minimization problem over $\pi$, we also consider this stationary measure in the following convergence analysis.

Similar to \citet{zheng2023universal}, we consider the Lyapunov function as follows:
\begin{align*}
&\Phi(\pi,p,\bar{\pi},\bar{p})\\
& := \chi(\pi,p,\bar{\pi},\bar{p})-\varphi_{\pi}(p,\bar{\pi},\bar{p}) + \varphi_{\pi,p}(\bar{\pi},\bar{p})-\varphi_{\pi}(p,\bar{\pi},\bar{p})\\
 & \quad +\varphi_{\pi,p,\bar{p}}(\bar{\pi})-\varphi_{\pi,p}(\bar{\pi},\bar{p})+\varphi_{\pi,p,\bar{p}}(\bar{\pi})-\underline{\chi}+\underline{\chi}.\nonumber 
\end{align*}

To interpret this Lyapunov function, the primal update and dual update correspond to the primal descent term $\chi(\pi,p,\bar{\pi},\bar{p})-\varphi_{\pi}(p,\bar{\pi},\bar{p})$ and dual ascent term $\varphi_{\pi,p}(\bar{\pi},\bar{p})-\varphi_{\pi}(p,\bar{\pi},\bar{p})$. The averaging updates of auxiliary variables $\bar{\pi}$ and $\bar{p}$ can be considered as an approximate gradient descent on $\varphi_{\pi,p,\bar{p}}(\bar{\pi})$ and an approximate gradient ascent on  $\varphi_{\pi,p}(\bar{\pi},\bar{p})$. With this function, we establish the following descent property of {\srpg}.
 
\begin{proposition}
[Informal]\label{prop:prop1_un}
Assume that parameters $\tau$, $\sigma$, $\beta$ and $\mu$ are chosen appropriately, and let $\spadesuit:= \|\pi(\bar{\pi}_{k+1},\bar{p}(\bar{\pi}_{k+1}))-\pi(\bar{\pi}_{k+1},\bar{p}_{k}^{+}(\bar{\pi}_{k+1}))\|^{2}$. 
Then, for any $k>0$, we have 
\begin{equation*}
\begin{split}
& \Phi_{k}-\Phi_{k+1} \geq\frac{r_{1}}{32}\|\pi_{k+1}-\pi_{k}\|^{2}+\frac{r_{2}}{15}\|p_{k}-p_{k}^{+}(\bar{\pi}_{k},\bar{p}_{k})\|^{2} \\
&\, +\frac{r_{1}}{5\beta}\|\bar{\pi}_{k}-\bar{\pi}_{k+1}\|^{2} +\frac{r_{2}}{4\mu}\|\bar{p}_{k}^{+}(\bar{\pi}_{k+1})-\bar{p}_{k}\|^{2}   - 4 r_1 \beta \cdot {\spadesuit}.
\end{split}
\end{equation*}
\end{proposition}

Although \cref{prop:prop1_un} explicitly characterizes the difference of the Lyapunov function between two consecutive iterates, the presence of the negative term of $\spadesuit$ still prevents the subsequent analysis of the decreasing nature of $\Phi$. In \citet{zheng2023universal}, they further bound the negative term of $\spadesuit$ by the assumption of the global K\L{} property or the concavity of the dual function, however, which is absent in our RMDPs setting. To overcome this challenge, we provide an alternative analysis starting from the gradient dominance property of the dual function, which complements the results in \citet{zheng2023universal}, and may be of independent interest. We give the following lemma for preparation.
\begin{lemma}\label{lem:p_grad_moreau}
    The Moreau envelope function $\varphi_{p}(\pi,\bar{\pi},\bar{p}):= \max_{p\in \mcal P}\chi(\pi,\bar{\pi},p,\bar{p})$ of $J_\rho(\pi, p)$ satisfies the gradient dominance property with respect to $\bar{p}$, i.e., there exists a constant $C_{\varphi p}>0$ such that $\max_{\bar{p}} \varphi_{p}(\pi,\bar{\pi},\bar{p})-\varphi_{p}(\pi,\bar{\pi},\bar{p})\leq C_{\varphi p}\norm{\nabla_{\bar{p}} \varphi_{p}(\pi, \bar{\pi}, \bar{p})}$ for $\forall \pi \in \Pi, \forall p \in \mcal P, \forall \bar{\pi}$.
\end{lemma}
\cref{lem:p_grad_moreau} follows from \cref{thm:moreau_grad_dom} and $J_{\rho}(\pi, p)$ is gradient-dominated in $p$ (see \cref{eq:partial_grad_dom_p}). Armed with~\cref{lem:p_grad_moreau}, we give the upper bound of term $\spadesuit$ in the next proposition.

\begin{proposition}[Informal]\label{prop:upperbound_neg}
With the above definitions, there exists a positive constant $\omega$ such that $ \spadesuit\leq \omega \norm{\bar{p}_+^k(\bar{\pi}_{k+1})-\bar{p}_k}$.
\end{proposition}

Equipped with \cref{prop:prop1_un} and \cref{prop:upperbound_neg}, we further establish the convergence result of SRPG.

\begin{theorem}\label{thm:stationary_result}
    Suppose that $\beta = \mcal O(K^{-1/2})$. Then for any $K>0$, we have: 1. There exists a positive integer $k_1 < K$ such that $(\pi_{k_1+1},p_{k_1+1})$ is an $\mcal O\brbra{(D_{\Pi}^{1/2} + D_{\mcal P}^{1/2})/K^{1/4}}$-game stationary point for problem \cref{eq:J_obj}; 2. There exists another positive integer $k_2<K$ such that $\bar{\pi}_{k_2+1}$ is an $\mcal O\brbra{(D_{\Pi}^{1/2} + D_{\mcal P}^{1/2})/K^{1/4}}$-optimization stationary point for problem \cref{eq:J_obj}. Here $D_{\Pi}:=\max_{\pi_1 \in \Pi, \pi_2\in \Pi}\norm{\pi_1 - \pi_2}$ and $D_{\mcal P}:=\max_{p_1 \in \mcal P, p_2\in \mcal P}\norm{p_1 - p_2}$ are the diameter of set $\Pi$ and $\mcal P$, respectively. 
\end{theorem}

\begin{remark}
While we mainly focus on the RMDP setting, \cref{thm:stationary_result} holds for the general nonconvex-nonconcave min-max optimization problem $\min_{x\in \mcal X}\max_{y\in \mcal Y} f(x,y)$ as long as the dual function $\max_{y\in \mcal Y} f(x,y)$ satisfies the gradient dominance property, that is, there exists a positive constant $C$ such that $\max_{y'\in \mcal Y}f(x,y')-f(x,y)\leq C\dist(0,-\nabla_y f(x,y)+\mcal N_{\mcal Y}(y))$ for all $y \in \mathcal{Y}$. 
This marks a novel finding for min-max optimization. 
\end{remark}

\begin{remark}
    In general, the $\vep$-game stationary point is not the same as the $\vep$-optimization stationary point. If $(\pi,p)$ is a $\vep$-game stationary point, then it is an $\mcal O(\vep^{1/2})$-optimal stationary point \cite{zheng2023universal}. However, \cref{thm:stationary_result} suggests that, when the last iterate $\pi_{k+1}$ is only an $\mcal O(\vep^{1/2})$-optimization stationary point, the exponential average point $\bar{\pi}_{k+1}$ is an $\vep$-optimization stationary point.
\end{remark}

We have established so far the convergence result of SRPG. In the rest of the subsection, we strengthen the result and further derive the global convergence result of SRPG. 
Before proceeding, we provide the definitions of $\vep$-global optimal solution and $(\vep, \delta)$-global optimal solution as follows.

\begin{definition}
    A solution $\pi$ is an $\vep$-global optimal solution for problem \cref{eq:J_obj} if $\phi(\pi)-\phi(\pi^*)\leq \vep$. Moreover, a solution $\pi$ is an $(\vep, \delta)$-global optimal solution if $\dist(\pi,\Pi^*_{\ell_{\pi}}(\vep))\leq \delta$, where $\Pi_{\ell_{\pi}}^{*}(\vep):=\{\pi \mid\phi_{2\ell_{\pi}}(\pi)-\phi(\pi^*)<\vep\}$.
\end{definition}
\begin{remark}
    The $(\vep, \delta)$-global optimal solution is crucial for the analysis in \cref{sec:reg_exp}, and we define it here for compactness. We emphasize that an $(\vep, 0 )$-global optimal solution is not an $\vep$-global optimal solution, since there exists a gap between $\phi_{2\ell_{\pi}}(\cdot)$ and $\phi(\cdot)$.
\end{remark}
With this preparation, we state our main global convergence result of SRPG in the next theorem.
\begin{theorem}\label{thm:global_opt}
Under the same setting of \cref{thm:stationary_result}, for the sequence $\cbra{\bar{\pi}_k}_{k=1}^{K}$ generated by {\srpg}, there exists a positive integer $k<K$ such that $\bar{\pi}_{k+1}$ is a $\mcal O\brbra{(D_{\Pi}^{1/2} + D_{\mcal P}^{1/2})/K^{1/4}}$-global optimal solution.
\end{theorem}

\subsection{\label{sec:reg_exp}Convergence analysis: a deeper dive}
We provide a more precise analysis of the convergence rate of {\srpg}. More specifically, we show that, when the algorithm proceeds and the sequence $\{\bar{\pi}_k\}$ approaches to the optimal solution, the convergence rate of SRPG gradually slows down and degenerates to $\mcal{O}(1/\vep^{4})$ stated in \cref{thm:global_opt}. Although this observation does not enhance the iteration complexity with respect to $\vep$, it offers a noteworthy explanation for the phenomena observed later in our numerical experiments in \cref{sec:num}.

To begin the analysis, we introduce the regional exponential growth property as follows.
\begin{lemma}[Regional Exponential Growth]\label{lem:regional_exp_grow}
Under the same setting of \cref{thm:Moreau_smooth_grad_dom}, for $\ell_{\pi}$-weakly convex function $\phi(\pi)$, we have
\begin{equation*}
\phi(\pi)-\phi^{*}\geq\zeta\exp\brbra{\dist(\pi,\Pi_{\ell_{\pi}}^{*}(\zeta))/C_{\ell_{\pi}}}, \forall \pi\in\Pi\setminus\Pi_{\ell_{\pi}}^{*}(\zeta), 
\end{equation*}
where $\zeta>0$ is any given constant.
\end{lemma}

\cref{lem:regional_exp_grow} suggests that within region $\Pi\setminus \Pi_{\ell_{\pi}}^{*}(\zeta)$, the function value of $\phi(\cdot)$ exhibits the exponential growth property, in contrast to the general quadratic growth property guaranteed by the P\L{} condition \cite{necoara2019linear}. We highlight that this region is not the optimal solution set to $\min_{\pi \in \Pi} \phi(\pi)$, but is instead associated with the Moreau envelope function $\phi_{2\ell_{\pi}}(\cdot)$ and a pre-determined constant $\zeta$. In fact, we can prove that $ \Pi^*(\zeta) \subseteq \Pi_{\ell_{\pi}}^{*}(\zeta)  $, where $\Pi^*(\zeta):=\cbra{\pi\mid \phi(\pi)-\phi^* \leq \zeta}$. After taking $\zeta = \vep$, we obtain the following proposition:
\begin{proposition}
\label{prop:region_growth}
    Under the same setting of \cref{thm:Moreau_smooth_grad_dom}, for the sequence $\cbra{\bar{\pi}_k}$ generated by {\srpg} and $\bar{\pi}_k\in \Pi\setminus \Pi_{\ell_{\pi}}^{*}(\vep)$, we have $  \dist(\bar{\pi}_k, \Pi^*_{\ell_{\pi}}(\vep)) = \mcal O\brbra{ \log\brbra{(D_{\Pi}^{1/2} + D_{\mcal P}^{1/2}) k^{-0.25}\vep^{-1}}}$. 
\end{proposition}

\begin{remark}
\cref{prop:region_growth} implies the following interesting observation that the iteration complexity to obtain any $(\vep, \delta)$-optimal solution can be bounded by $\mcal O\brbra{(D_{\Pi}^{2} + D_{\mcal P}^{2})\vep^{-4}\exp(-4\delta)}$. When $\delta$ approaches $0$, the rate degenerates to $\mcal O(1/\vep^{4})$, which is the same as we proved in \cref{thm:global_opt}. For a comprehensive illustration, one may refer to \cref{fig:convergence_rate_change}. The iteration sequence within the white area indicates that $\delta > 0$. A larger $\delta$ corresponds to a faster convergence performance. Upon entering the small blue region, \cref{lem:regional_exp_grow} and \cref{prop:region_growth} do not hold anymore, leading to a worse convergence rate $\mcal O(1/\vep^{4})$. Ultimately, the sequence reaches the orange dashed region, signifying the attainment of an $\vep$-optimal solution.
\end{remark}
\begin{remark}
An average convergence rate for $(\vep, 0)$ global optimal solution can be estimated by 
\begin{equation*}
\frac{(D_{\Pi}^{2} + D_{\mcal P}^{2})}{\tilde{D}_\Pi\cdot\vep^{4}}\int_{0}^{\tilde{D}_\Pi}\exp(-4\delta)d\delta
\leq \frac{D_{\Pi}^{2} + D_{\mcal P}^{2}}{4\tilde{D}_\Pi\cdot \vep^{4}},
\end{equation*}
where $\tilde{D}_{\Pi}$ is the distance between initial point $\pi_0$ and set $\Pi_{\ell_{\pi}}^*(\vep)$. This finding indicates that an unfavorable initial point (where $\tilde{D}_{\Pi}$ is large) may not significantly impact the algorithm's performance. 
\end{remark}

\input{regional_growth}

%% file: regional_growth.tex
\begin{figure}[h]
\centering

\tikzset{every picture/.style={line width=0.75pt}} 

\begin{tikzpicture}[x=0.75pt,y=0.75pt,yscale=-0.6,xscale=0.65]

\draw  [fill={rgb, 255:red, 74; green, 144; blue, 226 }  ,fill opacity=1 ] (402.71,155.64) .. controls (441.43,177.29) and (528.43,116.29) .. (549.43,150.29) .. controls (570.43,184.29) and (554.71,244.64) .. (506.71,252.64) .. controls (458.71,260.64) and (377.71,250.64) .. (400.71,234.64) .. controls (423.71,218.64) and (364,134) .. (402.71,155.64) -- cycle ;
\draw  [fill={rgb, 255:red, 74; green, 144; blue, 226 }  ,fill opacity=1 ] (166.43,121.29) .. controls (143.43,92.29) and (345,101) .. (322.71,127.64) .. controls (300.43,154.29) and (288.43,232.29) .. (294.71,249.64) .. controls (301,267) and (161.43,217.29) .. (157.43,229.29) .. controls (153.43,241.29) and (189.43,150.29) .. (166.43,121.29) -- cycle ;
\draw   (109,172) .. controls (109,101.31) and (215.2,44) .. (346.21,44) .. controls (477.22,44) and (583.43,101.31) .. (583.43,172) .. controls (583.43,242.69) and (477.22,300) .. (346.21,300) .. controls (215.2,300) and (109,242.69) .. (109,172) -- cycle ;
\draw    (514.43,83.29) .. controls (454.16,83.57) and (334.63,96.16) .. (287.14,173.12) ;
\draw [shift={(286.43,174.29)}, rotate = 301.07] [color={rgb, 255:red, 0; green, 0; blue, 0 }  ][line width=0.75]    (10.93,-3.29) .. controls (6.95,-1.4) and (3.31,-0.3) .. (0,0) .. controls (3.31,0.3) and (6.95,1.4) .. (10.93,3.29)   ;
\draw    (514.43,83.29) .. controls (501.26,84.21) and (505.23,82.42) .. (492.11,84) ;
\draw [shift={(490.14,84.25)}, rotate = 352.65] [color={rgb, 255:red, 0; green, 0; blue, 0 }  ][line width=0.75]    (10.93,-3.29) .. controls (6.95,-1.4) and (3.31,-0.3) .. (0,0) .. controls (3.31,0.3) and (6.95,1.4) .. (10.93,3.29)   ;
\draw    (514.43,83.29) .. controls (503.46,84.05) and (492.98,83.58) .. (481.35,84.36) .. controls (478.99,84.52) and (476.59,84.73) .. (474.12,85.01) ;
\draw [shift={(472.14,85.25)}, rotate = 352.65] [color={rgb, 255:red, 0; green, 0; blue, 0 }  ][line width=0.75]    (10.93,-3.29) .. controls (6.95,-1.4) and (3.31,-0.3) .. (0,0) .. controls (3.31,0.3) and (6.95,1.4) .. (10.93,3.29)   ;
\draw    (514.43,83.29) .. controls (503.46,84.05) and (492.98,83.58) .. (481.35,84.36) .. controls (474.71,84.8) and (466.5,85.71) .. (460.1,86.49) .. controls (456.8,86.89) and (453.99,87.25) .. (452.1,87.49) ;
\draw [shift={(450.14,87.75)}, rotate = 352.65] [color={rgb, 255:red, 0; green, 0; blue, 0 }  ][line width=0.75]    (10.93,-3.29) .. controls (6.95,-1.4) and (3.31,-0.3) .. (0,0) .. controls (3.31,0.3) and (6.95,1.4) .. (10.93,3.29)   ;
\draw    (514.43,83.29) .. controls (503.46,84.05) and (492.98,83.58) .. (481.35,84.36) .. controls (469.72,85.14) and (466.5,85.71) .. (460.1,86.49) .. controls (454.4,87.18) and (434.32,90.41) .. (427.51,91.82) ;
\draw [shift={(425.64,92.25)}, rotate = 344.05] [color={rgb, 255:red, 0; green, 0; blue, 0 }  ][line width=0.75]    (10.93,-3.29) .. controls (6.95,-1.4) and (3.31,-0.3) .. (0,0) .. controls (3.31,0.3) and (6.95,1.4) .. (10.93,3.29)   ;
\draw    (514.43,83.29) .. controls (503.46,84.05) and (492.98,83.58) .. (481.35,84.36) .. controls (469.72,85.14) and (466.5,85.71) .. (460.1,86.49) .. controls (453.69,87.26) and (444.07,88.86) .. (435.34,90.37) .. controls (427.7,91.7) and (401.29,98.43) .. (393.93,100.38) ;
\draw [shift={(392.06,100.89)}, rotate = 342.04] [color={rgb, 255:red, 0; green, 0; blue, 0 }  ][line width=0.75]    (10.93,-3.29) .. controls (6.95,-1.4) and (3.31,-0.3) .. (0,0) .. controls (3.31,0.3) and (6.95,1.4) .. (10.93,3.29)   ;
\draw    (514.43,83.29) .. controls (503.46,84.05) and (492.98,83.58) .. (481.35,84.36) .. controls (469.72,85.14) and (466.5,85.71) .. (460.1,86.49) .. controls (453.69,87.26) and (444.07,88.86) .. (435.34,90.37) .. controls (426.6,91.89) and (415.13,96.08) .. (402.72,98.11) .. controls (391.79,99.89) and (361.64,112.2) .. (354.04,115.39) ;
\draw [shift={(352.3,116.13)}, rotate = 333.43] [color={rgb, 255:red, 0; green, 0; blue, 0 }  ][line width=0.75]    (10.93,-3.29) .. controls (6.95,-1.4) and (3.31,-0.3) .. (0,0) .. controls (3.31,0.3) and (6.95,1.4) .. (10.93,3.29)   ;
\draw    (514.43,83.29) .. controls (503.46,84.05) and (492.98,83.58) .. (481.35,84.36) .. controls (469.72,85.14) and (466.5,85.71) .. (460.1,86.49) .. controls (453.69,87.26) and (444.07,88.86) .. (435.34,90.37) .. controls (426.6,91.89) and (415.13,96.08) .. (402.72,98.11) .. controls (390.3,100.13) and (374.62,107.05) .. (362.47,111.93) .. controls (350.88,116.6) and (328.64,129.94) .. (317.92,138.02) ;
\draw [shift={(316.49,139.13)}, rotate = 321.73] [color={rgb, 255:red, 0; green, 0; blue, 0 }  ][line width=0.75]    (10.93,-3.29) .. controls (6.95,-1.4) and (3.31,-0.3) .. (0,0) .. controls (3.31,0.3) and (6.95,1.4) .. (10.93,3.29)   ;
\draw  [fill={rgb, 255:red, 245; green, 166; blue, 35 }  ,fill opacity=1 ] (235.86,148.64) -- (254.04,172.76) -- (294.69,176.63) -- (265.27,195.4) -- (272.22,221.91) -- (235.86,209.39) -- (199.5,221.91) -- (206.44,195.4) -- (177.03,176.63) -- (217.68,172.76) -- cycle ;
\draw  [fill={rgb, 255:red, 245; green, 166; blue, 35 }  ,fill opacity=1 ] (505.71,168.64) -- (516.88,189.78) -- (541.85,193.17) -- (523.78,209.63) -- (528.05,232.86) -- (505.71,221.89) -- (483.38,232.86) -- (487.64,209.63) -- (469.57,193.17) -- (494.55,189.78) -- cycle ;

\draw (415,182) node [anchor=north west][inner sep=0.75pt]  [font=\small]  {$\Pi _{\ell _{\pi }}^{*}( \varepsilon )$};
\draw (484.55,186.78) node [anchor=north west][inner sep=0.75pt]    {$\Pi ^{*}( \varepsilon )$};
\draw (341,73) node [anchor=north west][inner sep=0.75pt]    {$\Pi $};
\draw (228,114) node [anchor=north west][inner sep=0.75pt]    {$\Pi _{\ell _{\pi }}^{*}( \varepsilon )$};
\draw (211,179) node [anchor=north west][inner sep=0.75pt]    {$\Pi ^{*}( \varepsilon )$};

\end{tikzpicture}
\caption{\label{fig:convergence_rate_change}Illustration of the changing convergence rate: a denser arrow implies a faster convergence rate.}
\end{figure}
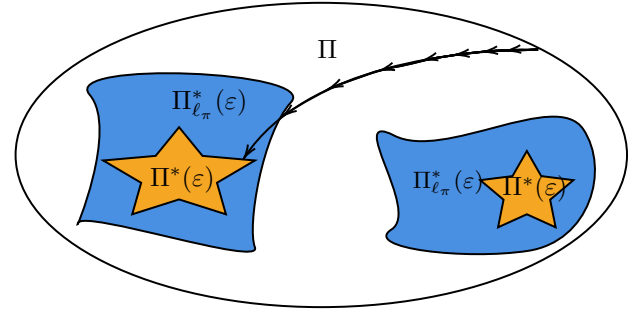

%% file: icml_numerical.tex
\ifthenelse{\boolean{ICML}}{
\begin{figure*}[htb]
\begin{centering}
\begin{minipage}[t]{0.4\columnwidth}%
\includegraphics[width=4.3cm]{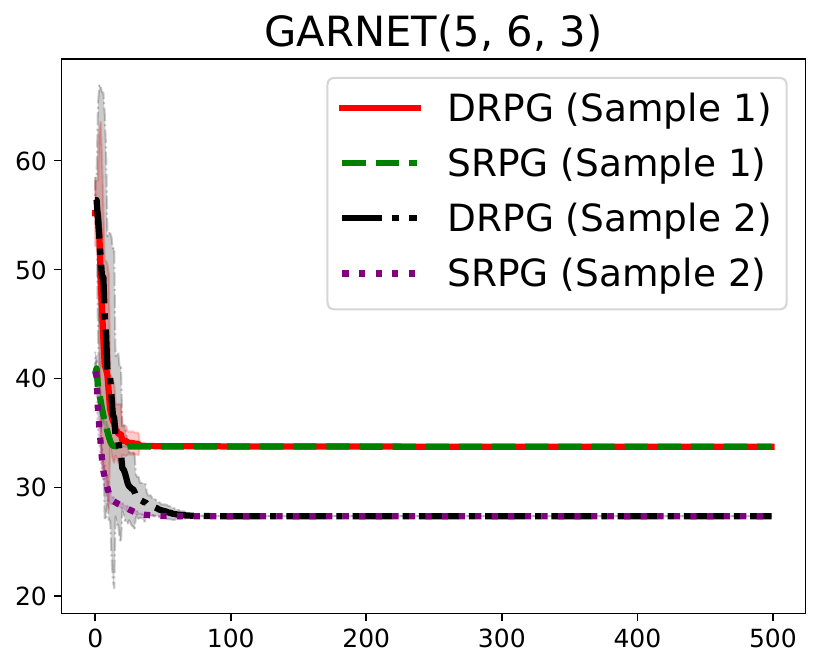}%
\end{minipage}\hfill
\begin{minipage}[t]{0.4\columnwidth}%
\includegraphics[width=4.3cm]{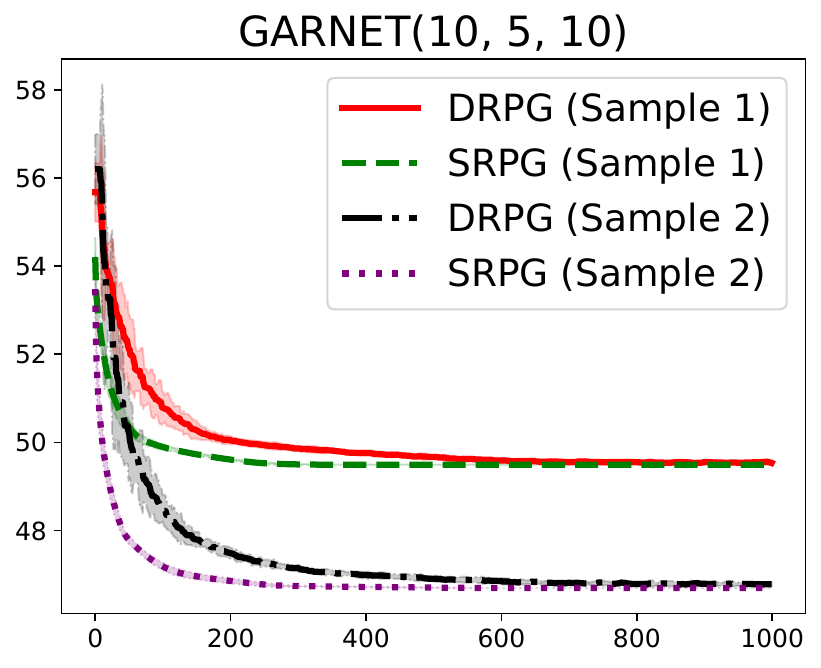}%
\end{minipage}\hfill
\begin{minipage}[t]{0.4\columnwidth}%
\includegraphics[width=4.3cm]{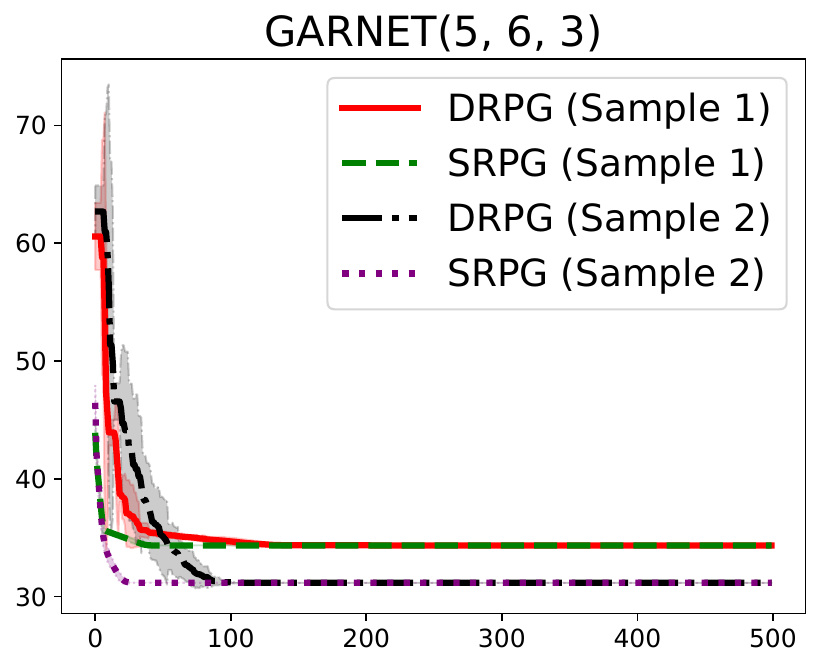}%
\end{minipage}\hfill
\begin{minipage}[t]{0.4\columnwidth}%
\includegraphics[width=4.3cm]{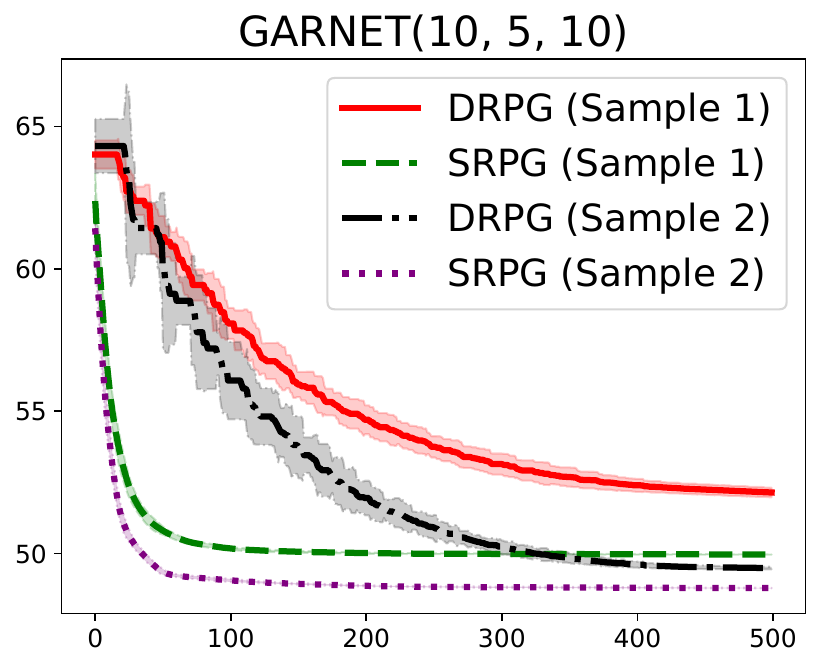}%
\end{minipage}\hfill
\end{centering}
\caption{\label{fig:garnet}The global convergence behavior of SRPG and DRPG on the {\garnet} MDPs. The solid curves depict the average value of $\phi(\pi_k)$, represented by $y$-axis, while the shaded areas correspond to its $95\%$ confidence interval. The $x$-axis represents the total number of updates on $\pi$ and $p$. The left two figures are the results for $s$-rectangular ambiguity set, and the right two figures are the results for $(s, a)$-rectangular ambiguity set. We remark that each sample in each figure corresponds to a different task.}
\end{figure*}
}

\ifthenelse{\boolean{Arxiv}}{
\begin{figure*}[htb]
\begin{centering}
\begin{minipage}[t]{0.2\columnwidth}%
\includegraphics[width=4.3cm]{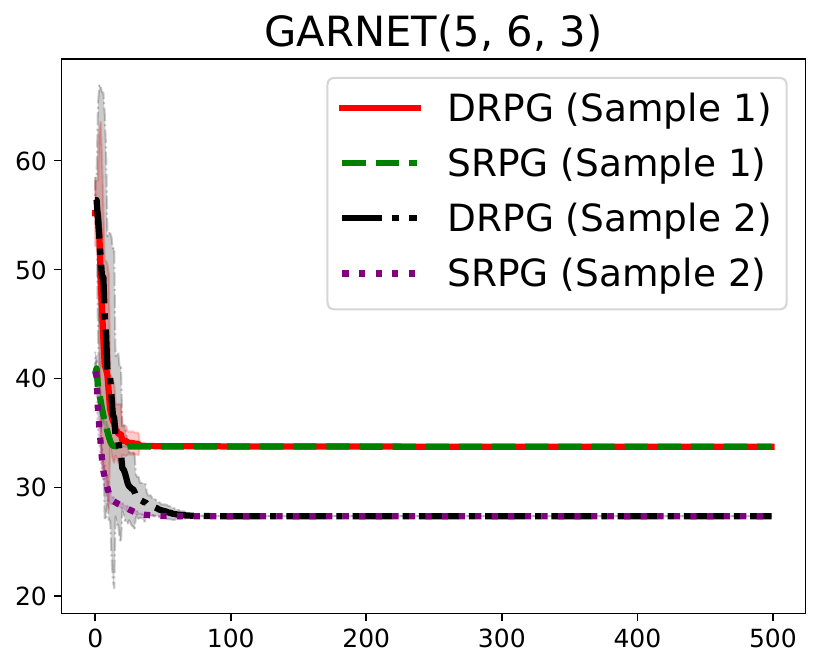}%
\end{minipage}\hfill
\begin{minipage}[t]{0.2\columnwidth}%
\includegraphics[width=4.3cm]{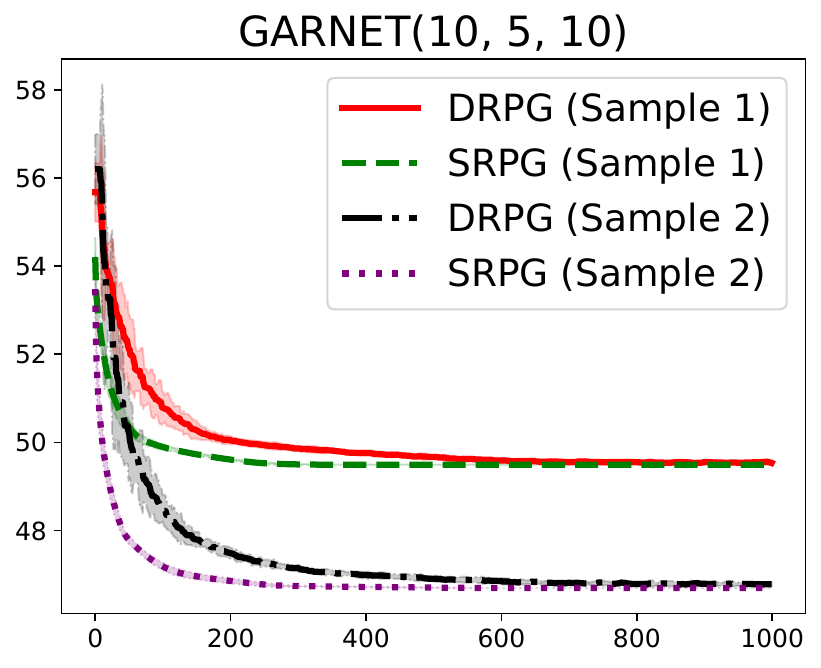}%
\end{minipage}\hfill
\begin{minipage}[t]{0.2\columnwidth}%
\includegraphics[width=4.3cm]{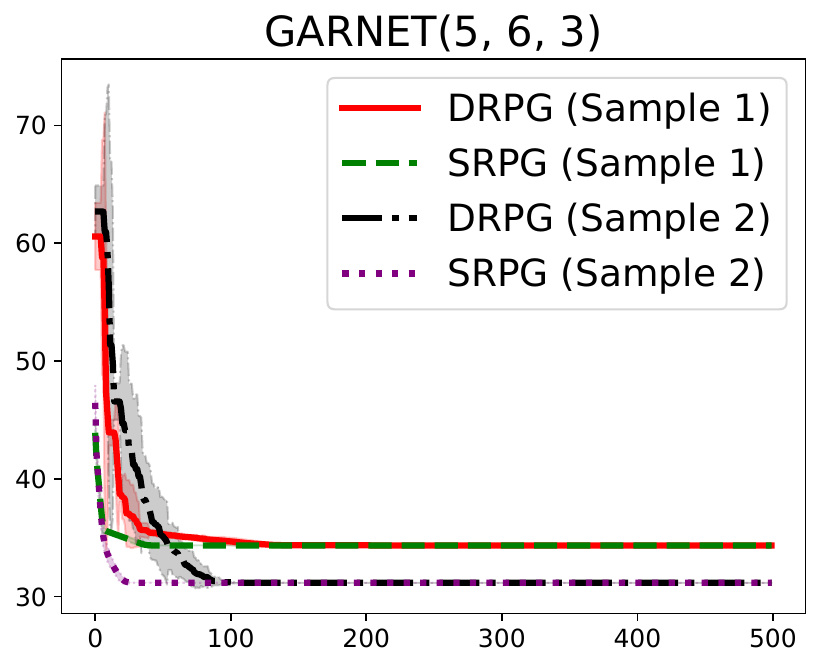}%
\end{minipage}\hfill
\begin{minipage}[t]{0.2\columnwidth}%
\includegraphics[width=4.3cm]{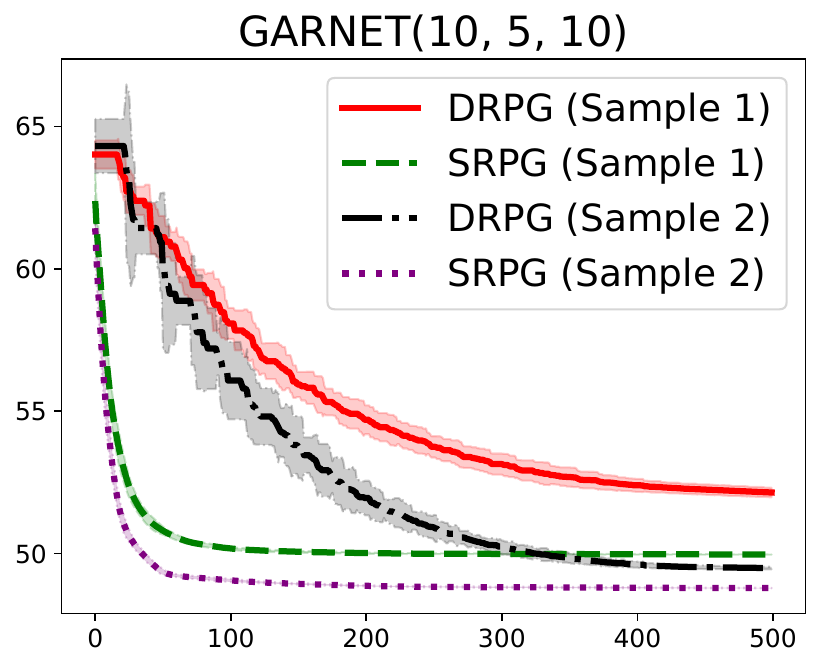}%
\end{minipage}\hfill
\end{centering}
\caption{\label{fig:garnet}The global convergence behavior of SRPG and DRPG on the {\garnet} MDPs. The solid curves depict the average value of $\phi(\pi_k)$, represented by $y$-axis, while the shaded areas correspond to its $95\%$ confidence interval. The $x$-axis represents the total number of updates on $\pi$ and $p$. The left two figures are the results for $s$-rectangular ambiguity set, and the right two figures are the results for $(s, a)$-rectangular ambiguity set. We remark that each sample in each figure corresponds to a different task.}
\end{figure*}
}

\section{\label{sec:num}Numerical Experiments}

We conduct several experiments to investigate the performance of  SRPG compared with DRPG~\citep{wang2023policy}. In particular, we consider two different problems, including GARNET MDPs and an inventory management problem. To measure the performance of SRPG and DRPG, we use the worst-case expected return for a given $\pi_k$, i.e., $\phi(\pi_k)$. 

We apply the projected gradient method (PGM) to solve the inner maximization problem in DRPG. Following the setup of \citet{wang2023policy}, we terminate PGM when the relative error between two iterations is smaller than $10^{-4}$ or the iteration number reaches $200$. Note that PGM is also used to evaluate $\phi(\pi_k)$ for SRPG and each projection operator needed in the two algorithms. To solve the quadratic subproblems in PGM, we resort to the state-of-the-art commercial solver GUROBI \cite{gurobi}. We provide the code in \href{https://github.com/zhenweilin/srpg}{this link}.

\subsection{{\garnet} MDPs}

In the first experiment, we consider the {\garnet} MDPs, which are introduced by \citet{archibald1995generation} and widely recognized as a key benchmark in RMDPs \cite{wang2023policy, li2024policy}. 
Specifically, a standard {\garnet} MDP, denoted by {\garnet}$(S, A, b)$, consists of three different parameters. 
Here, $S$ and $A$ indicate the finite numbers of states and actions, respectively.
The branching factor $b$ determines the number of states that are reachable from any given state-action pair in one transition.

\textbf{Problem setup.} We randomly generate the nominal transition kernel $\bar{p}$ according to two different GARNET MDPs: {\garnet}$(5,6,3)$ and {\garnet}$(10,5,10)$. We let the discount factor $\gamma=0.95$, and sample the cost $c_{sas'}$ i.i.d. from the uniform distribution supported on $[0,5]$. Each element in the initial state distribution $\rho \in \Delta^{S}$ is sampled from $[0, 5]$ uniformly, then we project $\rho$ into the probability space.
For each $\bar{p}$, we construct both $L_1$-normed $s$- and $(s,a)$-rectangular ambiguity sets. The $s$-rectangular ambiguity set is defined by $\mcal P=\times_{s\in [S]}\mcal P_s$, where $\mcal P_s$ is
\begin{equation*}
    \mcal P_s:=\Bcbra{(p_{s1},\ldots,p_{sA} )\in (\Delta^S)^A\mid \sum_{a\in \mcal A}\norm{p_{sa} - \bar{p}_{sa}}_1 \leq \kappa_{s}}.
\end{equation*}
Similarly, the $(s,a)$-rectangular ambiguity set is formulated as $\mcal P=\times_{s\in [S],a\in [A]}\mcal P_{sa}$, with each $\mcal P_{sa}$ being 
\[\mcal P_{sa}:=\bcbra{p_{sa}\in \Delta^A \mid \norm{p_{sa} - \bar{p}_{sa}}_1 \leq \kappa_{sa}}.
\]
Here, $\kappa_s$ and $\kappa_{sa}$  control the uncertainty included in the ambiguity sets, and are sampled uniformly from $[0.1, 0.5]$. For each generated $\bar{p}$ and each ambiguity set, we conduct two independent experiments, which brings $8$ different tasks.

\textbf{Experiment configuration.} Throughout all the tasks, we tune SRPG as follows. We choose the primal stepsize $\tau$ and dual stepsize $\sigma$ from $\bcbra{0.01, 0.05, 0.1}$. We also choose the extrapolation parameters $\beta$ and $\mu$ from $\bcbra{0.01, 0.05, 0.1, 0.2, 0.4}$ for ${\srpg}$. For DRPG, we also tune its primal and dual stepsize from $\bcbra{0.01, 0.05, 0.1}$. To facilitate a fair comparison between the two algorithms, we run each algorithm with an identical total number of updates on $\pi$ and $p$, and choose the number from $\{500, 1000\}$. For each task, we conduct $10$ random experiments with different initial points, and record the average values of $\phi(\pi_k)$.

\textbf{Comparison.} \cref{fig:garnet} shows the global convergence behavior of two algorithms by plotting the average value of the sequence $\{\phi(\pi_k)\}$ against the total number of updates on $\pi$ and $p$. It demonstrates that our {\srpg} converges consistently and significantly faster than DRPG over all $8$ tasks, which indicates the superior performance and robustness of \srpg. Generally speaking, both algorithms exhibit a significant decrease in the value of $\phi(\pi_k)$ within the initial several iterations. However, as the algorithms proceed, the convergence rate slows down notably. This validates the regional exponential growth property discussed in \cref{sec:reg_exp}, which states that when the iterate is near the optimal solution, the convergence rate will gradually become $ \mathcal{O}(1/\vep^{4})$. Furthermore, DRPG displays an oscillation pattern, especially in the last figure. The underlying reason of this phenomenon could be attributed to the policy $\pi$ not being updated within the inner loop of DRPG, resulting in significant changes to the dual variable during the update of the primal variable $\pi$. This also highlights the advantage of our single-loop algorithm. The convergence behavior of SRPG is smoother, implying a steady and consistent improvement without significant fluctuations.

\subsection{Inventory management problem}

Our second experiment considers the inventory management problem, in which a retailer engages in the ordering, storage, and sale of a single product over an infinite time horizon~\citep{porteus2002foundations,ho2018fast}. It can be formulated as an RMDP as follows. 
The inventory levels and order quantities correspond to the states
and actions of the RMDP in any given time step, respectively.
The distribution of demand, which is unknown to the retailer, gives the transition kernel. Moreover, an item held in inventory will incur a deterministic per time step holding cost. We aim to find a policy that minimizes the worst-case total cost.

Because the numerical results in the last subsection have already demonstrated the superior performance of \srpg{} over DRPG in the tabular setting, we now go beyond the assumption of $s$-rectangular ambiguity set and test the performance of two algorithms in the parameterization setting. This is more suitable and practical for the large-scale RMDP. Specifically, we first test with the parameterization of transition kernel, which is also employed by \citet{wang2023policy}, and then test with parameterization of both policy and transition kernel. The detailed parameterization methods are provided in \cref{sec:appendix_num}.

\textbf{Experiment configuration.}  Similar to the experiments on {\garnet} MDPs, we tune SRPG by choosing the primal stepsize $\tau$ and dual stepsize $\sigma$ from $\bcbra{0.01, 0.05, 0.1}$ and selecting the extrapolation parameters $\beta$ and $\mu$ from $\bcbra{0.1, 0.2, 0.3}$. For DRPG, we also tune its primal and dual stepsize from $\bcbra{0.01, 0.05, 0.1}$. When testing with transition kernel parameterization, we run the algorithms with an identical total number of updates on both primal and dual variables and conduct ten random experiments with different initial points. We use the same performance measure as in the last subsection. When considering parameterization on both policy and transition kernel, we run DRPG on each instance with 15,000 total updates, recording the time and the obtained objective value. Then, we run \srpg{} and record the time required to obtain the same objective value as DRPG.
\begin{figure}[h]
    \centering
    \includegraphics[width = 0.45\textwidth]{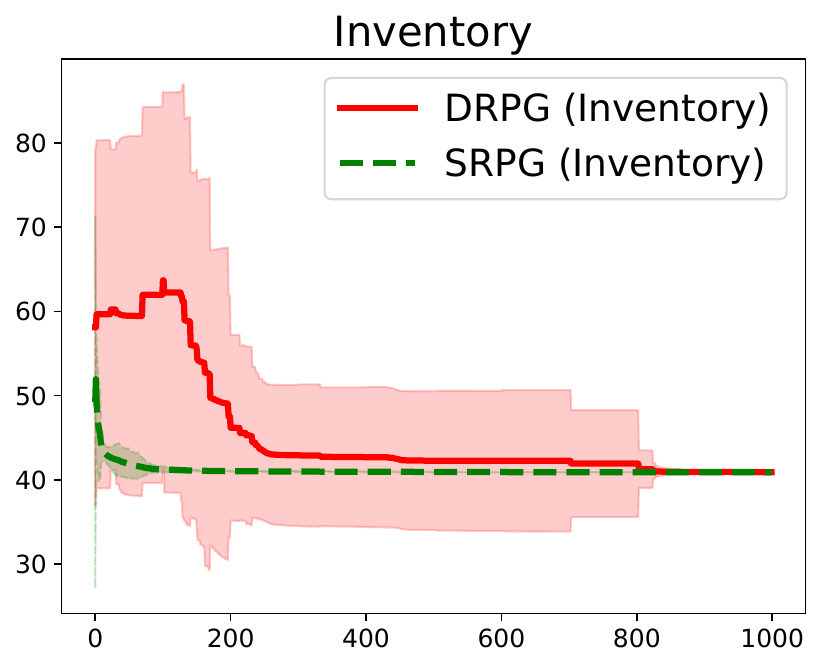}
    \caption{The global convergence behavior SRPG and DRPG on the inventory management problem. The figure is plotted in the same way as \cref{fig:garnet}.}
    \label{fig:inventory}
\end{figure}

\textbf{Discussion on computational cost for projection operator.} Let the computational cost for computing the projection operator of primal variable $\pi$ be $M_{\pi}$, and the one of dual variable $p$ be $M_p$. Then, for SRPG, the total computational cost for the projection operator is $\mathcal{O}(\epsilon^{-4}(M_\pi+M_p))$. For DRPG, it is a double-loop algorithm and needs to compute the projection operator $\mathcal{O}(\epsilon^{-2}_{k})$ for the $k$-th outer iteration. It at least needs the computational cost of $\mathcal{O}(\epsilon^{-4}(M_\pi+M_p\epsilon^{-2}))$ when we take $\epsilon_k = \epsilon$ for all $k$. Therefore, DRPG is much more computationally expensive than ours in terms of projection operator, as well as gradient computation.

\textbf{Comparison.} When only parameterizing the transition kernel, \cref{fig:inventory} shows the average values of sequence $\{\phi(\pi_k)\}$ for the two algorithms. It can be observed that {\srpg} demonstrates superior robustness compared to DRPG, and converges faster. Despite minor fluctuations in the initial phase, {\srpg} quickly stabilizes and converges. In contrast, DRPG oscillates and faces formidable challenges in converging to the optimal policy. When parameterizing both policy and transition kernel, \cref{tab:both-para} demonstrates that {\srpg} significantly decreases the time needed to find a solution when compared to DRPG, and thus more efficient. This is also consistent with the theoretical analysis. 

\begin{table}[h]
    \centering
    \begin{tabular}{ccrr}
    \hline
    $S$ & $A$ & DRPG  & \srpg \\
    \hline
    200 & 20 &	9,377.126	&\textbf{1,517.999} \\
    300	& 20 &12,758.126&	\textbf{4,719.497} \\
    500	& 20&	35,108.616	&\textbf{3,737.307} \\
    200	& 30&	7,605.747&	\textbf{1,764.330} \\
    300	& 30&	16,725.184&	\textbf{5,852.938} \\
    200	& 50	&12,910.183&	\textbf{5,351.032} \\
    300	& 50&	32,342.525&	\textbf{2,114.342} \\
    \hline
    \end{tabular}
    \caption{Comparison of DRPG and \srpg{} on both parameterization on policy and transition kernel. The smaller figure indicates less runtime and is bolded.}
    \label{tab:both-para}
\end{table}

%% file: icml_conclusion.tex
\section{Conclusion}
This paper introduces the first single-loop robust policy gradient method (SRPG) for RMDPs. We provide its convergence analysis, demonstrating its ability to converge to the globally optimal policy with complexity $\mcal O(1/\vep^{4})$. Additionally, we unveil the regional exponential growth property and leverage it for a more precise convergence analysis of SRPG. Our numerical experiments showcase that SRPG exhibits faster and more stable convergence behavior when compared to its double-loop counterpart. For future work, a promising direction is to incorporate the mirror descent method into our sing-loop framework and attempt to obtain a better convergence rate for solving RMDPs. It is also interesting to capture the noise gradient and function approximation in the analysis, which is more suitable for practical applications.

\section*{Acknowledgement}
The authors are grateful to the Area Chairs and the anonymous reviewers for their constructive comments. This research is partially supported by the Major Program of National Natural Science Foundation of China (Grant 72394360, 72394364).

\section*{Impact Statement}
This paper presents work whose goal is to advance the field of Machine learning. There are many potential societal consequences of our work, but none of which we feel must be specifically highlighted here. 

%% file: icml_appendix.tex
\ifthenelse{\boolean{ICML}}{
\doparttoc
\faketableofcontents
\part{}
\appendix
\onecolumn
\addcontentsline{toc}{section}{Appendix}
\part{Appendix} 
\parttoc
}

\ifthenelse{\boolean{Arxiv}}{
\clearpage
\part{Appendix} 
\insertseparator
\startcontents[appendix] 
\printcontents[appendix]{}{1}{}
\insertseparator
}

\section*{Structure of the Appendix}
The appendix is organized as follows. \cref{sec:app_notations} introduces some definitions and notations essential for understanding the proof process. It also includes more discussion about the difference between gradient dominance and K\L{} property. \cref{sec:auxiliary_lemmas} presents some crucial results that underpin our convergence analysis. Detailed proofs for convergence results are thoroughly outlined in~\cref{sec:appen_convergence_ana}, and we provide an illustration in \cref{fig:enter-label}. Furthermore, \cref{sec:appendix_num} offers more extensive details on our experiments.
\ifthenelse{\boolean{ICML}}{
\begin{figure}[H]
    \centering
    \includegraphics[width = 15.6cm, height = 10.3cm]{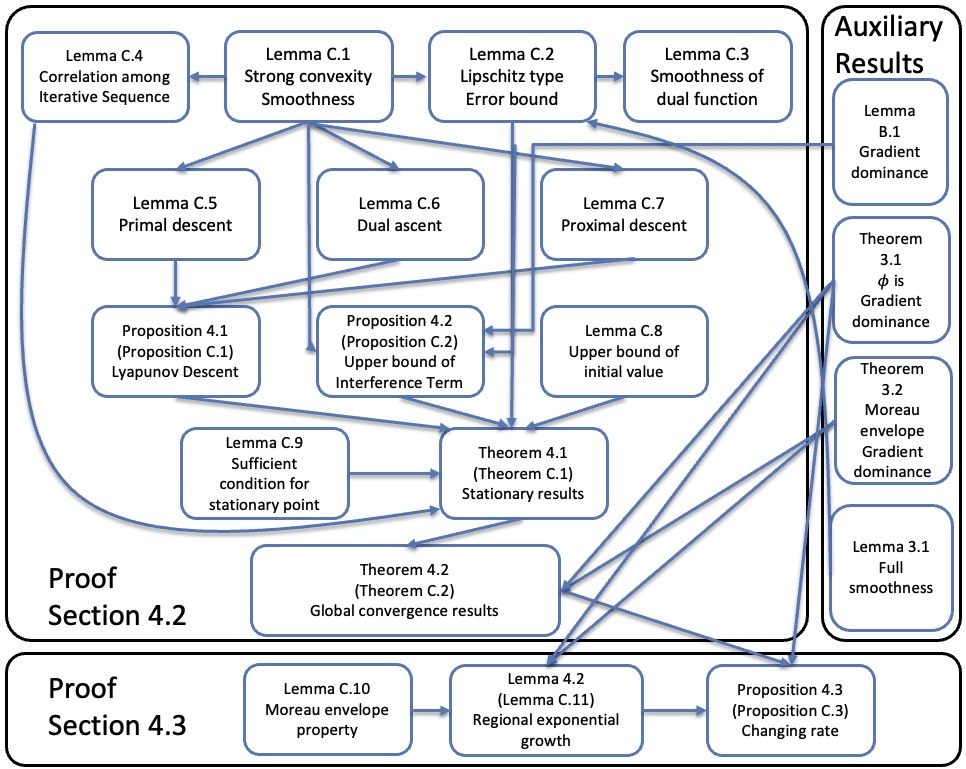}
    \caption{Flowchart of proof process.}
    \label{fig:enter-label}
\end{figure}
}

\ifthenelse{\boolean{Arxiv}}{
\begin{figure}[H]
    \centering
    \includegraphics[width = 15.6cm, height = 8.5cm]{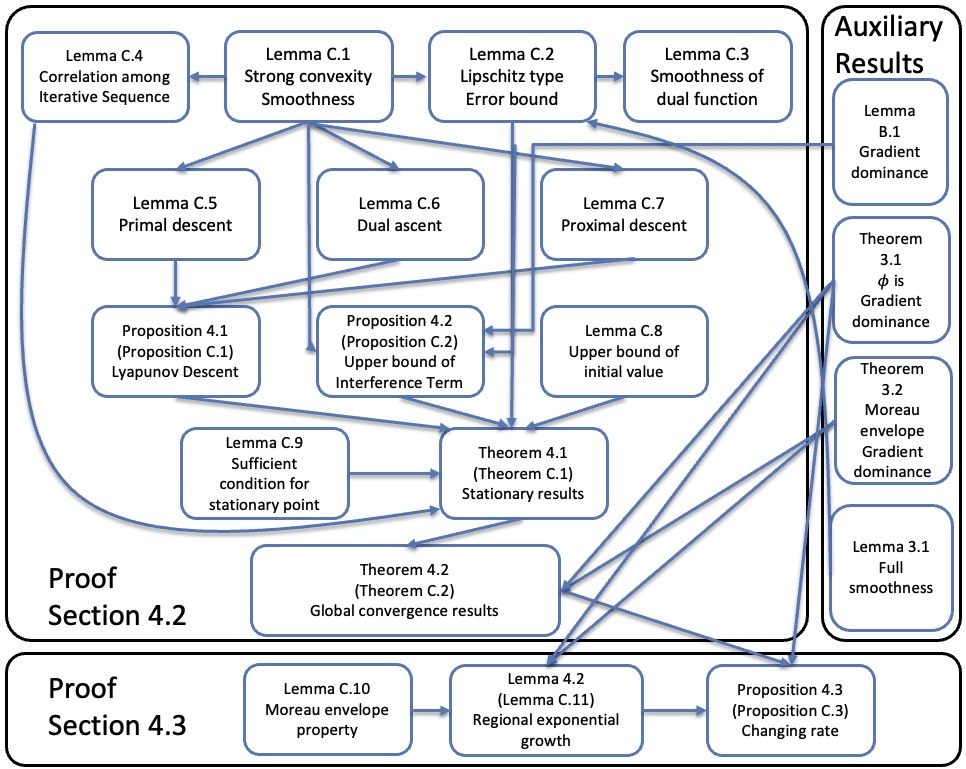}
    \caption{Flowchart of proof process.}
    \label{fig:enter-label}
\end{figure}
}

\clearpage
\section{\label{sec:app_notations} Preliminaries and Notations}
We first give some definitions used in this paper, and list some useful notations in~\cref{tab:Notations}.

\begin{definition}[$L$-Lipschitz Function]
    A function $f$ is called $L$-Lipschitz with respect to norm $\|\cdot \|$ if we have
    \begin{equation*}
        |f(x) - f(y)| \leq L \|x-y\| \quad \forall x, y \in \mathbb{R}^n.
    \end{equation*}
\end{definition}

\begin{definition}[$L$-smooth Function]
    A function $f$ is called $L$-smooth if it is continuously differentiable and its gradient is Lipschitz continuous with Lipschitz constant $L$, 
    \begin{equation*}
        \|\nabla f(x) - \nabla f(y) \|_2 \leq L\|x - y\|_2 \quad \forall x, y \in \mathbb{R}^n.
    \end{equation*}
\end{definition}

\begin{definition}[Weakly Convex Function]
A function $f$ is called weakly convex if there exists $\rho \geq 0$ such that $f(x)+\frac{\rho}{2}\|x\|^2$ is a convex function.
\end{definition}

\begin{definition}[Normal Cone]
Let $C \subset \mathbb{R}^n$ be a convex set with $\bar{x} \in C$. The normal cone to $C$ at $\bar{x}$ is
\begin{equation*}
\mathcal{N}_{C}(\bar{x}):=\left\{v \in \mathbb{R}^n \mid\langle v, x-\bar{x}\rangle \leq 0, \forall x \in C\right\}.
\end{equation*}
\end{definition}

\begin{definition}[Proximal Operator]
    Let $ r > 0$, then the proximal operator of a function $f$ is given by
    \begin{equation*}
        \text{prox}_{f, r}(x) = \arg\min_{y} \left\{ f(y) + \frac{r \|x-y\|^2}{2}\right\} .
    \end{equation*}
\end{definition}

\begin{definition}[Moreau envelope]
The Moreau envelope of a function $f$ is given by
\begin{equation*}
f_\mu(x)=\inf _y\left\{f(y)+\frac{\mu}{2}\|x-y\|_2^2\right\}.
\end{equation*}
\end{definition}

\begin{definition}[Bergman Divergence]
Let $\psi: \Omega \rightarrow \mathbb{R}$ be a strictly convex and continuously differentiable function defined on a closed convex set $\Omega$. Then the Bregman divergence is defined as
\begin{equation*}
\mathcal{D}(x, y)=\psi(x)-\psi(y)-\langle\nabla \psi(y), x-y\rangle, \quad \forall x, y \in \Omega .
\end{equation*}
\end{definition}

\begin{definition}[Bregman Envelope]
    The Bregman envelope of a proper lower semi-continuous convex function $f$ is defined as 
    $$f_{\mathcal{D},\lambda}(\bar{x}):=\inf_{x\in X}\{f(x)+\lambda \mathcal{D}(x,\bar{x})\}.$$
    where $\mathcal{D}(x,\bar{x})$ is the Bergman divergence.
\end{definition}

\begin{table}[htb]
\begin{centering}
\renewcommand{\arraystretch}{1.3}
\begin{tabular}{|l|l|l|}
\hline 
Notation & Meaning & Remark\tabularnewline
\hline 
\hline 
$\chi(\pi,p,\bar{\pi},\bar{p})$ & $J_{\rho}(\pi,p)+\frac{r_{1}}{2}\norm{\pi-\bar{\pi}}^{2}-\frac{r_{2}}{2}\norm{p-\bar{p}}^{2}$  &\ \ -\tabularnewline
\hline 
$\varphi_{\pi}(p,\bar{\pi},\bar{p})$ & $\min_{\pi\in\Pi}\chi(\pi,p,\bar{\pi},\bar{p})$ & $\pi(p,\bar{\pi},\bar{p})\in\argmin_{\pi\in\Pi}\chi(\pi,p,\bar{\pi},\bar{p})$\tabularnewline
\hline 
$\varphi_{p}(\pi,\bar{\pi},\bar{p})$ & $\max_{p\in \mcal P}\chi(\pi,p,\bar{\pi},\bar{p})$ & $p(\pi,\bar{\pi},\bar{p})\in\argmax{}_{p\in\mcal P}\chi(\pi,p,\bar{\pi},\bar{p})$\tabularnewline
\hline 
$\varphi_{\pi,p}(\bar{\pi},\bar{p})$ & $\min_{\pi\in\Pi}\max_{p\in\mcal P}\chi(\pi,p,\bar{\pi},\bar{p})$  &\ \ -\tabularnewline
\hline 
$\varphi_{\pi,p}(\bar{\pi},\bar{p})$ & $\min_{\pi\in\Pi}\varphi_{p}(\pi,\bar{\pi},\bar{p})$ & $\pi(\bar{\pi},\bar{p})=\pi\brbra{p(\bar{\pi},\bar{p}),\bar{\pi},\bar{p}}$\tabularnewline
\hline 
$\varphi_{\pi,p}(\bar{\pi},\bar{p})$ & $\max_{p\in\mcal P}\varphi_{\pi}(p,\bar{\pi},\bar{p})$ & $p(\bar{\pi},\bar{p})=p\brbra{\pi(\bar{\pi},\bar{p}),\bar{\pi},\bar{p}}$\tabularnewline
\hline 
$\varphi_{\pi,p,\bar{\pi}}(\bar{p})$ & $\min_{\bar{\pi}}\varphi_{\pi,p}(\bar{\pi},\bar{p})$ & $\bar{\pi}(\bar{p})\in\argmin_{\bar{\pi}\in\mbb R^{\abs{\mcal S}\times\abs{\mcal A}}}\varphi_{\pi,p}(\bar{\pi},\bar{p})$\tabularnewline
\hline 
$\varphi_{\pi,p,\bar{p}}(\bar{\pi})$ & $\max_{\bar{p}}\varphi_{\pi,p}(\bar{\pi},\bar{p})$ & $\bar{p}(\bar{\pi})\in\argmax_{\bar{p}}\varphi_{\pi,p}(\bar{\pi},\bar{p})$\tabularnewline
\hline 
$p^{+}(\bar{\pi},\bar{p})$ & $\proj_{\mcal P}(p+\sigma\nabla_{p}\chi(\pi(p,\bar{\pi},\bar{p}),p,\bar{\pi},\bar{p}))$  &\ \ -\tabularnewline
\hline 
$\bar{p}^{+}(\bar{\pi})$ & $\bar{p}+\mu\brbra{p(\pi(\bar{\pi},\bar{p}),\bar{\pi},\bar{p})-\bar{p}}$  &\ \ -\tabularnewline
\hline 
$\psi_{p}(\pi;\bar{\pi})$ & $\max_{p\in\mcal P}J_{\rho}(\pi,p)+\frac{r_{1}}{2}\norm{\pi-\bar{\pi}}^{2}$  &\ \ -\tabularnewline
\hline 
\multirow{2}{*}{$\psi(\pi,p,\bar{\pi})$} & \multirow{2}{*}{$J_{\rho}(\pi,p)+\frac{r_{1}}{2}\norm{\pi-\bar{\pi}}^{2}$} & $\pi(p,\bar{\pi};r_{1})\in\argmin_{\pi\in\Pi}\psi(\pi,p,\bar{\pi})$\tabularnewline
\cline{3-3} 
 &  & $\pi^{*}(\bar{\pi})\in\argmin_{\pi\in\Pi}\psi_{p}(\pi;\bar{\pi})$\tabularnewline
\hline 
$p^{+}(\bar{p})$ & $\proj_{\mcal P}(p+\sigma\nabla_{p}J_{\rho}(\pi(p,\bar{\pi};r_{1}),p))$  &\ \ -\tabularnewline
\hline 
\end{tabular}
\par\end{centering}
\caption{\label{tab:Notations}Notations.}
\end{table}

\subsection*{Difference between gradient dominance property and K\L{} property}

We now discuss the difference between gradient dominance property and K\L{} property. We recall the definition of K\L{} property.

\begin{definition}[K\L{} property~\cite{yu2022kurdyka}] We say that a proper closed function $h:\mcal X \to \mbb R\cup \cbra{\infty}$ satisfies the K\L{} property at $\hat{x}\in \dom \partial h$ if there are $c\in(0, \infty]$, a neighborhood $V$ of $\hat{x}$ and a continuous concave function $\Omega:[0, c)\to [0, \infty)$ with $\Omega(0) = 0$ such that 
\begin{itemize}
    \item $\Omega$ is continuously differentiable on $(0, c)$ with $\Omega'>0$ on $(0,c)$;
    \item For any $x\in V$ with $h(\hat{x})<h(x)<h(\hat{x}) + c,$ it holds that 
    \begin{equation}\label{eq:gradient_descent}
        \Omega'(h(x)-h(\hat{x}))\dist(0,\partial h(x))\geq 1.
    \end{equation}
\end{itemize}
\end{definition}

When we take $\Omega(z)=\bar{c}z^{1-\theta}$ for some $\bar{c}>0$ and $\theta\in[0,1)$. Then~\eqref{eq:gradient_descent} can be written as
\begin{equation*}
    \bar{c}(1-\theta) (h(x)-h(\hat{x}))^{-\theta}\dist(0,\partial h(x))\geq 1.
\end{equation*}
One may view gradient dominance property as a special case of K\L{} property with $\theta = 1$. However, this will make the first condition and \cref{eq:gradient_descent} invalid. Hence, it is challenging to make K\L{} property reduce to the gradient dominance property. This also implies that some interesting properties induced by K\L{} property fail in the context of gradient dominance property. Whereas in this paper, we combine the gradient dominance property with the weakly convex function, and derive that the Moreau envelope of a weakly convex and gradient-dominated function is also gradient-dominated.

\section{Auxillary Results and Proof of Results in~\cref{sec:properties_rmdp}\label{sec:auxiliary_lemmas}}
\begin{lemma}
[Lemma 4.3 and E.2 in~\citet{wang2023policy}]\label{lem:gradient_dominance}We
have two excellent properties:
\end{lemma}
\begin{enumerate}
\item For any fixed $p\in\cal P$, $J_{\rho}(\pi,p)$ satisfies the following
condition: 
\begin{equation}
J_{\rho}(\pi,p)-J_{\rho}(\pi^{*}(p),p)\leq\bar{D}_{\gamma}\max_{\bar{\pi}\in\Pi}\inner{\pi-\bar{\pi}}{\nabla_{\pi}J_{\rho}(\pi,p)}\leq\bar{D}_{\gamma}\max_{\bar{\pi}\in\Pi}\inner{\pi-\bar{\pi}}{\nabla_{\pi}J_{\rho}(\pi,p)+\varkappa},\forall\varkappa\in\mcal N_{\Pi}(\pi),
\end{equation}
 where $J_{\rho}(\pi^{*}(p),p):=\min_{\pi\in\Pi}J_{\rho}(\pi,p)$
and $\mcal N_{\Pi}(\pi):=\{\varkappa\mid\inner{\varkappa}{\bar{\pi}-\pi}\leq0,\forall\bar{\pi}\in\Pi\}$. 
\item Assume the ambiguity set is $s$ rectangular. Then for any $p \in \cal P$, we have 
\begin{equation}
J_{\rho}(\pi,p^{*}(\pi))-J_{\rho}(\pi,p)\leq\bar{D}_{\gamma}\max_{\bar{p}\in\mcal P}\inner{\bar{p}-p}{\nabla_{p}J_{\rho}(\pi,p)}\leq\bar{D}_{\gamma}\max_{\bar{p}\in\mcal P}\inner{\bar{p}-p}{\nabla_{p}J_{\rho}(\pi,p)-\varkappa},\forall\varkappa\in\mcal N_{\mcal P}(p),
\end{equation}
 where $J_{\rho}(\pi,p^{*}(\pi)):=\max_{p\in\mcal P}J_{\rho}(\pi,p)$,
$\bar{D}_{\gamma}=D(1-\gamma)^{-1}$ and $\mcal N_{\mcal P}(p):=\{\varkappa\mid\inner{\varkappa}{\bar{p}-p}\leq0,\forall\bar{p}\in\mcal P\}$.
\end{enumerate}
\begin{remark}
The above~\cref{lem:gradient_dominance} implies the following
two inequalities
\begin{align}
J_{\rho}(\pi,p^{*}(\pi))-J_{\rho}(\pi,p) & \leq\bar{D}_{p}\dist\brbra{\nabla_{p}J_{\rho}(\pi,p)-\mcal N_{\mcal P}(p)}\label{eq:appen_p_graddom}\\
J_{\rho}(\pi,p)-J_{\rho}(\pi^{*}(p),p) & \leq\bar{D}_{\pi}\dist\brbra{\nabla_{\pi}J_{\rho}(\pi,p)+\mcal N_{\Pi}(\pi)}\label{eq:appen_pi_graddom}
\end{align}
\end{remark}


\subsection*{Proof of~\cref{lem:smoothmdp}}


Now, we prepare to give the proof of~\cref{lem:smoothmdp}.

\begin{proof}
Item~\ref{item:i)}, Item~\ref{item:ii)} and Item~\ref{item:iii)}
can be found in Lemmas 3.1, 4.1, 4.2 in~\citet{wang2023policy}. Now,
we give a proof of Item~\ref{item:iv)}. Define 
\begin{equation*}
p_{ss'}^{\pi}(1)=\sum_{a\in\mcal A}\pi_{sa}p_{sas'},\ \ p_{ss'}^{\pi}(t-1)p_{s's''}^{\pi}(1)=p_{ss''}^{\pi}(t),\ \ q_{sa}^{\pi,p}=\sum_{s'\in\mcal S}p_{sas'}c_{sas'}+\gamma p_{sas'}(v_{s'}^{\pi,p})
\end{equation*}
Then we first give some middle results
\begin{equation*}
\begin{split}\frac{\partial p_{ss'}^{\pi}(1)}{\partial p_{sas'}}=\pi_{sa},\ \  & \frac{\partial p_{ss'}^{\pi}(1)}{\partial p_{s''as'}}\mid_{s\neq s''}=0,\ \ \frac{\partial q_{sa}^{\pi,p}}{\partial p_{\hat{s}\hat{a}\hat{s}'}}\mid_{(s,a)\neq(\hat{s},\hat{a})}=\gamma\sum_{s'\in\mcal S}p_{sas'}\frac{\partial v_{s'}^{\pi,p}}{\partial p_{\hat{s}\hat{a}\hat{s}'}},\\
\frac{\partial q_{sa}^{\pi,p}}{\partial p_{\hat{s}\hat{a}\hat{s}'}}\mid_{(s,a)=(\hat{s},\hat{a})} & =c_{sa\hat{s}'}+\gamma\sum_{s'\in\mcal S}\frac{\partial p_{sas'}}{\partial p_{\hat{s}\hat{a}\hat{s}'}}v_{s'}^{\pi,p}+\gamma\sum_{s'\in\mcal S}p_{sas'}\frac{\partial v_{s'}^{\pi,p}}{\partial p_{\hat{s}\hat{a}\hat{s}'}} =c_{sa\hat{s}'}+\gamma v_{\hat{s}'}^{\pi,p}+\gamma\sum_{s'\in\mcal S}p_{sas'}\frac{\partial v_{s'}^{\pi,p}}{\partial p_{\hat{s}\hat{a}\hat{s}'}},
\end{split}
\end{equation*}
which will be used in the following process, repeatly. Since $J_{\rho}(\pi,p)$
is $\ell_{\pi},\ell_{p}$-smooth w.r.t. $\pi$ and $p$, respectively.
Hence, if all there exist constant $C$ and $U$ such that $\abs{\partial v_{s_{1}}^{\pi,p}/(\partial\pi_{s_{2}a_{1}}\partial p_{s_{3}a_{2}s})}\leq C<\infty$ and $\left\|\frac{\partial J_{\rho}(\pi, p)}{\partial \pi \partial p}\right\|\leq U<\infty$
hold for any $s_{1},s_{2},s_{3},a_{2},a_{3},s$ then we complete
our proof of Item~\ref{item:iv)}. Now, we
give the details of how to prove the existence of $U$. 

The first order derivatives are shown as follows: 
\begin{equation}
\begin{split}
        \Bsbra{\frac{\partial v_{\hat{s}}^{\pi,p}}{\partial\pi_{sa}}}_{\hat{s}\neq s}  &=\gamma\sum_{s'\in\mcal S}p_{\hat{s}s'}^{\pi}(1)\frac{\partial v_{s'}^{\pi,p}}{\partial\pi_{sa}}\label{eq:appen_derivative1}\\
        \Bsbra{\frac{\partial v_{\hat{s}}^{\pi,p}}{\partial\pi_{sa}}}_{\hat{s}=s} &=q_{sa}^{\pi,p}+\gamma\sum_{s'\in\mcal S}p_{ss'}^{\pi,p}(1)(\frac{\partial v_{s'}^{\pi,p}}{\partial\pi_{sa}})\\
        \Bsbra{\frac{\partial v_{\hat{s}}^{\pi,p}}{\partial p_{sas'}}}_{\hat{s}\neq s}  &=\gamma\sum_{\hat{s}'\in\mcal S}p_{\hat{s}\hat{s}'}^{\pi}(1)\frac{\partial v_{\hat{s}'}^{\pi,p}}{\partial p_{sas'}}\\
        \Bsbra{\frac{\partial v_{\hat{s}}^{\pi,p}}{\partial p_{sas'}}}_{\hat{s}=s} &=\gamma\sum_{\hat{s}'}p_{s\hat{s}'}^{\pi}(1)\frac{\partial v_{\hat{s}'}^{\pi,p}}{\partial p_{sas'}}+\pi_{sa}\brbra{c_{sas'}+\gamma v_{s'}^{\pi,p}}.
\end{split}
\end{equation}
Consider $\partial v_{s_{1}}^{\pi,p}/(\partial\pi_{s_{2}a_{1}}\partial p_{s_{3}a_{2}s})$,
for simplicity, we abbreviate it as $\partial(s_{1},s_{2},a_{2},s_{3},a_{3})$:
\begin{equation}
\begin{aligned}
\partial(s_{1},s_{1},a_{2},s_{1},a_{2})&=c_{s_{1}a_{2}s}+\gamma v_{s}^{\pi,p}+\gamma\sum_{s'\in\mcal S}p_{s_{1}a_{2}s'}\frac{\partial v_{s'}^{\pi,p}}{\partial p_{s_{1}a_{2}s}} \\
& \ \ +\gamma\sum_{s'\in\mcal S}\Brbra{\pi_{s_{1}a_{2}}\cdot\frac{\partial v_{s'}^{\pi,p}}{\partial\pi_{s_{1}a_{2}}}+p_{s_{1}s'}^{\pi,p}(1)\frac{\partial v_{s'}^{\pi,p}}{\partial\pi_{s_{1}a_{2}}\partial p_{s_{1}a_{2}s}}}\\
\end{aligned}
\end{equation}
\begin{equation}
    \partial(s_{1},s_{1},a_{2},s_{1},a_{3})  =\gamma\sum_{s'\in\mcal S}p_{s_{1}a_{2}s'}\frac{\partial v_{s'}^{\pi,p}}{\partial p_{s_{1}a_{3}s}}+\gamma\sum_{s'\in\mcal S}\Brbra{\pi_{s_{1}a_{3}}\cdot\frac{\partial v_{s'}^{\pi,p}}{\partial\pi_{s_{1}a_{2}}}+p_{s_{1}s'}^{\pi,p}(1)\partial(s',s_{1},a_{2},s_{1},a_{3})}\label{eq:appen_derivate5}
\end{equation}
\begin{equation}
    \partial(s_{1},s_{2},a_{2},s_{2},a_{2})  =\gamma\sum_{s'\in\mcal S}\Brbra{p_{s_{1}s'}^{\pi}(1)\partial(s',s_{2},a_{2},s_{2},a_{2})}
\end{equation}
\begin{equation}
\partial(s_{1},s_{2},a_{2},s_{2},a_{3}) =\gamma\sum_{s'\in\mcal S}\Brbra{p_{s_{1}s'}^{\pi}(1)\partial(s',s_{2},a_{2},s_{2},a_{3})}    
\end{equation}
\begin{equation}
    \partial(s_{1},s_{2},a_{2},s_{1},a_{2})  =\gamma\sum_{s'\in\mcal S}\Brbra{\pi_{s_{1}a_{2}}\cdot\frac{\partial v_{s'}^{\pi,p}}{\partial\pi_{s_{2}a_{2}}}+p_{s_{1}s'}^{\pi}(1)\partial(s',s_{2},a_{2},s_{1},a_{2})}
\end{equation}
\begin{equation}
\partial(s_{1},s_{2},a_{2},s_{1},a_{3}) = \gamma\sum_{s'\in\mcal S}\Brbra{\pi_{s_{1}a_{2}}\cdot\frac{\partial v_{s'}^{\pi,p}}{\partial\pi_{s_{2}a_{2}}}+p_{s_{1}s'}^{\pi}(1)\partial(s',s_{2},a_{2},s_{1},a_{3})}    
\end{equation}
\begin{align}
\partial(s_{1},s_{1},a_{2},s_{3},a_{2}) & =\gamma\sum_{s'\in\mcal S}p_{s_{1}a_{2}s'}\frac{\partial v_{s'}^{\pi,p}}{\partial p_{s_{3}a_{2}s}}+\gamma\sum_{s'\in\mcal S}\Brbra{p_{s_{1}s'}^{\pi}(1)\partial(s',s_{1},a_{2},s_{3},a_{2})}\\
\partial(s_{1},s_{1},a_{2},s_{3},a_{3}) & =\gamma\sum_{s'\in\mcal S}p_{s_{1}a_{2}s'}\frac{\partial v_{s'}^{\pi,p}}{\partial p_{s_{3}a_{3}s}}+\gamma\sum_{s'\in\mcal S}\Brbra{p_{s_{1}s'}^{\pi}(1)\partial(s',s_{1},a_{2},s_{3},a_{3})}\\
\partial(s_{1},s_{2},a_{2},s_{3},a_{2}) & =\gamma\sum_{s'\in\mcal S}\Brbra{p_{s_{1}s'}^{\pi}(1)\partial(s',s_{2},a_{2},s_{3},a_{2})}\\
\partial(s_{1},s_{2},a_{2},s_{3},a_{3}) & =\gamma\sum_{s'\in\mcal S}\Brbra{p_{s_{1}s'}^{\pi}(1)\partial(s',s_{2},a_{2},s_{3},a_{3})}\label{eq:appen_derivative14}
\end{align}

Now, based on the above results, we proceed as follows:
\begin{equation*}
\begin{split} & \partial(s_{1},s_{1},a_{2},s_{1},a_{2})-(c_{s_{1}a_{2}s}+\gamma v_{s}^{\pi,p})\\
 & =\gamma\sum_{s'\in\mcal S}p_{s_{1}a_{2}s'}\frac{\partial v_{s'}^{\pi,p}}{\partial p_{s_{1}a_{2}s}}\\
 & \ \ +\gamma\sum_{s'\in\mcal S}\Brbra{\pi_{s_{1}a_{2}}\cdot\frac{\partial v_{s'}^{\pi,p}}{\partial\pi_{s_{1}a_{2}}}+p_{s_{1}s'}^{\pi,p}(1)\frac{\partial v_{s'}^{\pi,p}}{\partial\pi_{s_{1}a_{2}}\partial p_{s_{1}a_{2}s}}}\\
     & ={\gamma\sum_{s'\neq s_{1}}p_{s_{1}a_{2}s'}\cdot\gamma\sum_{s''\in\mcal S}p_{s's''}^{\pi}(1)\frac{\partial v_{s''}^{\pi,p}}{\partial p_{s_{1}a_{2}s}}}+\gamma p_{s_{1}a_{2}s_{1}}\Brbra{{\gamma\sum_{s''}p_{s's''}^{\pi}(1)\frac{\partial v_{s''}^{\pi,p}}{\partial p_{s_{1}a_{2}s}}}+\pi_{s_{1}a_{2}}\brbra{c_{s_{1}a_{2}s}+\gamma v_{s}^{\pi,p}}}\\
 & \ \ +\gamma\Brbra{\pi_{s_{1}a_{2}}\cdot\frac{\partial v_{s_{1}}^{\pi,p}}{\partial\pi_{s_{1}a_{2}}}+p_{s_{1}s_{1}}^{\pi,p}(1)\frac{\partial v_{s_{1}}^{\pi,p}}{\partial\pi_{s_{1}a_{2}}\partial p_{s_{1}a_{2}s}}}+\gamma\sum_{s'\neq s_{1}}\Brbra{\pi_{s_{1}a_{2}}\cdot\frac{\partial v_{s'}^{\pi,p}}{\partial\pi_{s_{1}a_{2}}}+p_{s_{1}s'}^{\pi,p}(1)\frac{\partial v_{s'}^{\pi,p}}{\partial\pi_{s_{1}a_{2}}\partial p_{s_{1}a_{2}s}}}\\
 &\aeq {\gamma\pi_{s_{1}a_{2}}\cdot q_{s_{1}a_{2}}^{\pi,p}+\gamma p_{s_{1}a_{2}s_{1}}\pi_{s_{1}a_{2}}\brbra{c_{s_{1}a_{2}s}+\gamma v_{s}^{\pi,p}}+\gamma p_{s_{1}s_{1}}^{\pi,p}(1)\brbra{c_{s_{1}a_{2}s}+\gamma v_{s}^{\pi,p}}}\\
 & \ \ {+\gamma^{2}\sum_{s',s''}\Brbra{\pi_{s_{1}a_{2}}p_{s's''}^{\pi,p}(1)\frac{\partial v_{s''}^{\pi,p}}{\partial\pi_{s_{1}a_{2}}}+p_{s_{1}a_{2}s'}p_{s's''}^{\pi}(1)\frac{\partial v_{s''}^{\pi,p}}{\partial p_{s_{1}a_{2}s}}}}\\
 & \ \ +\gamma^{2}p_{s_{1}s_{1}}^{\pi,p}(1)\sum_{s'}\Brbra{p_{s_{1}a_{2}s'}\frac{\partial v_{s'}^{\pi,p}}{\partial p_{s_{1}a_{2}s}}+\brbra{\pi_{s_{1}a_{2}}\cdot\frac{\partial v_{s'}^{\pi,p}}{\partial\pi_{s_{1}a_{2}}}}}\\
 & \ \ +\gamma^{2}\sum_{s'}p_{s_{1}s'}^{\pi,p}(2)\Brbra{\frac{\partial v_{s''}^{\pi,p}}{\partial\pi_{s_{1}a_{2}}\partial p_{s_{1}a_{2}s}}},
\end{split}
\end{equation*}
where $(a)$ holds by~\eqref{eq:appen_derivative1} and~\eqref{eq:appen_derivative14}.

Since $J_{\rho}(\pi,p)$ is $L_{\pi},L_{p}$ continuous,
then there exist two constants $c_{p},c_{\pi}$ such that
\begin{equation*}
\abs{\pi_{s_{1}a_{2}}p_{s's''}^{\pi,p}(1)\frac{\partial v_{s''}^{\pi,p}}{\partial\pi_{s_{1}a_{2}}}}\leq c_{p},\ \ \abs{p_{s_{1}a_{2}s'}p_{s's''}^{\pi}(1)\frac{\partial v_{s''}^{\pi,p}}{\partial p_{s_{1}a_{2}s}}}\leq c_{\pi},\forall s'',s_{1},a_{2}.
\end{equation*}
Since $c_{s_{t}a_{t}s_{t+1}}\in[0,1]$, then we have
\begin{align}
\abs{v_{s}^{\pi,p}}=\mbb E_{\pi,p}\Bsbra{\sum_{t=0}^{\infty}\gamma^{t}c_{s_{t}a_{t}s_{t+1}}\mid s_{0}=s} & \leq\frac{1}{1-\gamma},\label{eq:appen_vup}\\
\abs{q_{sa}^{\pi,p}}=\Babs{\mbb E_{\pi,p}\bsbra{\sum_{t=0}^{\infty}\gamma^{t}c_{s_{t}a_{t}s_{t+1}}\mid s_{0}=s,a_{0}=a}} & \leq\frac{1}{1-\gamma}.\label{eq:appen_qup}
\end{align}
Now we have
\begin{equation*}
\begin{split} & \Babs{\partial(s_{1},s_{1},a_{2},s_{1},a_{2})}=c_{s_{1}a_{2}s}+\gamma v_{s}^{\pi,p}+\gamma\abs{h(1)}\\
 & \leq c_{s_{1}a_{2}s}+\gamma v_{s}^{\pi,p}+\gamma\brbra{\frac{3}{1-\gamma}}+\gamma^{2}\abs S^{2}\brbra{c_{p}+c_{\pi}}+\gamma^{2}\abs{h(2)}\\
 & =c_{s_{1}a_{2}s}+\gamma v_{s}^{\pi,p}+\gamma C_{1}+\gamma^{2}C_{2}+\gamma^{2}\abs{h(2)}\leq c_{s_{1}a_{2}s}+\gamma v_{s}^{\pi,p}+\gamma^{2}(C_{2}+C_{1})+\gamma^{2}\abs{h(2)},
\end{split}
\end{equation*}
where $h(t)=p_{s_{1}s_{1}}^{\pi,p}(t-1)\sum_{s'}\Brbra{p_{s_{1}a_{2}s'}\frac{\partial v_{s'}^{\pi,p}}{\partial p_{s_{1}a_{2}s}}+\brbra{\pi_{s_{1}a_{2}}\cdot\frac{\partial v_{s'}^{\pi,p}}{\partial\pi_{s_{1}a_{2}}}}}+\sum_{s'}p_{s_{1}s'}^{\pi,p}(t)\Brbra{\frac{\partial v_{s''}^{\pi,p}}{\partial\pi_{s_{1}a_{2}}\partial p_{s_{1}a_{2}s}}}$,
$C_{1}=\frac{3}{1-\gamma}$ and $C_{2}=\abs{\mcal S}^{2}\brbra{c_{p}+c_{\pi}}$,
and $(a)$ holds by~\eqref{eq:appen_qup} and~\eqref{eq:appen_vup}, $c_{s_{t}a_{t}s_{t+1}}\leq1$,
$p_{s_{1}s_{1}}^{\pi,p}(1)\leq1$ and $\pi_{s_{1}a_{2}}\leq1$. Hence,
we conclude that
\begin{equation*}
\Babs{\partial(s_{1},s_{1},a_{2},s_{1},a_{2})}\leq c_{s_{1}a_{2}s}+\gamma v_{s}^{\pi,p}+(C_{2}+C_{1})\sum_{t=0}^{\infty}\gamma^{t}\leq\frac{1+C_{1}+C_{2}}{1-\gamma}.
\end{equation*}
Similar results hold for the remaining 9 cases, then we have there
exists a constant $C<\infty$ such that 
\begin{equation*}
\Babs{\partial(s_{1},s_{1},a_{2},s_{1},a_{2})}\leq C,\forall s_{1},s_{2},s_{3}\in\mcal S,\forall a_{2},a_{3}\in\mcal A.
\end{equation*}
By the equivalence of the matrix norm, we have there exists a constant
$U$ such that
\begin{equation*}
\left\|\frac{\partial J_{\rho}(\pi,p)}{\partial\pi\partial p}\right\|\leq U<\infty.
\end{equation*}
Here the $\norm{\cdot}$ is the spectral norm.
\end{proof}

\subsection*{Proof of~\cref{thm:Moreau_smooth_grad_dom}}
\begin{theorem}
\label{thm:append_Moreau_smooth_grad_dom}
Let $\pi^{*}$ be a global optimal policy for RMDPs. Denote $C_{\ell_{\pi}} =\frac{D\sqrt{S A}}{1-\gamma}+\frac{L_{\pi}}{2\ell_{\pi}},$
where $D:=\sup_{\pi\in\Pi,p\in\mcal P}\norm{d_{\rho}^{\pi,p}/\rho}<\infty$.
Then for any policy $\pi\in\Pi$, we have
\begin{equation}\label{eq:append_gradient-like-property}
\phi_{2\ell_{\pi}}(\pi) \leq \phi(\pi)\leq C_{\ell_{\pi}} \norm{\nabla\phi_{2\ell_{\pi}}(\pi)}+\phi(\pi^{*}),
\end{equation}
where $\phi_{2\ell_{\pi}}(\pi):=\inf_{\pi'\in\Pi}\{\phi(\pi')+\ell_{\pi}\norm{\pi-\pi'}^{2}\}$.
\end{theorem}
\begin{proof}
    Proof can be found in~\citet{wang2023policy} or~\citet{li2024policy}.
\end{proof}
\begin{remark}
    Mismatch coefficient $D<\infty$ can be guaranteed by~\cref{assu:irreduc}.
\end{remark}

\subsection*{Proof of~\cref{thm:moreau_grad_dom}}
\begin{proof}
Let $\tilde{x}(x)\in\argmin_{z\in\mcal X}f(z)+r\norm{z-x}^{2}$. It
follows from the definition of $f_{2r}(x)$ and $\min_{x\in\mcal X}f(x)=\min_{x\in\mcal X}f_{2r}(x)=f^{*}$
that
\begin{equation}
\begin{split} & f_{2r}(x)-f^{*}\\
 & =\inf_{z\in\mcal X}f(z)+r\norm{z-x}^{2}-f^{*}\\
 & \leq f(x)-f^{*}=f(x)-f(\tilde{x}(x))+f(\tilde{x}(x))-f^{*}\\
 & \aleq L_{f}\norm{\tilde{x}(x)-x}+C\dist\brbra{\partial f(\tilde{x}(x))+\mcal N_{\mcal X}(\tilde{x}(x))},
\end{split}
\label{eq:appen_temp1}
\end{equation}
where $(a)$ holds by $f(x)$ is $L_{f}$ continuity and the gradient
dominance property of $f$. By the definition of $f_{2r}(x)$, we
have
\begin{equation}
0\in\partial f(\tilde{x}(x))+2r(\tilde{x}(x)-x)+\mcal N_{\mcal X}(\tilde{x}(x)).\label{eq:appen_temp2-1}
\end{equation}
Furthermore, by the Moreau envelope property, we have
\begin{equation}
\nabla f_{2r}(x)\in\partial f(\tilde{x}(x))+\mcal N_{\mcal X}(\tilde{x}(x)),\ \ \nabla f_{2r}(x)=2r(\tilde{x}(x)-x).\label{eq:appen_temp3}
\end{equation}
Combining~\eqref{eq:appen_temp1},~\eqref{eq:appen_temp2-1} and~\eqref{eq:appen_temp3} yields
\begin{equation*}
f_{2r}(x)-f^{*}\leq \Brbra{\frac{L_{f}}{2r}+C}\norm{\nabla f_{2r}(x)}.
\end{equation*}
\end{proof}

\section{Convergence Analysis\label{sec:appen_convergence_ana}}
\subsection{Proof of Results in Section~\ref{sec:general_con}}
\begin{lemma}\label{lem:descent_lemma}For any $\pi,\pi'\in\Pi$,
$p,p'\in\mathcal{P}$, $\bar{\pi}$
and $\bar{p}$, we have
\[
\begin{aligned} & \frac{r_{1}-\ell_{\pi p}}{2}\|\pi-\pi'\|^{2}\leq\chi(\pi^{\prime},p,\bar{\pi},\bar{p})-\chi(\pi,p,\bar{\pi},\bar{p})-\langle\nabla_{\pi}\chi(\pi,p,\bar{\pi},\bar{p}),\pi'-\pi \rangle\leq\frac{\ell_{\pi p}+r_{1}}{2}\|\pi-\pi'\|^{2},\\
  &\ \ -\frac{\ell_{p\pi}+r_{2}}{2}\|p-p'\|^{2}\leq\chi(\pi,p^{\prime},\bar{\pi},\bar{p})-\chi(\pi,p,\bar{\pi},\bar{p})-\langle\nabla_{p}\chi(\pi,p,\bar{\pi},\bar{p}),p'-p\rangle\leq\frac{\ell_{p\pi}-r_{2}}{2}\|p-p'\|^{2}.
\end{aligned}
\]
\end{lemma}
\begin{proof}
Since $J_{\rho}(\pi,p)$ is $\ell_{\pi p}$, $\ell_{p\pi}$ smooth,
respectively, we have

\begin{equation}
\begin{aligned}  &\ \ -\frac{\ell_{\pi p}}{2}\|\pi-\pi'\|^{2}\leq J_{\rho}(\pi',p)-J_{\rho}(\pi,p)-\langle\nabla_{\pi}J_{\rho}(\pi,p),\pi'-\pi \rangle\leq\frac{\ell_{\pi p}}{2}\|\pi-\pi'\|^{2},\\
  &\ \ -\frac{\ell_{p\pi}}{2}\|p-p'\|^{2}\leq J_{\rho}(\pi,p')-J_{\rho}(\pi,p)-\langle\nabla_{p}J_{\rho}(\pi,p),p'-p\rangle\leq\frac{\ell_{p\pi}}{2}\|p-p'\|^{2}.
\end{aligned}
\label{eq:appen_Lsmooth}
\end{equation}

Hence, consider function $\chi(\pi,p,\bar{\pi},\bar{p})$, we know
that
\begin{equation}
\begin{aligned} & \chi(\pi^{\prime},p,\bar{\pi},\bar{p})-\chi(\pi,p,\bar{\pi},\bar{p})-\langle\nabla_{\pi}\chi(\pi,p,\bar{\pi},\bar{p}),\pi'-\pi\rangle\\
=\  & J_{\rho}(\pi',p)-J_{\rho}(\pi,p)-\langle\nabla_{\pi}J_{\rho}(\pi,p)+r_{1}(\pi-\bar{\pi}),\pi'-\pi \rangle+\frac{r_{1}}{2}\|\pi'-\bar{\pi}\|^{2}-\frac{r_{1}}{2}\|\pi-\bar{\pi}\|^{2}\\
=\  & J_{\rho}(\pi',p)-J_{\rho}(\pi,p)-\langle\nabla_{\pi}J_{\rho}(\pi,p),\pi'-\pi\rangle+\frac{r_{1}}{2}\|\pi'-\pi\|^{2}
\end{aligned}
\label{eq:appen_1}
\end{equation}
and similarly 
\begin{equation}
\begin{aligned} & \chi(\pi,p^{\prime},\bar{\pi},\bar{p})-\chi(\pi,p,\bar{\pi},\bar{p})-\langle\nabla_{p}\chi(\pi,p,\bar{\pi},\bar{p}),p'-p\rangle\\
=\  & J_{\rho}(\pi,p')-J_{\rho}(\pi,p)-\langle\nabla_{p}J_{\rho}(\pi,p)-r_{2}(p-\bar{p}),p'-p\rangle-\frac{r_{2}}{2}\|p'-\bar{p}\|^{2}+\frac{r_{2}}{2}\|p-\bar{p}\|^{2}\\
=\  & J_{\rho}(\pi,p')-J_{\rho}(\pi,p)-\langle\nabla_{p}J_{\rho}(\pi,p),p'-p\rangle-\frac{r_{2}}{2}\|p'-p\|^{2}.
\end{aligned}
\label{eq:appen_2}
\end{equation}
Combing \eqref{eq:appen_Lsmooth}, \eqref{eq:appen_1} and \eqref{eq:appen_2}, we directly
obtain the desired results. 
\end{proof}
\begin{lemma}
[Lemma 2 in~\citet{zheng2023universal}]\label{lem:lem2_un}Suppose
that $r_{2}>(\frac{\ell_{p\pi}}{r_{1}-\ell_{\pi p}}+2)\ell_{p\pi}$,
then for any $\pi,\pi'\in\Pi$, $p,p'\in\mathcal{P}$, $\bar{\pi},\bar{\pi}'\in\mathbb{R}^{|\mathcal{S}|\times|\mathcal{A}|}$
and $p,p'\in\mbb R^{|\mathcal{S}|}$. Then the following inequalities
hold:
\end{lemma}
\begin{enumerate}
\item \label{enu:1}$\|\pi(p^{\prime},\bar{\pi},\bar{p})-\pi(p,\bar{\pi},\bar{p})\|\leq\kappa_{1}\|p'-p\|$,
\item \label{enu:2}$\ensuremath{\|\pi(p,\bar{\pi}',\bar{p})-\pi(p,\bar{\pi},\bar{p})\|\leq\kappa_{2}\|\bar{\pi}-\bar{\pi}'\|}$,
\item \label{enu:3}$\ensuremath{\|\pi(\bar{\pi}',\bar{p})-\pi(\bar{\pi},\bar{p})\|\leq\kappa_{2}\|\bar{\pi}-\bar{\pi}'\|}$,
\item \label{enu:4}$\ensuremath{\|p(\bar{\pi},\bar{p})-p(\bar{\pi}',\bar{p})\|\leq\kappa_{3}\|\bar{\pi}-\bar{\pi}'\|}$,
\item \label{enu:5}$\|p(\pi,\bar{\pi},\bar{p})-p(\pi',\bar{\pi},\bar{p})\|\leq\kappa_{4}\|\pi-\pi'\|$
\item \label{enu:6}$\ensuremath{\|p(\pi,\bar{\pi},\bar{p})-p(\pi,\bar{\pi},\bar{p}')\|\leq\kappa_{5}\|\bar{p}-\bar{p}'\|}$
\item \label{enu:7}$\ensuremath{\|p(\bar{\pi},\bar{p})-p(\bar{\pi},\bar{p}')\|\leq\kappa_{5}\|\bar{p}-\bar{p}'\|}$,
\end{enumerate}
where $\kappa_{1}=\frac{\ell_{p\pi}+r_{1}-\ell_{\pi p}}{r_{1}-\ell_{\pi p}}$,
$\kappa_{2}=\frac{r_{1}}{r_{1}-\ell_{\pi p}}$, $\kappa_{3}=\frac{r_{1}\kappa_{1}}{r_{2}-\ell_{p\pi}}+\frac{\kappa_{2}}{\kappa_{1}}$,
$\kappa_{4}=\frac{\ell_{\pi p}+r_{2}-\ell_{p\pi}}{r_{2}-\ell_{p\pi}}$,
and $\kappa_{5}=\frac{r_{2}}{r_{2}-\ell_{p\pi}}$.
\begin{proof}
Item~\ref{enu:1}: It follows from~\cref{lem:descent_lemma}
that
\begin{align}
\chi(\pi(p,\bar{\pi},\bar{p}),p^{\prime},\bar{\pi},\bar{p})-\chi(\pi(p^{\prime},\bar{\pi},\bar{p}),p^{\prime},\bar{\pi},\bar{p})  & \geq\frac{r_{1}-\ell_{\pi p}}{2}\|\pi(p,\bar{\pi},\bar{p})-\pi(p^{\prime},\bar{\pi},\bar{p})\|^{2},\label{eq:appen_lem2_un_1}\\
\chi(\pi(p,\bar{\pi},\bar{p}),p^{\prime},\bar{\pi},\bar{p})-\chi(\pi(p,\bar{\pi},\bar{p}),p,\bar{\pi},\bar{p}) & \leq\langle\nabla_{p}\chi(\pi(p,\bar{\pi},\bar{p}),p,\bar{\pi},\bar{p}),p'-p\rangle+\frac{\ell_{p\pi}-r_{2}}{2}\|p-p'\|^{2},\label{eq:appen_lem2_un_2}\\
\chi(\pi(p^{\prime},\bar{\pi},\bar{p}),p,\bar{\pi},\bar{p})-\chi(\pi(p^{\prime},\bar{\pi},\bar{p}),p^{\prime},\bar{\pi},\bar{p}) & \leq\langle\nabla_{p}\chi(\pi(p^{\prime},\bar{\pi},\bar{p}),p,\bar{\pi},\bar{p}),p-p'\rangle+\frac{\ell_{p\pi}+r_{2}}{2}\|p-p'\|^{2},\label{eq:appen_lem2_un_3}\\
\chi(\pi(p,\bar{\pi},\bar{p}), p, \bar{\pi},\bar{p})-\chi(\pi(p^{\prime},\bar{\pi},\bar{p}),p, \bar{\pi},\bar{p}) & \leq\frac{\ell_{\pi p}-r_{1}}{2}\|\pi(p,\bar{\pi},\bar{p})-\pi(p^{\prime},\bar{\pi},\bar{p})\|^{2}.\label{eq:appen_lem2_un_4}
\end{align}

Summing up~\eqref{eq:appen_lem2_un_2} and~\eqref{eq:appen_lem2_un_3} and combining
it with~\eqref{eq:appen_lem2_un_1} yields
\begin{equation}
\begin{split} & (r_{1}-\ell_{\pi p})\|\pi(p,\bar{\pi},\bar{p})-\pi(p^{\prime},\bar{\pi},\bar{p})\|^{2}\\
 & \leq\langle\nabla_{p}\chi(\pi(p,\bar{\pi},\bar{p}),p,\bar{\pi},\bar{p})-\nabla_{p}\chi(\pi(p^{\prime},\bar{\pi},\bar{p}),p,\bar{\pi},\bar{p}),p'-p\rangle+\ell_{p\pi}\|p-p'\|^{2}\\
 & =\inner{\brbra{\nabla_{p}J(\pi(p,\bar{\pi},\bar{p}),p)-r_{2}\brbra{p-\bar{p}}}-\brbra{\nabla_{p}J(\pi(p^{\prime},\bar{\pi},\bar{p}),p)-r_{2}\brbra{p-\bar{p}}}}{p'-p}+\ell_{p\pi}\|p-p'\|^{2}\\
 & \aleq\ell_{p\pi}\|\pi(p^{\prime},\bar{\pi},\bar{p})-\pi(p,\bar{\pi},\bar{p})\|\|p'-p\|+\ell_{p\pi}\|p-p'\|^{2},
\end{split}
\label{eq:appen_mid1}
\end{equation}
where $(a)$ holds by Cauchy Schwarz inequality and Item~\ref{item:iv)} in~\cref{lem:smoothmdp}.
Let $\Lambda:=\|\pi(p^{\prime},\bar{\pi},\bar{p})-\pi(p,\bar{\pi},\bar{p})\|/\norm{p'-p}$.
Then~\eqref{eq:appen_mid1} can be rewritten as 
\begin{align}
\Lambda^{2} & \leq\frac{\ell_{p\pi}}{r_{1}-\ell_{\pi p}}+\frac{\ell_{p\pi}}{r_{1}-\ell_{\pi p}}\Lambda\aleq\frac{\ell_{p\pi}}{r_{1}-\ell_{\pi p}}+\frac{1}{2}\Brbra{\frac{\ell_{p\pi}}{r_{1}-\ell_{\pi p}}}^{2}+\frac{1}{2}\Lambda^{2},\label{eq:appen_lambda1}\\
\Rightarrow\Lambda & \leq\sqrt{\Brbra{\frac{\ell_{p\pi}}{r_{1}-\ell_{\pi p}}}^{2}+2\Brbra{\frac{\ell_{p\pi}}{r_{1}-\ell_{\pi p}}}}\leq\frac{\ell_{p\pi}}{r_{1}-\ell_{\pi p}}+1:=\kappa_{1}.\nonumber 
\end{align}
where $(a)$ holds by $\inner{\abf}{\bbf}\leq(\norm{\abf}^{2}+\norm{\bbf}^{2})/2$.

Item~\ref{enu:2} and Item~\ref{enu:3}: It follows from from~\cref{lem:descent_lemma}
that
\begin{equation}
\begin{split}\chi(\pi(p,\bar{\pi},\bar{p}),p,\bar{\pi}',\bar{p})-\chi(\pi(p,\bar{\pi}',\bar{p}),p,\bar{\pi}',\bar{p})  & \geq\frac{r_{1}-\ell_{\pi p}}{2}\|\pi(p,\bar{\pi},\bar{p})-\pi(p,\bar{\pi}',\bar{p})\|^{2},\\
\chi(\pi(p,\bar{\pi},\bar{p}),p,\bar{\pi},\bar{p})-\chi(\pi(p,\bar{\pi}',\bar{p}),p,\bar{\pi},\bar{p}) & \leq\frac{\ell_{\pi p}-r_{1}}{2}\|\pi(p,\bar{\pi},\bar{p})-\pi(p,\bar{\pi}',\bar{p})\|^{2}.
\end{split}
\label{eq:appen_temp4}
\end{equation}
By the definition of $\chi$, we have
\begin{equation}
\begin{split}\chi(\pi(p,\bar{\pi},\bar{p}),p,\bar{\pi}',\bar{p})-\chi(\pi(p,\bar{\pi},\bar{p}),p,\bar{\pi},\bar{p}) & =\frac{r_{1}}{2}\langle\bar{\pi}'+\bar{\pi}-2\pi(p,\bar{\pi},\bar{p}),\bar{\pi}'-\bar{\pi}\rangle,\\
\chi(\pi(p,\bar{\pi}',\bar{p}),p,\bar{\pi},\bar{p})-\chi(\pi(p,\bar{\pi}',\bar{p}),p,\bar{\pi}',\bar{p}) & =\frac{r_{1}}{2}\langle\bar{\pi}+\bar{\pi}'-2\pi(p,\bar{\pi}',\bar{p}),\bar{\pi}-\bar{\pi}'\rangle.
\end{split}
\label{eq:appen_temp5}
\end{equation}
Putting~\eqref{eq:appen_temp4} and~\eqref{eq:appen_temp5} together and using
the Cauchy-Schwarz inequality, we obtain
\[
\begin{split} & (r_{1}-\ell_{\pi p})\|\pi(p,\bar{\pi},\bar{p})-\pi(p,\bar{\pi}',\bar{p})\|^{2}\\
 & \leq r_{1}\langle\pi(p,\bar{\pi}',\bar{p})-\pi(p,\bar{\pi},\bar{p}),\bar{\pi}'-\bar{\pi}\rangle\\
 & \leq r_{1}\|\pi(p,\bar{\pi}',\bar{p})-\pi(p,\bar{\pi},\bar{p})\|\|\bar{\pi}'-\bar{\pi}\|,
\end{split}
\]
which completes the proof of Item~\ref{enu:2}. Since $\varphi_{p}(\pi,\bar{\pi},\bar{p})$
is $\brbra{r_{1}-\ell_{\pi p}}-$strongly convex in $\pi$, then we
have
\begin{equation}
\begin{split}\varphi_{p}(\pi(\bar{\pi},\bar{p}),\bar{\pi}',\bar{p})-\varphi_{p}(\pi(\bar{\pi}',\bar{p}),\bar{\pi}',\bar{p})  & \geq\frac{r_{1}-\ell_{\pi p}}{2}\norm{\pi(\bar{\pi},\bar{p})-\pi(\bar{\pi}',\bar{p})}^{2}\\
\varphi_{p}(\pi(\bar{\pi},\bar{p}),\bar{\pi},\bar{p})-\varphi_{p}(\pi(\bar{\pi}',\bar{p}),\bar{\pi},\bar{p}) & \leq\frac{\ell_{\pi p}-r_{1}}{2}\norm{\pi(\bar{\pi},\bar{p})-\pi(\bar{\pi}',\bar{p})}^{2}
\end{split}
\label{eq:appen_temp6}
\end{equation}
By the definition of $\varphi_{p}$, we have
\begin{equation}
\begin{split}\varphi_{p}(\pi(\bar{\pi},\bar{p}),\bar{\pi}',\bar{p})-\varphi_{p}(\pi(\bar{\pi},\bar{p}),\bar{\pi},\bar{p}) & =\frac{r_{1}}{2}\inner{\bar{\pi}+\bar{\pi}'-2\pi(\bar{\pi},\bar{p})}{\bar{\pi}'-\bar{\pi}}\\
\varphi_{p}(\pi(\bar{\pi}',\bar{p}),\bar{\pi},\bar{p})-\varphi_{p}(\pi(\bar{\pi}',\bar{p}),\bar{\pi}',\bar{p}) & =\frac{r_{1}}{2}\inner{\bar{\pi}+\bar{\pi}'-2\pi(\bar{\pi}',\bar{p})}{\bar{\pi}-\bar{\pi}'}.
\end{split}
\label{eq:appen_temp7}
\end{equation}
Putting~\eqref{eq:appen_temp6} and~\eqref{eq:appen_temp7} together and using
the Cauchy-Schwarz inequality, we obtain the result Item~\ref{enu:3}. 

Item~\ref{enu:4}: Since $-\varphi_{\pi}(p,\bar{\pi},\bar{p})$ is
$(r_{2}-\ell_{p\pi})$ strongly convex w.r.t. $p$, then we have 
\begin{equation}
\begin{split}\varphi_{\pi}(p(\bar{\pi},\bar{p}),\bar{\pi},\bar{p})-\varphi_{\pi}(p(\bar{\pi}',\bar{p}),\bar{\pi},\bar{p})  & \geq\frac{r_{2}-\ell_{p\pi}}{2}\|p(\bar{\pi},\bar{p})-p(\bar{\pi}',\bar{p})\|^{2},\\
\varphi_{\pi}(p(\bar{\pi}',\bar{p}),\bar{\pi}',\bar{p})-\varphi_{\pi}(p(\bar{\pi},\bar{p}),\bar{\pi}',\bar{p})  & \geq\frac{r_{2}-\ell_{p\pi}}{2}\|p(\bar{\pi}',\bar{p})-p(\bar{\pi},\bar{p})\|^{2},
\end{split}
\label{eq:appen_temp8}
\end{equation}
By the definition of $\varphi_{\pi}$, we have
\begin{equation}
\begin{split} & \varphi_{\pi}(p(\bar{\pi},\bar{p}),\bar{\pi},\bar{p})-\varphi_{\pi}(p(\bar{\pi},\bar{p}),\bar{\pi}',\bar{p})\\
 & =\chi(\pi(p(\bar{\pi},\bar{p}),\bar{\pi},\bar{p}),p(\bar{\pi},\bar{p}),\bar{\pi},\bar{p})-\chi(\pi(p(\bar{\pi},\bar{p}),\bar{\pi}',\bar{p}),p(\bar{\pi},\bar{p}),\bar{\pi}',\bar{p})\\
 & \leq\chi(\pi(p(\bar{\pi},\bar{p}),\bar{\pi}',\bar{p}),p(\bar{\pi},\bar{p}),\bar{\pi},\bar{p})-\chi(\pi(p(\bar{\pi},\bar{p}),\bar{\pi}',\bar{p}),p(\bar{\pi},\bar{p}),\bar{\pi}',\bar{p})\\
 & =\frac{r_{1}}{2}\langle\bar{\pi}+\bar{\pi}'-2\pi(p(\bar{\pi},\bar{p}),\bar{\pi}',\bar{p}),\bar{\pi}-\bar{\pi}'\rangle
\end{split}
\label{eq:appen_temp9}
\end{equation}
and 
\begin{equation}
\begin{split} & \varphi_{\pi}(p(\bar{\pi}',\bar{p}),\bar{\pi}',\bar{p})-\varphi_{\pi}(p(\bar{\pi}',\bar{p}),\bar{\pi},\bar{p})\\
 & =\chi(\pi(p(\bar{\pi}',\bar{p}),\bar{\pi}',\bar{p}),p(\bar{\pi}',\bar{p}),\bar{\pi}',\bar{p})-\chi(\pi(p(\bar{\pi}',\bar{p}),\bar{\pi},\bar{p}),p(\bar{\pi}',\bar{p}),\bar{\pi},\bar{p})\\
 & \leq\chi(\pi(p(\bar{\pi}',\bar{p}),\bar{\pi},\bar{p}),p(\bar{\pi}',\bar{p}),\bar{\pi}',\bar{p})-\chi(\pi(p(\bar{\pi}',\bar{p}),\bar{\pi},\bar{p}),p(\bar{\pi}',\bar{p}),\bar{\pi},\bar{p})\\
 & =\frac{r_{1}}{2}\langle\bar{\pi}'+\bar{\pi}-2x(p(\bar{\pi}',\bar{p}),\bar{\pi},\bar{p}),\bar{\pi}'-\bar{\pi}\rangle.
\end{split}
\label{eq:appen_temp10}
\end{equation}

Putting~\eqref{eq:appen_temp8},~\eqref{eq:appen_temp9} and~\eqref{eq:appen_temp10}
together yields
\begin{equation}
\begin{split} & (r_{2}-\ell_{p\pi})\norm{p(\bar{\pi},\bar{p})-p(\bar{\pi}',\bar{p})}^{2}\\
 & \leq r_{1}\langle\pi(p(\bar{\pi}',\bar{p}),\bar{\pi},\bar{p})-\pi(p(\bar{\pi},\bar{p}),\bar{\pi}',\bar{p}),\bar{\pi}-\bar{\pi}'\rangle\\
 & \aleq r_{1}\norm{\bar{\pi}-\bar{\pi}'}\brbra{\kappa_{1}\norm{\pi(p(\bar{\pi}',\bar{p})-\pi(p(\bar{\pi},\bar{p})}+\kappa_{2}\norm{\bar{\pi}-\bar{\pi}'}}\\
 & =r_{1}\kappa_{1}\norm{\bar{\pi}-\bar{\pi}'}\norm{\pi(p(\bar{\pi}',\bar{p})-\pi(p(\bar{\pi},\bar{p})}+r_{1}\kappa_{2}\norm{\bar{\pi}-\bar{\pi}'}^{2}
\end{split}
\label{eq:appen_temp11}
\end{equation}
where $(a)$ holds by Item~\ref{enu:1} and Item~\ref{enu:2}. Let
$\Lambda=\norm{p(\bar{\pi},\bar{p})-p(\bar{\pi}',\bar{p})}/\norm{\bar{\pi}-\bar{\pi}'}$
here, then~\eqref{eq:appen_temp11} implies 
\begin{equation}
\begin{split} & \Lambda^{2}\leq\frac{r_{1}\kappa_{1}}{r_{2}-\ell_{p\pi}}\Lambda+\frac{r_{1}\kappa_{2}}{r_{2}-\ell_{p\pi}}\leq\frac{1}{2}\Brbra{\frac{r_{1}\kappa_{1}}{r_{2}-\ell_{p\pi}}}^{2}+\frac{1}{2}\Lambda^{2}+\frac{r_{1}\kappa_{2}}{r_{2}-\ell_{p\pi}}\\
\Rightarrow & \Lambda^{2}\leq\Brbra{\frac{r_{1}\kappa_{1}}{r_{2}-\ell_{p\pi}}}^{2}+\frac{2r_{1}\kappa_{2}}{r_{2}-\ell_{p\pi}}\leq\Brbra{\frac{r_{1}\kappa_{1}}{r_{2}-\ell_{p\pi}}+\frac{\kappa_{2}}{\kappa_{1}}}^{2}=:\kappa_{3}^{2},
\end{split}
\label{eq:appen_lambda2}
\end{equation}
which complete our proof of Item~\ref{enu:4}.

Item~\ref{enu:5}. It follows from $\chi(\cdot,p,\cdot,\cdot)$ is
$(r_{2}-\ell_{p\pi})-$strongly concave that
\begin{equation*}
\begin{split}\chi(\pi,p(\pi,\bar{\pi},\bar{p}),\bar{\pi},\bar{p})-\chi(\pi,p(\pi',\bar{\pi},\bar{p}),\bar{\pi},\bar{p})  & \geq\frac{r_{2}-\ell_{p\pi}}{2}\|p(\pi,\bar{\pi},\bar{p})-p(\pi',\bar{\pi},\bar{p})\|^{2},\\
\chi(\pi',p(\pi,\bar{\pi},\bar{p}),\bar{\pi},\bar{p})-\chi(\pi',p(\pi',\bar{\pi},\bar{p}),\bar{\pi},\bar{p}) & \leq\frac{\ell_{p\pi}-r_{2}}{2}\|p(\pi,\bar{\pi},\bar{p})-p(\pi',\bar{\pi},\bar{p})\|^{2},\\
\chi(\pi,p(\pi,\bar{\pi},\bar{p}),\bar{\pi},\bar{p})-\chi(\pi',p(\pi,\bar{\pi},\bar{p}),\bar{\pi},\bar{p}) & \leq\langle\nabla_{\pi}\chi(\pi',p(\pi,\bar{\pi},\bar{p}),\bar{\pi},\bar{p}),\pi-\pi'\rangle+\frac{\ell_{\pi p}+r_{1}}{2}\|\pi-\pi'\|^{2},\\
\chi(\pi',p(\pi',\bar{\pi},\bar{p}),\bar{\pi},\bar{p})-\chi(\pi,p(\pi',\bar{\pi},\bar{p}),\bar{\pi},\bar{p}) & \leq\langle\nabla_{\pi}\chi(\pi',p(\pi',\bar{\pi},\bar{p}),\bar{\pi},\bar{p}),x'-x\rangle+\frac{\ell_{\pi p}-r_{1}}{2}\|\pi-\pi'\|^{2}.
\end{split}
\end{equation*}
Summing them up, we derive that 
\[
(r_{2}-\ell_{p\pi})\|p(\pi,\bar{\pi},\bar{p})-p(\pi',\bar{\pi},\bar{p})\|^{2}\leq\ell_{\pi p}\|\pi-\pi'\|^{2}+\ell_{\pi p}\|\pi-\pi'\|\|p(\pi,\bar{\pi},\bar{p})-p(\pi',\bar{\pi},\bar{p})\|.
\]
Let $\Lambda=\norm{p(\pi,\bar{\pi},\bar{p})-p(\pi',\bar{\pi},\bar{p})}/\norm{\pi-\pi'}$,
then we have
\begin{equation*}
\Lambda^{2}\leq\frac{\ell_{\pi p}}{r_{2}-\ell_{p\pi}}+\frac{\ell_{\pi p}}{r_{2}-\ell_{p\pi}}\Lambda\overset{(a)}{\Raw}\Lambda\leq\frac{\ell_{\pi p}+r_{2}-\ell_{p\pi}}{r_{2}-\ell_{p\pi}}=:\kappa_{4},
\end{equation*}
where $(a)$ holds by similary argument in~\eqref{eq:appen_lambda1} and~\eqref{eq:appen_lambda2}.

Item~\ref{enu:6} and Item~\ref{enu:7}: By the definition of $\chi$
and the $(r_{2}-\ell_{p\pi})$ strongly conevx of $\chi(\cdot,p,\cdot,\cdot)$
w.r.t. $p$, we have
\begin{equation*}
\begin{split}\chi(\pi,p(\pi,\bar{\pi},\bar{p}),\bar{\pi},\bar{p})-\chi(\pi,p(\pi,\bar{\pi},\bar{p}),\bar{\pi},\bar{p}') & = \frac{r_{2}}{2}\langle\bar{p}'+\bar{p}-2p(\pi,\bar{\pi},\bar{p}),\bar{p}'-\bar{p}\rangle,\\
\chi(\pi,p(\pi,\bar{\pi},\bar{p}'),\bar{\pi},\bar{p}')-\chi(\pi,p(\pi,\bar{\pi},\bar{p}'),\bar{\pi},\bar{p}) & = \frac{r_{2}}{2}\langle v+\bar{p}'-2p(\pi,\bar{\pi},\bar{p}'),\bar{p}-\bar{p}'\rangle,\\
\chi(\pi,p(\pi,\bar{\pi},\bar{p}),\bar{\pi},\bar{p})-\chi(\pi,p(\pi,\bar{\pi},\bar{p}'),\bar{\pi},\bar{p}) & \geq\frac{r_{2}-\ell_{p\pi}}{2}\|p(\pi,\bar{\pi},\bar{p})-p(\pi,\bar{\pi},\bar{p}')\|^{2},\\
\chi(\pi,p(\pi,\bar{\pi},\bar{p}),\bar{\pi},\bar{p}')-\chi(\pi,p(\pi,\bar{\pi},\bar{p}'),\bar{\pi},\bar{p}') & \leq\frac{\ell_{p\pi}-r_{2}}{2}\|p(\pi,\bar{\pi},\bar{p})-p(\pi,\bar{\pi},\bar{p}')\|^{2}.
\end{split}
\end{equation*}
Armed with these inequalities and Cauchy Schwarz inequality, we conclude
\begin{equation*}
(r_{2}-\ell_{p\pi})\|p(\pi,\bar{\pi},\bar{p})-p(\pi,\bar{\pi},\bar{p}')\|^{2}\leq r_{2}\|p(\pi,\bar{\pi},\bar{p}')-p(\pi,\bar{\pi},\bar{p})\|\|\bar{p}-\bar{p}'\|,
\end{equation*}
which complete our proof of Item~\ref{enu:6}. By the definition
of $\varphi_{\pi}$ and similar argument, we can obtain result Item~\ref{enu:7}.
\end{proof}
\begin{lemma}
[Lemma 3 in~\citet{zheng2023universal}]The function $\varphi_{\pi}(p,\cdot,\cdot)$
is continuously differentiable with the gradient $\nabla_{p}\varphi_{\pi}(p,\bar{\pi},\bar{p})=\nabla_{p}\chi(\pi(p,\bar{\pi},\bar{p}),p,\bar{\pi},\bar{p})$
and 
\begin{equation*}
\|\nabla_{p}\varphi_{\pi}(p,\bar{\pi},\bar{p})-\nabla_{p}\varphi_{\pi}(p^{\prime},\bar{\pi},\bar{p})\|\leq L_{\phi\pi}\|p-p'\|,
\end{equation*}
where $L_{\phi\pi}:=\ell_{p\pi}\kappa_{1}+\ell_{p\pi}+r_{2}$.
\end{lemma}
\begin{proof}
Using Danskin's theorem, we know that $\varphi_{\pi}(\cdot,\bar{\pi},\bar{p})$
is differentiable with $\nabla_{p}\varphi_{\pi}(p,\bar{\pi},\bar{p})=\nabla_{p}\chi(\pi(p,\bar{\pi},\bar{p}),p,\bar{\pi},\bar{p})$.
Also, it follows from~\cref{lem:smoothmdp} that
\begin{equation*}
\begin{split} & \|\nabla_{p}\varphi_{\pi}(p,\bar{\pi},\bar{p})-\nabla_{p}\varphi_{\pi}(p^{\prime},\bar{\pi},\bar{p})\|\\
 & =\|\nabla_{p}\chi(\pi(p,\bar{\pi},\bar{p}),\bar{\pi},\bar{p})-\nabla_{p}\chi(\pi(p^{\prime},\bar{\pi},\bar{p}),p^{\prime},\bar{\pi},\bar{p})\|\\
 & \leq\|\nabla_{p}\chi(\pi(p,\bar{\pi},\bar{p}),\bar{\pi},\bar{p})-\nabla_{p}\chi(\pi(p^{\prime},\bar{\pi},\bar{p}),\bar{\pi},\bar{p})\|\\
 & \ \ +\|\nabla_{p}\chi(\pi(p^{\prime},\bar{\pi},\bar{p}),\bar{\pi},\bar{p})-\nabla_{p}\chi(\pi(p^{\prime},\bar{\pi},\bar{p}),p^{\prime},\bar{\pi},\bar{p})\|\\
 & \leq\ell_{p\pi}\|\pi(p^{\prime},\bar{\pi},\bar{p})-\pi(p,\bar{\pi},\bar{p})\|+(\ell_{p\pi}+r_{2})\|p-p'\|\\
 & \aleq(\ell_{p\pi}\kappa_{1}+\ell_{p\pi}+r_{2})\|p-p'\|=L_{\phi\pi}\|p-p'\|,
\end{split}
\end{equation*}
where $(a)$ holds by Item~\ref{enu:1} in~\cref{lem:lem2_un}.
\end{proof}
\begin{lemma}
[Lemma 4 in~\citet{zheng2023universal}]\label{lem:lem4_un}For any
$k\geq0$, the following inequlities holds

\begin{enumerate}[label=\alph*)]

\item \label{item:a)}$\|\pi_{k+1}-\pi(p_{k},\bar{\pi}_{k},\bar{p}_{k})\|\leq\kappa_{6}\|\pi_{k+1}-\pi_{k}\|$

\item \label{item:b)}$\|p_{k+1}-p(\pi_{k+1},\bar{\pi}_{k},\bar{p}_{k})\|\leq\kappa_{7}\|p_{k+1}-p_{k}\|$

\item \label{item:c)}$\|p(\bar{\pi}_{k},\bar{p}_{k})-p_{k}\|\leq\kappa_{8}\|p_{k}-p_{k}^{+}(\bar{\pi}_{k},\bar{p}_{k})\|$

\item \label{item:d)}$\|p_{k+1}-p_{k}^{+}(\bar{\pi}_{k},\bar{p}_{k})\|\leq\ell_{p\pi}\sigma\kappa_{6}\|\pi_{k}-\pi_{k+1}\|$,

\end{enumerate}where $\kappa_{6}=(2\tau r_{1}+1)/(\tau r_{1}-\tau\ell_{\pi p})$,
$\kappa_{7}=(2\sigma r_{2}+1)/(\sigma r_{2}-\sigma\ell_{p\pi})$,
and $\kappa_{8}=(1+\sigma L_{\phi\pi})/(\sigma(r_{2}-\ell_{p\pi}))$.
\end{lemma}
\begin{proof}
Item~\ref{item:a)}: It follows from~\cref{lem:descent_lemma}
that $\chi(\pi,p,\bar{\pi},\bar{p})$ is $(r_{1}-\ell_{\pi p})$ strongly
convex and $(r_{1}+\ell_{\pi p})$ smooth w.r.t. $\pi$ on $\Pi$
for any $p\in\mcal P,\bar{\pi}\in\mbb R^{\abs{\mcal S}\times\abs{\mcal A}},\bar{p}\in\mbb R^{\abs{\mcal S}\times\abs{\mcal A}\times\abs{\mcal S}}$.
Adopting the proof in Theorem 3.1 in~\citet{pang1987posteriori},
we have
\begin{equation*}
\|\pi_{k}-\pi(p_{k},\bar{\pi}_{k},\bar{p}_{k})\|\leq\frac{\tau\ell_{\pi p}+\tau\cdot r_{1}+1}{\tau\cdot r_{1}-\tau\ell_{\pi p}}\|\pi_{k+1}-\pi_{k}\|,
\end{equation*}
which implies that
\[
\|\pi_{k+1}-\pi(p_{k},\bar{\pi}_{k},\bar{p}_{k})\|\leq\|\pi_{k+1}-\pi_{k}\|+\|\pi_{k}-\pi(p_{k},\bar{\pi}_{k},\bar{p}_{k})\|\leq\frac{2\tau\cdot r_{1}+1}{\tau\cdot r_{1}-\tau\ell_{\pi p}}\|\pi_{k+1}-\pi_{k}\|.
\]

Item~\ref{item:b)}: Similar to Item~\ref{item:a)}, Item~\ref{item:b)}
holds by $-\chi(\pi,p,\bar{\pi},\bar{p})$ is $(r_{2}-\ell_{p\pi})$-strongly
convex and $(r_{2}+\ell_{p\pi})$ smooth w.r.t. $p$.

Item~\ref{item:c)}: Similar to Item~\ref{item:a)}, Item~\ref{item:c)}
holds by $-\varphi_{\pi}(p,\bar{\pi},\bar{p})$ is $(r_{2}-\ell_{p\pi})$-strongly
convex and $L_{\phi\pi}$ smooth w.r.t. $p$.

Item~\ref{item:d)}: It follows from the non-expansivity of projection
operator that
\begin{equation*}
\begin{split} & \|p_{k+1}-p_{k}^{+}(\bar{\pi}_{k},\bar{p}_{k})\|\\
=\  & \|\proj_{\mcal P}(p_{k}+\sigma\nabla_{p}F(\pi_{k+1},p_{k},\bar{\pi}_{k},\bar{p}_{k}))-\proj_{\mcal P}(p_{k}+\sigma\nabla_{p}\chi(\pi(p_{k},\bar{\pi}_{k},\bar{p}_{k}),p_{k},\bar{\pi}_{k},\bar{p}_{k}))\|\\
\leq\  & \sigma\|\nabla_{p}F(\pi_{k+1},p_{k},\bar{\pi}_{k},\bar{p}_{k})-\nabla_{p}\chi(\pi(p_{k},\bar{\pi}_{k},\bar{p}_{k}),p_{k},\bar{\pi}_{k},\bar{p}_{k})\|\\
\leq\  & \sigma\ell_{p\pi}\|\pi_{k+1}-\pi(p_{k},\bar{\pi}_{k},\bar{p}_{k})\|\leq\ell_{p\pi}\sigma\kappa_{6}\|\pi_{k}-\pi_{k+1}\|.
\end{split}
\end{equation*}
\end{proof}
\begin{lemma}
[Lemma 5 in~\citet{zheng2023universal}]\label{lem:lem5_un}For any
$k\geq0$, the following inequality holds:
\begin{equation*}
\begin{split} & \chi(\pi_{k},p_{k},\bar{\pi}_{k},\bar{p}_{k})\\
  & \geq\chi(\pi_{k+1},p_{k+1},\bar{\pi}_{k+1},\bar{p}_{k+1})+\left(\frac{1}{\tau}-\frac{\ell_{\pi p}+r_{1}}{2}\right)\|\pi_{k+1}-\pi_{k}\|^{2}\\
 & \ \ +\langle\nabla_{p}F(\pi_{k+1},p_{k},\bar{\pi}_{k},\bar{p}_{k}),p_{k}-p_{k+1}\rangle+\frac{r_{2}-\ell_{p\pi}}{2}\|p_{k+1}-p_{k}\|^{2}\\
 & \ \ +\frac{2-\beta}{2\beta}r_{1}\|\bar{\pi}_{k+1}-\bar{\pi}_{k}\|^{2}+\frac{\mu-2}{2\mu}r_{2}\|\bar{p}_{k+1}-\bar{p}_{k}\|^{2}.
\end{split}
\end{equation*}
\end{lemma}
\begin{proof}
We split the target as follows:
\[
\begin{split} & \chi(\pi_{k},p_{k},\bar{\pi}_{k},\bar{p}_{k})-\chi(\pi_{k+1},p_{k+1},\bar{\pi}_{k+1},\bar{p}_{k+1})\\
 & =\underbrace{\chi(\pi_{k},p_{k},\bar{\pi}_{k},\bar{p}_{k})-\chi(\pi_{k+1},p_{k},\bar{\pi}_{k},\bar{p}_{k})}_{\rone}+\underbrace{\chi(\pi_{k+1},p_{k},\bar{\pi}_{k},\bar{p}_{k})-\chi(\pi_{k+1},p_{k+1},\bar{\pi}_{k},\bar{p}_{k})}_{\text{\ensuremath{\rtwo}}}\\
 & \ \ +\underbrace{\chi(\pi_{k+1},p_{k+1},z^{t},\bar{p}_{k})-\chi(\pi_{k+1},p_{k+1},\bar{\pi}_{k+1},\bar{p}_{k})}_{\rthree}+\underbrace{\chi(\pi_{k+1},p_{k+1},\bar{\pi}_{k+1},\bar{p}_{k})-\chi(\pi_{k+1},p_{k+1},\bar{\pi}_{k+1},\bar{p}_{k+1})}_{\rfour}.
\end{split}
\]

Term $\rone$: By the optimality condition $\langle\pi_{k}-\tau\nabla_{\pi}\chi(\pi_{k},p_{k},\bar{\pi}_{k},\bar{p}_{k})-\pi_{k+1},\pi_{k}-\pi_{k+1}\rangle\leq0$
and $\chi$ is $\ell_{\pi p}+r_{1}$ smooth w.r.t. $\pi$, we have
\[
\begin{split}\chi(\pi_{k},p_{k},\bar{\pi}_{k},\bar{p}_{k})-\chi(\pi_{k+1},p_{k},\bar{\pi}_{k},\bar{p}_{k})  & \geq\langle\nabla_{\pi}\chi(\pi_{k},p_{k},\bar{\pi}_{k},\bar{p}_{k}),\pi_{k}-\pi_{k+1}\rangle-\frac{\ell_{\pi p}+r_{1}}{2}\|\pi_{k+1}-\pi_{k}\|^{2}\\
  & \geq\left(\frac{1}{\tau}-\frac{\ell_{\pi p}+r_{1}}{2}\right)\|\pi_{k+1}-\pi_{k}\|^{2}.
\end{split}
\]

Term $\rtwo$: By $\chi$ is $r_{2}-\ell_{p\pi}$ strongly convex
w.r.t. $p$, we have
\[
\chi(\pi_{k+1},p_{k},\bar{\pi}_{k},\bar{p}_{k})-\chi(\pi_{k+1},p_{k+1},\bar{\pi}_{k},\bar{p}_{k})\geq\langle\nabla_{p}\chi(\pi_{k+1},p_{k},\bar{\pi}_{k},\bar{p}_{k}),p_{k}-p_{k+1}\rangle+\frac{r_{2}-\ell_{p\pi}}{2}\|p_{k+1}-p_{k}\|^{2}.
\]

Term $\rthree$: It follows from $\bar{\pi}_{k+1}=\bar{\pi}_{k}+\beta(\pi_{k+1}-\bar{\pi}_{k})$
that
\[
\chi(\pi_{k+1},p_{k+1},\bar{\pi}_{k},\bar{p}_{k})-\chi(\pi_{k+1},p_{k+1},\bar{\pi}_{k+1},\bar{p}_{k})=\frac{2-\beta}{2\beta}r_{1}\|\bar{\pi}_{k+1}-\bar{\pi}_{k}\|^{2}.
\]

Term $\rfour$: By $\bar{p}_{k+1}=\bar{p}_{k}+\mu(p_{k+1}-\bar{p}_{k})$,
we have
\[
\chi(\pi_{k+1},p_{k+1},\bar{\pi}_{k+1},\bar{p}_{k})-\chi(\pi_{k+1},p_{k+1},\bar{\pi}_{k+1},\bar{p}_{k+1})=\frac{\mu-2}{2\mu}r_{2}\|\bar{p}_{k+1}-\bar{p}_{k}\|^{2}.
\]

Combining all the above bounds leads to the conclusion.
\end{proof}
\begin{lemma}
[Lemma 6 in~\cite{zheng2023universal}]\label{lem:lem6_un}For any
$k\geq0$, the following inequality holds:
\begin{equation*}
\begin{split}\varphi_{\pi}(p_{k+1},\bar{\pi}_{k+1},\bar{p}_{k+1}) & \geq\varphi_{\pi}(p_{k},\bar{\pi}_{k},\bar{p}_{k})+\frac{(2-\mu)r_{2}}{2\mu}\|\bar{p}_{k+1}-\bar{p}_{k}\|^{2}\\
  &\ \ +\frac{r_{1}}{2}\langle\bar{\pi}_{k+1}+\bar{\pi}_{k}-2\pi(p_{k+1},\bar{\pi}_{k+1},\bar{p}_{k}),\bar{\pi}_{k+1}-\bar{\pi}_{k}\rangle\\
  &\ \ +\langle\nabla_{p}\chi(\pi(p_{k},\bar{\pi}_{k},\bar{p}_{k}),p_{k},\bar{\pi}_{k},\bar{p}_{k}),p_{k+1}-p_{k}\rangle-\frac{L_{\phi\pi}}{2}\|p_{k+1}-p_{k}\|^{2}.
\end{split}
\end{equation*}
\end{lemma}
\begin{proof}
We split the target as follows:
\begin{equation*}
\begin{split} & \varphi_{\pi}(p_{k+1},\bar{\pi}_{k+1},\bar{p}_{k+1})-\varphi_{\pi}(p_{k},\bar{\pi}_{k},\bar{p}_{k})\\
 & =\underbrace{\varphi_{\pi}(p_{k+1},\bar{\pi}_{k+1},\bar{p}_{k+1})-\varphi_{\pi}(p_{k+1},\bar{\pi}_{k+1},\bar{p}_{k})}_{\textrm{I}}+\underbrace{\varphi_{\pi}(p_{k+1},\bar{\pi}_{k+1},\bar{p}_{k})-\varphi_{\pi}(p_{k+1},\bar{\pi}_{k},\bar{p}_{k})}_{\textrm{II}}\\
 & \ \ +\underbrace{\varphi_{\pi}(p_{k+1},\bar{\pi}_{k},\bar{p}_{k})-d(y^{t},\bar{\pi}_{k},\bar{p}_{k})}_{\textrm{III}}.
\end{split}
\end{equation*}

Term $\rone$: By the update rule of $\bar{p}_{k+1}$, we have
\[
\begin{aligned}\textrm{I} & =\frac{r_{2}}{2}\left(\|p_{k+1}-\bar{p}_{k}\|^{2}-\|p_{k+1}-\bar{p}_{k+1}\|^{2}\right)=\frac{(2-\mu)r_{2}}{2\mu}\|\bar{p}_{k+1}-\bar{p}_{k}\|^{2}.\end{aligned}
\]

Term $\rtwo:$
\[
\begin{aligned}\textrm{II} & =\chi(\pi(p_{k+1},\bar{\pi}_{k+1},\bar{p}_{k}),p_{k+1},\bar{\pi}_{k+1},\bar{p}_{k})-\chi(\pi(p_{k+1},\bar{\pi}_{k},\bar{p}_{k}),p_{k+1},\bar{\pi}_{k},\bar{p}_{k})\\
  & \geq\chi(\pi(p_{k+1},\bar{\pi}_{k+1},\bar{p}_{k}),p_{k+1},\bar{\pi}_{k+1},\bar{p}_{k})-\chi(\pi(p_{k+1},\bar{\pi}_{k+1},\bar{p}_{k}),p_{k+1},\bar{\pi}_{k},\bar{p}_{k})\\
 & =\frac{r_{1}}{2}\langle\bar{\pi}_{k+1}+\bar{\pi}_{k}-2\pi(p_{k+1},\bar{\pi}_{k+1},\bar{p}_{k}),\bar{\pi}_{k+1}-\bar{\pi}_{k}\rangle.
\end{aligned}
\]

Term $\rthree:$ It follows from $\varphi_{\pi}$ is $L_{\phi\pi}$-smooth
w.r.t. $p$ that
\[
\begin{aligned}\textrm{III} & =\langle\nabla_{p}\chi(\pi(p_{k},\bar{\pi}_{k},\bar{p}_{k}),p_{k},\bar{\pi}_{k},\bar{p}_{k}),p_{k+1}-p_{k}\rangle-\frac{L_{\phi\pi}}{2}\|p_{k+1}-p_{k}\|^{2}.\end{aligned}
\]

Combining the above inequalities finishes the proof.
\end{proof}
\begin{lemma}
[Lemma 7 in~\citet{zheng2023universal}]\label{lem:lem7_un}For all
$k\geq0$, the following inequality holds:
\[
\begin{aligned}q(\bar{\pi}_{k})\geq q(\bar{\pi}_{k+1})+\frac{r_{1}}{2}\langle\bar{\pi}_{k}+\bar{\pi}_{k+1}-2\pi(\bar{\pi}_{k},v(\bar{\pi}_{k+1})),\bar{\pi}_{k}-\bar{\pi}_{k+1}\rangle.\end{aligned}
\]
\end{lemma}
\begin{proof}
From Sion's minimax theorem~\cite{sion1958general}, we have
\[
\begin{aligned}\varphi_{\pi,p,\bar{p}}(\bar{\pi}) & = \max_{\bar{p}}\varphi_{\pi,p}(\bar{\pi},\bar{p})=\max_{\bar{p}}\min_{\pi\in\Pi}\max_{p\in\mathcal{P}}\chi(\pi,p,\bar{\pi},\bar{p})\\
 & = \max_{\bar{p}}\varphi_{p}(\pi(\bar{\pi},\bar{p}),\bar{\pi},\bar{p})=\max_{\bar{p}}\chi(\pi(p(\bar{\pi},\bar{p}),\bar{\pi},\bar{p}),p(\bar{\pi},\bar{p}),\bar{\pi},\bar{p}).
\end{aligned}
\]
Thus, we have
\begin{equation*}
\begin{split} & \varphi_{\pi,p,\bar{p}}(\bar{\pi}_{k})-\varphi_{\pi,p,\bar{p}}(\bar{\pi}_{k+1})\\
 & = \varphi_{p}(\pi(\bar{\pi}_{k},\bar{p}(\bar{\pi}_{k})),\bar{\pi}_{k},\bar{p}(\bar{\pi}_{k}))-\varphi_{p}(\pi(\bar{\pi}_{k+1},\bar{p}(\bar{\pi}_{k+1})),\bar{\pi}_{k+1},\bar{p}(\bar{\pi}_{k+1}))\\
 & \geq\varphi_{p}(\pi(\bar{\pi}_{k},\bar{p}(\bar{\pi}_{k+1})),\bar{\pi}_{k},\bar{p}(\bar{\pi}_{k+1}))-\varphi_{p}(\pi(\bar{\pi}_{k+1},\bar{p}(\bar{\pi}_{k+1})),\bar{\pi}_{k+1},\bar{p}(\bar{\pi}_{k+1}))\\
 & \geq\varphi_{p}(\pi(\bar{\pi}_{k},\bar{p}(\bar{\pi}_{k+1})),\bar{\pi}_{k},\bar{p}(\bar{\pi}_{k+1}))-\varphi_{p}(\pi(\bar{\pi}_{k},\bar{p}(\bar{\pi}_{k+1})),\bar{\pi}_{k+1},\bar{p}(\bar{\pi}_{k+1}))\\
 & \geq\chi(\pi(\bar{\pi}_{k},\bar{p}(\bar{\pi}_{k+1})),p(\pi(\bar{\pi}_{k},\bar{p}(\bar{\pi}_{k+1})),\bar{\pi}_{k+1},\bar{p}(\bar{\pi}_{k+1})),\bar{\pi}_{k},\bar{p}(\bar{\pi}_{k+1}))\\
  &\ \ -\chi(\pi(\bar{\pi}_{k},\bar{p}(\bar{\pi}_{k+1})),p(\pi(\bar{\pi}_{k},\bar{p}(\bar{\pi}_{k+1})),\bar{\pi}_{k+1},\bar{p}(\bar{\pi}_{k+1})),\bar{\pi}_{k+1},\bar{p}(\bar{\pi}_{k+1}))\\
 & = \frac{r_{1}}{2}\langle\bar{\pi}_{k}+\bar{\pi}_{k+1}-2\pi(\bar{\pi}_{k},\bar{p}(\bar{\pi}_{k+1})),\bar{\pi}_{k}-\bar{\pi}_{k+1}\rangle.
\end{split}
\end{equation*}
\end{proof}
\subsection*{Proof of~\cref{prop:prop1_un}}
\begin{proposition}
[Proposition 1 in~\citet{zheng2023universal}]\label{prop:appen_:prop1_un}Define
a Lyapunov function 
\begin{align*}
\Phi(\pi,p,\bar{\pi},\bar{p}) & :=\underbrace{\chi(\pi,p,\bar{\pi},\bar{p})-\varphi_{\pi}(p,\bar{\pi},\bar{p})}_{\text{Primal descent}}+\underbrace{\varphi_{\pi,p}(\bar{\pi},\bar{p})-\varphi_{\pi}(p,\bar{\pi},\bar{p})}_{\text{Dual ascent}}\\
 & \ \ +\underbrace{\varphi_{\pi,p,\bar{p}}(\bar{\pi})-\varphi_{\pi,p}(\bar{\pi},\bar{p})}_{\text{Proximal ascent}}+\underbrace{\varphi_{\pi,p,\bar{p}}(\bar{\pi})-\underline{\chi}}_{\text{Proximal descent}}+\underline{\chi},\nonumber 
\end{align*}
and let parameters satisfy the following properties:
\begin{align*}
r_{1}  & \geq2L,\ \ r_{2}\geq2\lambda L,\ \ \ell_{p\pi}=\lambda\ell_{\pi p}=\lambda L\\
0 & <\tau\leq\min\left\{ \frac{4}{3(L+r_{1})},\frac{1}{6\lambda L}\right\} ,0<\sigma\leq\min\left\{ \frac{2}{3\lambda L\kappa_{6}^{2}},\frac{1}{6L_{\phi\pi}},\frac{1}{5\lambda\sqrt{\lambda+5}L}\right\} ,\\
0 & <\beta\leq\min\left\{ \frac{24r_{1}}{360r_{1}+5r_{1}^{2}\lambda+(2\lambda L+5r_{1})^{2}},\frac{\sigma\lambda^{2}L^{2}}{384r_{1}(\lambda+5)(\lambda+1)^{2}}\right\} ,\\
0 & <\mu\leq\min\left\{ {\frac{(\lambda+5)}{2(\lambda+5)+\lambda^{2}L^{2}},\frac{(\lambda+5)}{\sigma\lambda^{2}L^{2}}},\frac{\sigma\lambda^{2}L^{2}}{64r_{2}(\lambda+5)}\right\} .
\end{align*}
Then for any $k>0$, 
\begin{equation}
\begin{split}\Phi_{k}-\Phi_{k+1} & \geq\frac{r_{1}}{32}\|\pi_{k+1}-\pi_{k}\|^{2}+\frac{r_{2}}{15}\|p_{k}-p_{k}^{+}(\bar{\pi}_{k},\bar{p}_{k})\|^{2}+\frac{r_{1}}{5\beta}\|\bar{\pi}_{k}-\bar{\pi}_{k+1}\|^{2}\\
  &\ \ +\frac{r_{2}}{4\mu}\|\bar{p}_{k}^{+}(\bar{\pi}_{k+1})-\bar{p}_{k}\|^{2}-4r_{1}\beta\|\pi(\bar{\pi}_{k+1},\bar{p}(\bar{\pi}_{k+1}))-\pi(\bar{\pi}_{k+1},\bar{p}_{k}^{+}(\bar{\pi}_{k+1}))\|^{2}
\end{split}
\label{eq:appen_suff_dec}
\end{equation}
where $\kappa_{6}\coloneqq(2\tau\cdot r_{1}+1)/(c(r_{1}-L))$.
\end{proposition}
\begin{proof}
It follows from~\cref{lem:lem5_un},~\cref{lem:lem6_un},~\cref{lem:lem7_un}
that
\begin{equation*}
\begin{split} & \Phi(\pi_{k},p_{k},\bar{\pi}_{k},\bar{p}_{k})\\
  & \geq\Phi(\pi_{k+1},p_{k+1},\bar{\pi}_{k+1},\bar{p}_{k+1})+\left(\frac{1}{\tau}-\frac{\ell_{\pi p}+r_{1}}{2}\right)\|\pi_{k+1}-\pi_{k}\|^{2}\\
 & \ +\left(\frac{r_{2}-\ell_{p\pi}}{2}-L_{\phi\pi}\right)\|p_{k+1}-p_{k}\|^{2}+\frac{(2-\beta)r_{1}}{2\beta}\|\bar{\pi}_{k+1}-\bar{\pi}_{k}\|^{2}\\
 & \ +\frac{(2-\mu)r_{2}}{2\mu}\|\bar{p}_{k+1}-\bar{p}_{k}\|^{2}+\underbrace{\langle\nabla_{p}\chi(\pi_{k+1},p_{k},\bar{\pi}_{k},\bar{p}_{k}),p_{k+1}-p_{k}\rangle}_{\textrm{I}}\\
 & \ +\underbrace{2\langle\nabla_{p}\chi(\pi(p_{k},\bar{\pi}_{k},\bar{p}_{k}),p_{k},\bar{\pi}_{k},\bar{p}_{k})-\nabla_{p}\chi(\pi_{k+1},p_{k},\bar{\pi}_{k},\bar{p}_{k}),p_{k+1}-p_{k}\rangle}_{\textrm{II}}\\
 & \ +\underbrace{2r_{1}\langle\pi(\bar{\pi}_{k},\bar{p}(\bar{\pi}_{k+1}))-\pi(p_{k+1},\bar{\pi}_{k+1},\bar{p}_{k}),\bar{\pi}_{k+1}-\bar{\pi}_{k}\rangle}_{\textrm{III}}.
\end{split}
\end{equation*}

Term $\rone$: using the projection update of $p_{k+1}$, we have
\[
\langle\nabla_{p}\chi(\pi_{k+1},p_{k},\bar{\pi}_{k},\bar{p}_{k}),p_{k+1}-p_{k}\rangle\geq\frac{1}{\sigma}\|p_{k}-p_{k+1}\|^{2}.
\]

Term $\rtwo$: Due to the Lipschitz gradient property and error bound
Item~\ref{item:a)} in~\cref{lem:lem4_un}: 
\begin{equation*}
\begin{split} & 2\langle\nabla_{p}\chi(\pi(p_{k},\bar{\pi}_{k},\bar{p}_{k}),p_{k},\bar{\pi}_{k},\bar{p}_{k})-\nabla_{p}\chi(\pi_{k+1},p_{k},\bar{\pi}_{k},\bar{p}_{k}),p_{k+1}-p_{k}\rangle\\
 & \geq-2\ell_{p\pi}\|\pi(p_{k},\bar{\pi}_{k},\bar{p}_{k})-\pi_{k+1}\|\|p_{k+1}-p_{k}\|\\
 & \geq-\ell_{p\pi}\kappa_{6}^{2}\|p_{k+1}-p_{k}\|^{2}-\ell_{p\pi}\kappa_{6}^{-2}\|\pi(p_{k},\bar{\pi}_{k},\bar{p}_{k})-\pi_{k+1}\|^{2}\\
 & \geq-\ell_{p\pi}\kappa_{6}^{2}\|p_{k+1}-p_{k}\|^{2}-\ell_{p\pi}\|\pi_{k+1}-\pi_{k}\|^{2}.
\end{split}
\end{equation*}

Term $\rthree$: for any $\kappa>0$, note that
\[
\begin{split} & 2r_{1}\langle\pi(\bar{\pi}_{k},\bar{p}(\bar{\pi}_{k+1}))-\pi(p_{k+1},\bar{\pi}_{k+1},\bar{p}_{k}),\bar{\pi}_{k+1}-\bar{\pi}_{k}\rangle\\
 & = 2r_{1}\langle\pi(\bar{\pi}_{k},\bar{p}(\bar{\pi}_{k+1}))-\pi(\bar{\pi}_{k+1},\bar{p}(\bar{\pi}_{k+1})),\bar{\pi}_{k+1}-\bar{\pi}_{k}\rangle+2r_{1}\langle\pi(\bar{\pi}_{k+1},\bar{p}(\bar{\pi}_{k+1}))-\pi(p_{k+1},\bar{\pi}_{k+1},\bar{p}_{k}),\bar{\pi}_{k+1}-\bar{\pi}_{k}\rangle\\
 & \geq-2r_{1}\kappa_{2}\|\bar{\pi}_{k+1}-\bar{\pi}_{k}\|^{2}-\frac{r_{1}}{\kappa}\|\bar{\pi}_{k+1}-\bar{\pi}_{k}\|^{2}-r_{1}\kappa\|\pi(\bar{\pi}_{k+1},\bar{p}(\bar{\pi}_{k+1}))-\pi(p_{k+1},\bar{\pi}_{k+1},\bar{p}_{k})\|^{2},
\end{split}
\]
where the inequality holds by Cauchy-Schwarz inequality, Young's inequality
and Item~\ref{enu:3} in~\cref{lem:lem2_un}. Hence, 
\begin{equation}
\begin{split} & \Phi(\pi_{k},p_{k},\bar{\pi}_{k},\bar{p}_{k})-\Phi(\pi_{k+1},p_{k+1},\bar{\pi}_{k+1},\bar{p}_{k+1})\\
  & \geq\left(\frac{1}{\tau}-\frac{\ell_{\pi p}+r_{1}}{2}-\ell_{p\pi}\right)\|\pi_{k+1}-\pi_{k}\|^{2}+\left(\frac{1}{\sigma}+\frac{r_{2}-\ell_{p\pi}}{2}-L_{\phi\pi}-\ell_{p\pi}\kappa_{6}^{2}\right)\|p_{k+1}-p_{k}\|^{2}\\
 & \ \ +r_{1}\left(\frac{2-\beta}{2\beta}-2\kappa_{2}-\frac{1}{\kappa}\right)\|\bar{\pi}_{k+1}-\bar{\pi}_{k}\|^{2}+\frac{(2-\mu)r_{2}}{2\mu}\|\bar{p}_{k+1}-\bar{p}_{k}\|^{2}\\
 & \ \ -r_{1}\kappa\underbrace{\|\pi(\bar{\pi}_{k+1},\bar{p}(\bar{\pi}_{k+1}))-\pi(p_{k+1},\bar{\pi}_{k+1},\bar{p}_{k})\|^{2}}_{\textrm{IV}}.
\end{split}
\label{eq:appen_descentPhi}
\end{equation}
Now, we give a upper bound of $\rfour$ as follows. Note that
\begin{equation}
\begin{split} & \|\pi(\bar{\pi}_{k+1},\bar{p}(\bar{\pi}_{k+1}))-\pi(p_{k+1},\bar{\pi}_{k+1},\bar{p}_{k})\|^{2}\\
 & \leq 2\|\pi(\bar{\pi}_{k+1},\bar{p}(\bar{\pi}_{k+1}))-\pi(\bar{\pi}_{k+1},\bar{p}_{k}^{+}(\bar{\pi}_{k+1}))\|^{2}+2\|\pi(\bar{\pi}_{k+1},\bar{p}_{k}^{+}(\bar{\pi}_{k+1}))-\pi(p_{k+1},\bar{\pi}_{k+1},\bar{p}_{k})\|^{2}\\
 & \aleq 2\|\pi(\bar{\pi}_{k+1},\bar{p}(\bar{\pi}_{k+1}))-\pi(\bar{\pi}_{k+1},\bar{p}_{k}^{+}(\bar{\pi}_{k+1}))\|^{2}+2\kappa_{1}^{2}\|p_{k+1}-p(\bar{\pi}_{k+1},\bar{p}_{k}^{+}(\bar{\pi}_{k+1}))\|^{2},
\end{split}
\label{eq:appen_IV_up1}
\end{equation}
where $(a)$ holds by Item~\ref{enu:1} in~\cref{lem:lem2_un},
and 
\begin{equation}
\begin{split} & \|p_{k+1}-p(\bar{\pi}_{k+1},\bar{p}_{k}^{+}(\bar{\pi}_{k+1}))\|^{2}\\
 &  \leq 3\|p_{k+1}-p(\bar{\pi}_{k},\bar{p}_{k})\|^{2}+3\|p(\bar{\pi}_{k},\bar{p}_{k})-p(\bar{\pi}_{k+1},\bar{p}_{k})\|^{2}+3\|p(\bar{\pi}_{k+1},\bar{p}_{k})-p(\bar{\pi}_{k+1},\bar{p}_{k}^{+}(\bar{\pi}_{k+1}))\|^{2}\\
& \bleq  6\ell_{p\pi}^{2}\sigma^{2}\kappa_{6}^{2}\|\pi_{k}-\pi_{k+1}\|^{2}+6(\kappa_{8}+1)^{2}\|p_{k}-p_{k}^{+}(\bar{\pi}_{k},\bar{p}_{k})\|^{2}+3\kappa_{3}^{2}\|\bar{\pi}_{k}-\bar{\pi}_{k+1}\|^{2}+3\kappa_{5}^{2}\|\bar{p}_{k}-\bar{p}_{+}^{t}(\bar{\pi}_{k+1})\|^{2},
\end{split}
\label{eq:appen_IV_up2}
\end{equation}
where $(b)$ follows from Item~\ref{enu:4} and Item~\ref{enu:7}
in~\cref{lem:lem2_un} and Item~\ref{item:c)} and Item~\ref{item:d)}
in~\cref{lem:lem4_un}. Putting~\eqref{eq:appen_IV_up1} and~\eqref{eq:appen_IV_up2}
together, we have 
\begin{equation}
\begin{split} & \|\pi(\bar{\pi}_{k+1},\bar{p}(\bar{\pi}_{k+1}))-\pi(p_{k+1},\bar{\pi}_{k+1},\bar{p}_{k})\|^{2}\\
 & \leq2\|\pi(\bar{\pi}_{k+1},\bar{p}(\bar{\pi}_{k+1}))-\pi(\bar{\pi}_{k+1},\bar{p}_{k}^{+}(\bar{\pi}_{k+1}))\|^{2}+12\ell_{p\pi}^{2}\sigma^{2}\kappa_{1}^{2}\kappa_{6}^{2}\|\pi_{k}-\pi_{k+1}\|^{2}\\
 & \ \ +12\kappa_{1}^{2}(\kappa_{8}+1)^{2}\|p_{k}-p_{k}^{+}(\bar{\pi}_{k},\bar{p}_{k})\|^{2}+6\kappa_{1}^{2}\kappa_{3}^{2}\|\bar{\pi}_{k}-\bar{\pi}_{k+1}\|^{2}+6\kappa_{1}^{2}\kappa_{5}^{2}\|\bar{p}_{k}-\bar{p}_{k}^{+}(\bar{\pi}_{k+1})\|^{2}.
\end{split}
\label{eq:appen_IV_up3}
\end{equation}
Combining~\eqref{eq:appen_descentPhi} and~\eqref{eq:appen_IV_up3} and letting
$z_{1}\coloneqq\frac{1}{\tau}-\frac{\ell_{\pi p}+r_{1}}{2}-\ell_{p\pi}$,
$z_{2}\coloneqq\frac{1}{\sigma}+\frac{r_{2}-\ell_{p\pi}}{2}-L_{\phi\pi}-\ell_{p\pi}\kappa_{6}^{2}$,
and $z_{3}\coloneqq r_{1}\left(\frac{2-\beta}{2\beta}-2\kappa_{2}-\frac{1}{\kappa}\right)$,
yields
\begin{equation}
\begin{split} & \Phi(\pi_{k},p_{k},\bar{\pi}_{k},\bar{p}_{k})-\Phi(\pi_{k+1},p_{k+1},\bar{\pi}_{k+1},\bar{p}_{k+1})\\
  & \geq z_{1}\|\pi_{k+1}-\pi_{k}\|^{2}+z_{2}\|p_{k+1}-p_{k}\|^{2}+z_{3}\|\bar{\pi}_{k+1}-\bar{\pi}_{k}\|^{2}+\frac{(2-\mu)r_{2}}{2\mu}\|\bar{p}_{k+1}-\bar{p}_{k}\|^{2}\\
 & \ \ -12r_{1}\kappa\ell_{p\pi}^{2}\sigma^{2}\kappa_{1}^{2}\kappa_{6}^{2}\|\pi_{k}-\pi_{k+1}\|^{2}-12r_{1}\kappa\kappa_{1}^{2}(\kappa_{8}+1)^{2}\|p_{k}-p_{k}^{+}(\bar{\pi}_{k},\bar{p}_{k})\|^{2}-6r_{1}\kappa\kappa_{1}^{2}\kappa_{3}^{2}\|\bar{\pi}_{k}-\bar{\pi}_{k+1}\|^{2}\\
 & \ \ -6r_{1}\kappa\kappa_{1}^{2}\kappa_{5}^{2}\|\bar{p}_{k}-\bar{p}_{k}^{+}(\bar{\pi}_{k+1})\|^{2}-2r_{1}\kappa\|\pi(\bar{\pi}_{k+1},\bar{p}(\bar{\pi}_{k+1}))-\pi(\bar{\pi}_{k+1},\bar{p}_{k}^{+}(\bar{\pi}_{k+1}))\|^{2}\\
 & =\left(z_{1}-12r_{1}\kappa\ell_{p\pi}^{2}\sigma^{2}\kappa_{1}^{2}\kappa_{6}^{2}\right)\|\pi_{k}-\pi_{k+1}\|^{2}+z_{2}\|p_{k+1}-p_{k}\|^{2}+\left(z_{3}-6r_{1}\kappa\kappa_{1}^{2}\kappa_{3}^{2}\right)\|\bar{\pi}_{k+1}-\bar{\pi}_{k}\|^{2}\\
 & \ \ +\frac{(2-\mu)r_{2}}{2\mu}\|\bar{p}_{k+1}-\bar{p}_{k}\|^{2}-6r_{1}\kappa\kappa_{1}^{2}\kappa_{5}^{2}\|\bar{p}_{k}-\bar{p}_{k}^{+}(\bar{\pi}_{k+1})\|^{2}-12r_{1}\kappa\kappa_{1}^{2}(\kappa_{8}+1)^{2}\|p_{k}-p_{k}^{+}(\bar{\pi}_{k},\bar{p}_{k})\|^{2}\\
 & \ \ -2r_{1}\kappa\|\pi(\bar{\pi}_{k+1},\bar{p}(\bar{\pi}_{k+1}))-\pi(\bar{\pi}_{k+1},\bar{p}_{k}^{+}(\bar{\pi}_{k+1}))\|^{2}.
\end{split}
\label{eq:appen_phi_descent2}
\end{equation}
Note that, by the Item~\ref{item:d)} in~\cref{lem:lem4_un},
we have
\begin{equation}
\begin{split}\|p_{k+1}-p_{k}\|^{2}\geq\  & \frac{1}{2}\|p_{k}-p_{k}^{+}(\bar{\pi}_{k},\bar{p}_{k})\|^{2}-\|p_{k+1}-p_{k}^{+}(\bar{\pi}_{k},\bar{p}_{k})\|^{2}\\
\geq\  & \frac{1}{2}\|p_{k}-p_{k}^{+}(\bar{\pi}_{k},\bar{p}_{k})\|^{2}-\ell_{p\pi}^{2}\sigma^{2}\kappa_{6}^{2}\|\pi_{k}-\pi_{k+1}\|^{2}.
\end{split}
\label{eq:appen_mid2}
\end{equation}
Similarly, we provide a lower bound for $\|\bar{p}_{k+1}-\bar{p}_{k}\|^{2}$
as follows:
\begin{equation}
\begin{split}\|\bar{p}_{k+1}-\bar{p}_{k}\|^{2}  & \geq\frac{1}{2}\|\bar{p}_{k}-\bar{p}_{k}^{+}(\bar{\pi}_{k+1})\|^{2}-\|\bar{p}_{k+1}-\bar{p}_{k}^{+}(\bar{\pi}_{k+1})\|^{2}\\
  & \geq\frac{1}{2}\|\bar{p}_{k}-\bar{p}_{k}^{+}(\bar{\pi}_{k+1})\|^{2}-\mu^{2}\Big(4\ell_{p\pi}^{2}\sigma^{2}\kappa_{6}^{2}\|\pi_{k}-\pi_{k+1}\|^{2}+4(\kappa_{8}+1)^{2}\|p_{k}-p_{k}^{+}(\bar{\pi}_{k},\bar{p}_{k})\|^{2}\\
 & \ \ +2\kappa_{3}^{2}\|\bar{\pi}_{k}-\bar{\pi}_{k+1}\|^{2}\Big).
\end{split}
\label{eq:appen_mid2-1}
\end{equation}
Recalling that $\bar{p}_{k}^{+}(\bar{\pi}_{k+1})=\bar{p}_{k}+\mu(p(\bar{\pi}_{k+1},\bar{p}_{k})-\bar{p}_{k})$,
the last inequality can be obtained by using Item~\ref{enu:4} in
Lemma~\ref{lem:lem2_un} and Item~\ref{item:c)} and Item~\ref{item:d)}
in~\cref{lem:lem5_un}, i.e.,
\[
\begin{split} & \|\bar{p}_{k+1}-\bar{p}_{k}^{+}(\bar{\pi}_{k+1})\|^{2}=\mu^{2}\|p_{k+1}-p(\bar{\pi}_{k+1},\bar{p}_{k})\|^{2}\\
& \leq\   \mu^{2}\left(2\|p_{k+1}-p(\bar{\pi}_{k},\bar{p}_{k})\|^{2}+2\|p(\bar{\pi}_{k},\bar{p}_{k})-p(\bar{\pi}_{k+1},\bar{p}_{k})\|^{2}\right)\\
 &  \leq\ \mu^{2}\left(4\ell_{p\pi}^{2}\sigma^{2}\kappa_{6}^{2}\|\pi_{k}-\pi_{k+1}\|^{2}+4(\kappa_{8}+1)^{2}\|p_{k}-p_{k}^{+}(\bar{\pi}_{k},\bar{p}_{k})\|^{2}+2\kappa_{3}^{2}\|\bar{\pi}_{k}-\bar{\pi}_{k+1}\|^{2}\right).
\end{split}
\]
Substituting~\eqref{eq:appen_mid2} and~\eqref{eq:appen_mid2-1} to~\eqref{eq:appen_phi_descent2}
yields
\begin{equation*}
\begin{split} & \Phi(\pi_{k},p_{k},\bar{\pi}_{k},\bar{p}_{k})-\Phi(\pi_{k+1},p_{k+1},\bar{\pi}_{k+1},\bar{p}_{k+1})\\
  & \geq\underbrace{\left(z_{1}-(12r_{1}\kappa\kappa_{1}^{2}+s_{2}+2\mu(2-\mu)r_{2})\ell_{p\pi}^{2}\sigma^{2}\kappa_{6}^{2}\right)}_{\rthree}\|\pi_{k+1}-\pi_{k}\|^{2}\\
 & \ \ +\underbrace{\left(\frac{z_{2}}{2}-(12r_{1}\kappa\kappa_{1}^{2}+2\mu(2-\mu)r_{2})(1+\kappa_{8})^{2}\right)}_{\rtwo}\|p_{k}-p_{k}^{+}(\bar{\pi}_{k},\bar{p}_{k})\|^{2}\\
 & \ \ +\underbrace{\left(z_{3}-(\mu(2-\mu)r_{2}+6r_{1}\kappa\kappa_{1}^{2})\kappa_{3}^{2}\right)}_{\rone}\|\bar{\pi}_{k}-\bar{\pi}_{k+1}\|^{2}+\left(\frac{(2-\mu)r_{2}}{4\mu}-6r_{1}\kappa\kappa_{1}^{2}\kappa_{5}^{2}\right)\|\bar{p}_{k}-\bar{p}_{k}^{+}(\bar{\pi}_{k+1})\|^{2}\\
 & \ \ -2r_{1}\kappa\|\pi(\bar{\pi}_{k+1},\bar{p}(\bar{\pi}_{k+1}))-\pi(\bar{\pi}_{k+1},\bar{p}_{k}^{+}(\bar{\pi}_{k+1}))\|^{2}.
\end{split}
\end{equation*}

Now, we give the upper bound of $\rone,\rtwo,\rthree$, separately.

Term $\rone$: Based on the parameters setting, we have
\begin{align*}
r_{1}\geq2\ell_{\pi p}\Rightarrow\kappa_{1}=\frac{\ell_{p\pi}}{r_{1}-\ell_{\pi p}}+1\leq\frac{\ell_{p\pi}}{\ell_{\pi p}}+1=\lambda+1, & \text{and }\kappa_{2}=\frac{r_{1}}{r_{1}-\ell_{\pi p}}\leq2,\\
r_{2}\geq2\ell_{p\pi}\Rightarrow\kappa_{3}=\frac{r_{1}\kappa_{1}}{r_{2}-\ell_{p\pi}}+\frac{\kappa_{2}}{\kappa_{1}}\leq\frac{r_{1}(\lambda+1)}{\ell_{p\pi}}+2, & \text{and }\kappa_{5}=\frac{r_{2}}{r_{2}-\ell_{p\pi}}\leq2.
\end{align*}
Then we obtain:
\begin{equation*}
\begin{split}\frac{(2-\mu)r_{2}}{4\mu}-6r_{1}\kappa\kappa_{1}^{2}\kappa_{5}^{2} & \ageq\frac{r_{2}}{2\mu}-\frac{r_{2}}{4}-48(\lambda+1)^{2}r_{1}\beta\bgeq\frac{r_{2}}{\mu}\left(\frac{1}{2}-\frac{\mu}{4}-\frac{\ell_{p\pi}^{2}\sigma\mu}{8(\lambda+5)}\right)\cgeq\frac{r_{2}}{4\mu},\end{split}
\end{equation*}
where $(a)$ holds by $\kappa_{2},\kappa_{5}\leq2,\kappa=2\beta$
and $\kappa_{1}\leq\lambda+1$, $(b)$ holds by the definition of
$\beta$ and $\ell_{p\pi}=\lambda L$ and $(c)$ follows from the
defintion of $\mu$ and $\mu\leq\frac{1}{2}$. Similarly, we have
\begin{equation*}
\begin{split} & \mu(2-\mu)r_{2}+6r_{1}\kappa\kappa_{1}^{2}\aleq2r_{2}\mu+12r_{1}\beta(\lambda+1)^{2}\bleq\frac{\sigma\ell_{p\pi}^{2}}{16(\lambda+5)}\\
\Rightarrow\rone & =z_{3}-\left(\mu(2-\mu)r_{2}+6r_{1}\kappa\kappa_{1}^{2}\right)\kappa_{3}^{2}=r_{1}\left(\frac{2-\beta}{2\beta}-2\kappa_{2}-\frac{1}{\kappa}\right)-\left(\mu(2-\mu)r_{2}+6r_{1}\kappa\kappa_{1}^{2}\right)\kappa_{3}^{2}\\
 & \cgeq\ r_{1}\left(\frac{1}{\beta}-\frac{1}{2}-4-\frac{1}{2\beta}\right)-\frac{\ell_{p\pi}^{2}}{16(\lambda+5)}\left(\frac{r_{1}^{2}(\lambda+1)^{2}}{\ell_{p\pi}^{2}}+4+\frac{4r_{1}(\lambda+1)}{\ell_{p\pi}}\right)\\
 & \dgeq\ \frac{r_{1}}{\beta}\left(\frac{1}{2}-\left(\frac{9}{2}+\frac{r_{1}(\lambda+1)}{16}+\frac{\ell_{p\pi}^{2}}{20r_{1}}+\frac{\ell_{p\pi}}{4}\right)\beta\right)\fgeq\frac{r_{1}}{5\beta},
\end{split}
\end{equation*}
where $(a)$ holds by $\kappa=2\beta$ and $\kappa_{1}\leq\lambda+1$,
$(b)$ follows from $\mu\leq(\sigma\lambda^{2}L^{2})/(64r_{2}(\lambda+5))$,
$(c)$ is by $\kappa_{2}\leq2$, $(b)$ and the definition of $\kappa_{3}$
and $(d),(f)$ follow from the definition of $\beta$.

Term $\rtwo$: Based on the parameters setting, we have
\begin{align*}
\tau^{-1}\geq\max\left\{ \frac{3}{4}(\ell_{\pi p}+r_{1}),6\ell_{p\pi}\right\} \Rightarrow z_{1}\geq\frac{1}{6\tau}, & \text{and }\kappa_{6}=\frac{2r_{1}+\frac{1}{\tau}}{r_{1}-\ell_{\pi p}}\geq2+\frac{1}{\tau\cdot r_{1}}\geq2+\frac{\frac{3}{4}(\ell_{\pi p}+r_{1})}{r_{1}}\geq\frac{11}{4}\\
\frac{1}{\sigma}\geq\max\left\{ \frac{3}{2}\ell_{p\pi}\kappa_{6}^{2},6L_{\phi\pi},5\sqrt{\lambda+5}\ell_{p\pi}\right\} \Rightarrow z_{2}  & \geq\frac{1}{6\sigma}+\frac{r_{2}-\ell_{p\pi}}{2}\text{and }\kappa_{8}=\frac{\frac{1}{\sigma}+L_{\phi\pi}}{r_{2}-\ell_{p\pi}}\leq\frac{2}{\sigma r_{2}}+\lambda+4
\end{align*}
Then we have: 
\begin{equation*}
\begin{split}\rtwo & =\frac{z_{2}}{2}-\left(12r_{1}\kappa\kappa_{1}^{2}+2\mu(2-\mu)r_{2}\right)(1+\kappa_{8})^{2}\\
  & \geq\frac{1}{12\sigma}+\frac{r_{2}-\ell_{p\pi}}{4}-\frac{\sigma\ell_{p\pi}^{2}}{8(\lambda+5)}\left(\frac{4}{\sigma^{2}r_{2}^{2}}+(\lambda+5)^{2}+\frac{4(\lambda+5)}{\sigma r_{2}}\right)\\
  & \geq\frac{1}{12\sigma}+\frac{\ell_{p\pi}}{4}-\frac{1}{8(\lambda+5)\sigma}-\frac{(\lambda+5)\sigma\ell_{p\pi}^{2}}{8}-\frac{\ell_{p\pi}}{4}\\
  & \geq\frac{1}{60\sigma}-\frac{(\lambda+5)\sigma\ell_{p\pi}^{2}}{8}\geq\frac{1}{90\sigma}\geq\frac{r_{2}}{15}
\end{split}
\end{equation*}

Term $\rthree$: 
\begin{equation*}
\begin{split}\rthree & =z_{1}-(12r_{1}\kappa\kappa_{1}^{2}+s_{2}+2\mu(2-\mu)r_{2})\ell_{p\pi}^{2}\sigma^{2}\kappa_{6}^{2}\geq\frac{1}{6\tau}-\frac{\sigma^{3}\ell_{p\pi}^{4}\kappa_{6}^{2}}{8(\lambda+5)}-\frac{2\ell_{p\pi}}{3}\\
  & \geq\frac{1}{6\tau}-\frac{13\ell_{p\pi}}{2500}-\frac{2\ell_{p\pi}}{3}\geq\frac{1}{6\tau}-\frac{7\ell_{p\pi}}{10}\geq\frac{1}{24c}\geq\frac{r_{1}}{32},
\end{split}
\end{equation*}
where the first inequality is due to $2/(3\sigma)\geq\ell_{p\pi}\kappa_{6}^{2}$
and 
\[
\begin{aligned}s_{2} & \leq\frac{1}{\sigma}+\frac{r_{2}-\ell_{p\pi}}{2}-\ell_{p\pi}\kappa_{6}^{2}-L_{\phi\pi}=\frac{1}{\sigma}-\frac{r_{2}+\ell_{p\pi}}{2}-\ell_{p\pi}\kappa_{6}^{2}-(\kappa_{1}+1)\ell_{p\pi}\leq\frac{1}{\sigma}.\end{aligned}
\]
Putting together all the pieces, we get 
\[
\begin{aligned} & \Phi(\pi_{k},p_{k},\bar{\pi}_{k},\bar{p}_{k})-\Phi(\pi_{k+1},p_{k+1},\bar{\pi}_{k+1},\bar{p}_{k+1})\\
 & \geq\frac{r_{1}}{32}\|\pi_{k+1}-\pi_{k}\|^{2}+\frac{r_{2}}{15}\|p_{k}-p_{k}^{+}(\bar{\pi}_{k},\bar{p}_{k})\|^{2}+\frac{r_{1}}{5\beta}\|\bar{\pi}_{k}-\bar{\pi}_{k+1}\|^{2}+\frac{r_{2}}{4\mu}\|\bar{p}_{k}^{+}(\bar{\pi}_{k+1})-\bar{p}_{k}\|^{2}\\
  &\ \ -4r_{1}\beta\|\pi(\bar{\pi}_{k+1},\bar{p}(\bar{\pi}_{k+1}))-\pi(\bar{\pi}_{k+1},\bar{p}_{k}^{+}(\bar{\pi}_{k+1}))\|^{2}.
\end{aligned}
\]
\end{proof}
\subsection*{Proof of~\cref{prop:upperbound_neg}}
\begin{proposition}
\label{prop:appen_:prop2}Under the setting of~\cref{prop:prop1_un},
then we have
\begin{equation}
\norm{\pi(\bar{\pi}_{k+1},\bar{p}_{k}^{+}(\bar{\pi}_{k+1}))-\pi(\bar{\pi}_{k+1},\bar{p}(\bar{\pi}_{k+1}))}^{2}\leq\omega_{0}\norm{\bar{p}_{k}^{+}(\bar{\pi}_{k+1})-\bar{p}_{k}},\label{eq:appen_prop_ineq1}
\end{equation}
where $\omega_{0}:=\frac{2}{(r_{1}-L)\tau}\left(\frac{r_{2}(1-\mu)}{\mu}+\frac{r_{2}^{2}}{r_{2}-\lambda L}\right)$.
Moreover, we have
\begin{equation}
\|\pi^{*}(\bar{\pi})-\pi(\bar{\pi},\bar{p})\|^{2}\leq\omega_{1}\|\bar{p}-p(\bar{\pi},\bar{p})\|,\label{eq:appen_prop_ineq2}
\end{equation}
where $\omega_{1}:=\frac{2r_{2}}{\tau(r_{1}-\ell_{\pi p})}$.
\end{proposition}
\begin{proof}
Inequality~\eqref{eq:appen_prop_ineq1}: Since $\varphi_{p}(\pi,\cdot,\cdot)$
is $(r_{1}-\ell_{\pi p})$-strongly convex, we have
\begin{equation}
\begin{split} & \max_{\bar{p}}\varphi_{p}(\pi(\bar{\pi}_{k+1},\bar{p}_{k}^{+}(\bar{\pi}_{k+1})),\bar{\pi}_{k+1},\bar{p})-\varphi_{p}(\pi(\bar{\pi}_{k+1},\bar{p}(\bar{\pi}_{k+1})),\bar{\pi}_{k+1},\bar{p}(\bar{\pi}_{k+1}))\\
\geq\  & \varphi_{p}(\pi(\bar{\pi}_{k+1},\bar{p}_{k}^{+}(\bar{\pi}_{k+1})),\bar{\pi}_{k+1},\bar{p}(\bar{\pi}_{k+1}))-\varphi_{p}(\pi(\bar{\pi}_{k+1},\bar{p}(\bar{\pi}_{k+1})),\bar{\pi}_{k+1},\bar{p}(\bar{\pi}_{k+1}))\\
\geq\  & \frac{r_{1}-\ell_{\pi p}}{2}\|\pi(\bar{\pi}_{k+1},\bar{p}_{k}^{+}(\bar{\pi}_{k+1}))-\pi(\bar{\pi}_{k+1},\bar{p}(\bar{\pi}_{k+1}))\|^{2}.
\end{split}
\label{eq:appen_key1}
\end{equation}
 Since $J_{\rho}(\cdot,p)$ is $\ell_{p\pi}$-smooth, namely $-J_{\rho}(\cdot,p)$
is $\ell_{p\pi}$-weakly convex w.r.t. $p$ and $-J_{\rho}(\cdot,p)$
is gradient dominace (\cref{lem:gradient_dominance}).
By~\cref{lem:p_grad_moreau}, there exists a constant $C_{\varphi p}>0$
such that
\begin{equation}
\begin{split} & \max_{\bar{p}}\varphi_{p}(\pi(\bar{\pi}_{k+1},\bar{p}_{k}^{+}(\bar{\pi}_{k+1})),\bar{\pi}_{k+1},\bar{p})-\varphi_{p}(\pi(\bar{\pi}_{k+1},\bar{p}(\bar{\pi}_{k+1})),\bar{\pi}_{k+1},\bar{p}(\bar{\pi}_{k+1}))\\
\leq\  & \max_{\bar{p}}\varphi_{p}(\pi(\bar{\pi}_{k+1},\bar{p}_{k}^{+}(\bar{\pi}_{k+1})),\bar{\pi}_{k+1},\bar{p})-\varphi_{p}(\pi(\bar{\pi}_{k+1},\bar{p}_{k}^{+}(\bar{\pi}_{k+1})),\bar{\pi}_{k+1},\bar{p}_{k}^{+}(\bar{\pi}_{k+1}))\\
\leq\  & C_{\varphi p}\|\nabla_{\bar{p}}\varphi_{p}(\pi(\bar{\pi}_{k+1},\bar{p}_{k}^{+}(\bar{\pi}_{k+1})),\bar{\pi}_{k+1},\bar{p}_{k}^{+}(\bar{\pi}_{k+1}))\|.
\end{split}
\label{eq:appen_key2}
\end{equation}
Next, we further bound the right-hand part.
\begin{equation}
\begin{split} & \|\nabla_{\bar{p}}\varphi_{p}(\pi(\bar{\pi}_{k+1},\bar{p}_{k}^{+}(\bar{\pi}_{k+1})),\bar{\pi}_{k+1},\bar{p}_{k}^{+}(\bar{\pi}_{k+1}))\|\\
& =\  r_{2}\|\bar{p}_{k}^{+}(\bar{\pi}_{k+1})-p(\pi(\bar{\pi}_{k+1},\bar{p}_{k}^{+}(\bar{\pi}_{k+1})),\bar{\pi}_{k+1},\bar{p}_{k}^{+}(\bar{\pi}_{k+1}))\|\\
& =\  r_{2}\|(1-\mu)\bar{p}_{k}+\mu p(\bar{\pi}_{k+1},\bar{p}_{k})-p(\bar{\pi}_{k+1},\bar{p}_{k}^{+}(\bar{\pi}_{k+1}))\|\\
&\leq\   r_{2}(1-\mu)\|\bar{p}_{k}-p(\bar{\pi}_{k+1},\bar{p}_{k})\|+r_{2}\|p(\bar{\pi}_{k+1},\bar{p}_{k})-p(\bar{\pi}_{k+1},\bar{p}_{k}^{+}(\bar{\pi}_{k+1}))\|\\
& \aleq\  r_{2}\left(\frac{1-\mu}{\mu}+\kappa_{5}\right)\|\bar{p}_{k}^{+}(\bar{\pi}_{k+1})-\bar{p}_{k}\|,
\end{split}
\label{eq:appen_key3}
\end{equation}
where $(a)$ holds by Item~\ref{enu:6} in~\cref{lem:lem2_un}
and the definition of $\bar{p}_{k}^{+}(\bar{\pi}_{k+1})$. Combing~\eqref{eq:appen_key1},~\eqref{eq:appen_key2},
and~\eqref{eq:appen_key3}, we have
\begin{equation*}
\|\pi(\bar{\pi}_{k+1},\bar{p}_{k}^{+}(\bar{\pi}_{k+1}))-\pi(\bar{\pi}_{k+1},\bar{p}(\bar{\pi}_{k+1}))\|^{2}\leq\frac{2}{(r_{1}-\ell_{\pi p})C_{\varphi p}}\left(\frac{r_{2}(1-\mu)}{\mu}+r_{2}\kappa_{5}\right)\|\bar{p}_{k}^{+}(\bar{\pi}_{k+1})-\bar{p}_{k}\|.
\end{equation*}

Inequality~\eqref{eq:appen_prop_ineq2}: Since $\phi_{p}(\pi;\bar{\pi})$
is $(r_{1}-\ell_{\pi p})$-strongly convex w.r.t. $\pi$, we have
\begin{equation*}
\psi_{p}(\pi(\bar{\pi},\bar{p});\bar{\pi})-\psi_{p}(\pi^{*}(\bar{\pi});\bar{\pi})\geq\frac{r_{1}-\ell_{\pi p}}{2}\norm{\pi^{*}(\bar{\pi})-\pi(\bar{\pi},\bar{p})}^{2}.
\end{equation*}
On the other hand, we have
\begin{equation*}
\begin{split} & \psi_{p}(\pi(\bar{\pi},\bar{p});\bar{\pi})-\psi_{p}(\pi^{*}(\bar{\pi});\bar{\pi})\\
 & =\psi_{p}(\pi(\bar{\pi},\bar{p});\bar{\pi})-\brbra{\max_{p\in\mcal P}J_{\rho}(\pi^{*}(\bar{\pi}),p)+\frac{r_{1}}{2}\norm{\pi^{*}(\bar{\pi})-\bar{\pi}}^{2}}\\
 & \aleq\psi_{p}(\pi(\bar{\pi},\bar{p});\bar{\pi})-\brbra{J_{\rho}(\pi(\bar{\pi},\bar{p}),p(\bar{\pi},p))+\frac{r_{1}}{2}\norm{\pi(\bar{\pi},\bar{p})-\bar{\pi}}^{2}}\\
 & =\max_{p\in\mcal P}J_{\rho}(\pi(\bar{\pi},\bar{p}),p)-J_{\rho}(\pi(\bar{\pi},\bar{p}),p(\bar{\pi},p)),
\end{split}
\end{equation*}
where $(a)$ holds by the definition of $\pi(\bar{\pi},\bar{p})$
and $\pi^{*}(\bar{\pi})$. It follows from~\cref{lem:gradient_dominance}
that
\begin{equation*}
\max_{p\in\mcal P}J_{\rho}(\pi(\bar{\pi},\bar{p}),p)-J_{\rho}(\pi(\bar{\pi},\bar{p}),p(\bar{\pi},\bar{p}))\leq\bar{D}_{p}\dist\brbra{\nabla_{p}J_{\rho}(\pi,p(\bar{\pi},\bar{p}))-\mcal N_{\mcal P}(p(\bar{\pi},\bar{p}))}\aleq\bar{D}_{p}\norm{r_{2}(\bar{p}-p(\bar{\pi},\bar{p}))},
\end{equation*}
where $(a)$ holds by $p(\bar{\pi},\bar{p})\in\argmax_{p\in\mcal P}J_{\rho}(\pi(\bar{\pi},\bar{p}),p)-\frac{r_{1}}{2}\norm{p-\bar{p}}^{2}$.
\end{proof}

\begin{lemma}
    For Lyapunov function $\Phi$, we have
    \begin{equation}\label{eq:upper_bound_Lya}
\Phi(\pi_{0},p_{0},\bar{\pi}_{0},\bar{p}_{0})-\underline{\chi}\leq L_{\pi}\norm{\pi_{0}-\pi^{*}}+L_{p}\norm{p_{0}-p^{*}}+L_{\pi}\norm{\pi(\bar{\pi}_{0})-\pi(p_{0},\bar{\pi}_{0},\bar{p}_{0})}+L_{p}\norm{p(\bar{\pi}_{0})-p_{0}}.
\end{equation}
\end{lemma}

\begin{proof}
Since $\Phi(\pi_{0},p_{0},\bar{\pi}_{0},\bar{p}_{0})-\underline{\chi}=\chi(\pi_{0},p_{0},\bar{\pi}_{0},\bar{p}_{0})-2\varphi_{\pi}(p_{0},\bar{\pi}_{0},\bar{p}_{0})+2\varphi_{\pi,p,\bar{p}}(\bar{\pi}_{0})-\chi(\pi^{*},p^{*},\bar{\pi}^{*},\bar{p}^{*})$,
then we have
\begin{equation*}
\begin{split} & \Phi(\pi_{0},p_{0},\bar{\pi}_{0},\bar{p}_{0})-\underline{\chi}\\
 & =\chi(\pi_{0},p_{0},\bar{\pi}_{0},\bar{p}_{0})-\chi(\pi^{*},p^{*},\bar{\pi}^{*},\bar{p}^{*})\\
 & \ \ +2\brbra{\varphi_{\pi,p,\bar{p}}(\bar{\pi}_{0})-\varphi_{\pi}(p_{0},\bar{\pi}_{0},\bar{p}_{0})}\\
 & =J_{\rho}(\pi_{0},p_{0})-J_{\rho}(\pi^{*},p^{*})+\frac{r_{1}}{2}\brbra{\norm{\pi_{0}-\bar{\pi}_{0}}^{2}-\norm{\pi^{*}-\bar{\pi}^{*}}^{2}}\\
 & \ \ -\frac{r_{2}}{2}\brbra{\norm{p_{0}-\bar{p}_{0}}^{2}-\norm{p^{*}-\bar{p}^{*}}^{2}}\\
 & \ \ +2\Brbra{\chi\brbra{\pi(\bar{\pi}_{0}),p(\bar{\pi}_{0}),\bar{\pi}_{0},\bar{p}(\bar{\pi}_{0})}-\chi\brbra{\pi(p_{0},\bar{\pi}_{0},\bar{p}_{0}),p_{0},\bar{\pi}_{0},\bar{p}_{0}}}
\end{split}
\end{equation*}
It is easy to know that $\pi^{*}=\bar{\pi}^{*},p^{*}=\bar{p}^{*}$
since $\bar{\pi}^{*}=\argmin_{\bar{\pi}}J_{\rho}(\pi^{*},p^{*})+\frac{r_{1}}{2}\norm{\pi^{*}-\bar{\pi}}^{2}$
and $\bar{p}^{*}=\argmax_{\bar{p}}J_{\rho}(\pi^{*},p^{*})-\frac{r_{2}}{2}\norm{p^{*}-\bar{p}^{*}}^{2}$.
Hence, we have
\begin{equation}
\begin{split} & \Phi(\pi_{0},p_{0},\bar{\pi}_{0},\bar{p}_{0})-\underline{\chi}\\
 & \leq L_{\pi}\norm{\pi_{0}-\pi^{*}}+L_{p}\norm{p_{0}-p^{*}}+\frac{r_{1}}{2}\norm{\pi_{0}-\bar{\pi}_{0}}^{2}\\
 & \ \ +2\Brbra{\chi\brbra{\pi(\bar{\pi}_{0}),p(\bar{\pi}_{0}),\bar{\pi}_{0},\bar{p}(\bar{\pi}_{0})}-\chi\brbra{\pi(p_{0},\bar{\pi}_{0},\bar{p}_{0}),p_{0},\bar{\pi}_{0},\bar{p}_{0}}}
\end{split}
\label{eq:append_upper1}
\end{equation}
For term $\chi\brbra{\pi(\bar{\pi}_{0}),p(\bar{\pi}_{0}),\bar{\pi}_{0},\bar{p}(\bar{\pi}_{0})}-\chi\brbra{\pi(p_{0},\bar{\pi}_{0},\bar{p}_{0}),p_{0},\bar{\pi}_{0},\bar{p}_{0}}$, we give upper bound as follows:
\begin{equation}
\begin{split} & \chi\brbra{\pi(\bar{\pi}_{0}),p(\bar{\pi}_{0}),\bar{\pi}_{0},\bar{p}(\bar{\pi}_{0})}-\chi\brbra{\pi(p_{0},\bar{\pi}_{0},\bar{p}_{0}),p_{0},\bar{\pi}_{0},\bar{p}(\bar{\pi}_{0})}\\
 & =\chi\brbra{\pi(\bar{\pi}_{0}),p(\bar{\pi}_{0}),\bar{\pi}_{0},\bar{p}(\bar{\pi}_{0})}-\chi\brbra{\pi(p_{0},\bar{\pi}_{0},\bar{p}_{0}),p(\bar{\pi}_{0}),\bar{\pi}_{0},\bar{p}(\bar{\pi}_{0})}\\
 & \ \ +\chi\brbra{\pi(p_{0},\bar{\pi}_{0},\bar{p}_{0}),p(\bar{\pi}_{0}),\bar{\pi}_{0},\bar{p}(\bar{\pi}_{0})}-\chi\brbra{\pi(p_{0},\bar{\pi}_{0},\bar{p}_{0}),p_{0},\bar{\pi}_{0},\bar{p}(\bar{\pi}_{0})}\\
 & \leq L_{\pi}\norm{\pi(\bar{\pi}_{0})-\pi(p_{0},\bar{\pi}_{0},\bar{p}_{0})}+L_{p}\norm{p(\bar{\pi}_{0})-p_{0}},
\end{split}
\label{eq:append_upper2}
\end{equation}
and 
\begin{equation}
\begin{split} & \chi\brbra{\pi(p_{0},\bar{\pi}_{0},\bar{p}_{0}),p_{0},\bar{\pi}_{0},\bar{p}(\bar{\pi}_{0})}-\chi\brbra{\pi(p_{0},\bar{\pi}_{0},\bar{p}_{0}),p_{0},\bar{\pi}_{0},\bar{p}_{0}}\\
 & =-\frac{r_{2}}{2}\norm{p_{0}-\bar{p}(\bar{\pi}_{0})}^{2}+\frac{r_{2}}{2}\norm{p_{0}-\bar{p}_{0}}^{2}.
\end{split}
\label{eq:append_upper3}
\end{equation}
Without loss of generality, taking $\bar{p}_{0}=p_{0},\bar{\pi}_{0}=\pi_{0}$
and combining~\eqref{eq:append_upper1},~\eqref{eq:append_upper2}
and~\eqref{eq:append_upper3} yields
\begin{equation*}
\Phi(\pi_{0},p_{0},\bar{\pi}_{0},\bar{p}_{0})-\underline{\chi}\leq L_{\pi}\norm{\pi_{0}-\pi^{*}}+L_{p}\norm{p_{0}-p^{*}}+L_{\pi}\norm{\pi(\bar{\pi}_{0})-\pi(p_{0},\bar{\pi}_{0},\bar{p}_{0})}+L_{p}\norm{p(\bar{\pi}_{0})-p_{0}}.
\end{equation*}

\end{proof}

\begin{lemma}
[Lemma 8 in~\citet{zheng2023universal}]\label{lem:lem8_un}Let $\varepsilon\geq0$.
Suppose that 
\[
\max\left\{ \frac{\|\pi_{k}-\pi_{k+1}\|}{\tau},\frac{\|p_{k}-p_{k}^{+}(\bar{\pi}_{k},\bar{p}_{k})\|}{\sigma},\frac{\|\bar{\pi}_{k}-\bar{\pi}_{k+1}\|}{\beta},\frac{\|\bar{p}_{k}-\bar{p}_{k+1}\|}{\mu}\right\} \leq\varepsilon.
\]
 Then, there exists a $\varpi>0$ such that $(\pi_{k+1},p_{k+1})$
is a $\varpi\varepsilon-$game stationary point.
\end{lemma}
\subsection*{Proof of~\cref{thm:stationary_result}}
\begin{theorem}\label{thm:appen_stationary_thm}
Under the setting of ~\cref{prop:prop1_un} and suppose
$\beta=\mcal O(K^{-1/2})$, then for any $K>0$, there exists a $k\leq K$
such that $(\pi^{k+1},p^{k+1})$ is an $\mcal O\brbra{(D_{\Pi}^{1/2} + D_{\mcal P}^{1/2})/K^{1/4}}$ game stationary point
and $\bar{\pi}^{k+1}$ is an $\mcal O\brbra{(D_{\Pi}^{1/2} + D_{\mcal P}^{1/2})/K^{1/4}}$ optimal stationary point, where $D_{\Pi}$ and $D_{\mcal P}$ are the diameter of set $\Pi$ and $\mcal P$, respectively.
\end{theorem}
\begin{proof}
It is easy to know that $\Phi(\pi,p,\bar{\pi},\bar{p})$ is lower
bounded by $\underline{\chi}$. Let 
\[
\iota:=\max\left\{ \frac{r_{1}}{32}\|\pi_{k+1}-\pi_{k}\|^{2},\frac{r_{2}}{15}\|p_{k}-p_{k}^{+}(\bar{\pi}_{k},\bar{p}_{k})\|^{2},\frac{r_{1}}{5\beta}\|\bar{\pi}_{k}-\bar{\pi}_{k+1}\|^{2},\frac{r_{2}}{4\mu}\|\bar{p}_{k}^{+}(\bar{\pi}_{k+1})-\bar{p}_{k}\|^{2}\ \right\} .
\]
 Then, we consider the following two cases separately: 

Case1: There exists $k\in[K-1]$ such that
\begin{equation*}
\frac{1}{2}\iota\leq4r_{1}\beta\|\pi(\bar{\pi}_{k+1},\bar{p}(\bar{\pi}_{k+1}))-\pi(\bar{\pi}_{k+1},\bar{p}_{k}^{+}(\bar{\pi}_{k+1}))\|^{2}.
\end{equation*}
It follows from ~\cref{prop:appen_:prop2} that 
\begin{equation*}
\|\bar{p}_{k}^{+}(\bar{\pi}_{k+1})-\bar{p}_{k}\|^{2}\leq\frac{32r_{1}\mu\beta}{r_{2}}\|\pi(\bar{\pi}_{k+1},\bar{p}(\bar{\pi}_{k+1}))-\pi(\bar{\pi}_{k+1},\bar{p}_{k}^{+}(\bar{\pi}_{k+1}))\|^{2}\leq\frac{32r_{1}\mu\beta\omega_{0}}{r_{2}}\|\bar{p}_{k}^{+}(\bar{\pi}_{k+1})-\bar{p}_{k}\|,
\end{equation*}
which implies that $\|\bar{p}_{k}^{+}(\bar{\pi}_{k+1})-\bar{p}_{k}\|\leq\varpi_{1}\beta$,
where $\ensuremath{\varpi_{1}:=32r_{1}\mu\omega_{0}/r_{2}}.$ Armed
with this, we can bound other terms as follows:
\begin{align*}
\frac{\|\pi_{k+1}-\pi_{k}\|^{2}}{\tau^{2}} & \leq\frac{256\beta}{\tau^{2}}\|\pi(\bar{\pi}_{k+1},\bar{p}(\bar{\pi}_{k+1}))-\pi(\bar{\pi}_{k+1},\bar{p}_{k}^{+}(\bar{\pi}_{k+1}))\|^{2}\\
 & \leq\frac{256\beta\omega_{0}}{\tau^{2}}\|\bar{p}_{k}^{+}(\bar{\pi}_{k+1})-\bar{p}_{k}\|=\varpi_{2}\beta^{2},\\
\frac{\|p_{k}-p_{k}^{+}(\bar{\pi}_{k},\bar{p}_{k})\|^{2}}{\sigma^{2}} & \leq\frac{120r_{1}\beta}{r_{2}\sigma^{2}}\|\pi(\bar{\pi}_{k+1},\bar{p}(\bar{\pi}_{k+1}))-\pi(\bar{\pi}_{k+1},\bar{p}_{k}^{+}(\bar{\pi}_{k+1}))\|^{2}\\
 & \leq\frac{120r_{1}\beta\omega_{0}}{r_{2}\sigma^{2}}\|\bar{p}_{k}^{+}(\bar{\pi}_{k+1})-\bar{p}_{k}\|=\varpi_{3}\beta,\\
\frac{\|\bar{\pi}_{k+1}-\bar{\pi}_{k}\|^{2}}{\beta^{2}} & \leq40\|\pi(\bar{\pi}_{k+1},\bar{p}(\bar{\pi}_{k+1}))-\pi(\bar{\pi}_{k+1},\bar{p}_{k}^{+}(\bar{\pi}_{k+1}))\|^{2}\leq40\omega_{0}\|z_{+}^{t}(\bar{p}_{k+1})-\bar{\pi}_{k}\|=\varpi_{4}\beta,\\
\frac{\|\bar{p}_{k+1}-\bar{p}_{k}\|^{2}}{\mu^{2}} & \leq\frac{2\|\bar{p}_{k}^{+}(\bar{\pi}_{k+1})-\bar{p}_{k}\|^{2}}{\mu^{2}}+2\Big(4\ell_{p\pi}^{2}\sigma^{2}\kappa_{6}^{2}\|\pi_{k}-\pi_{k+1}\|^{2}+4(\kappa_{8}+1)^{2}\|p_{k}-p_{k}^{+}(\bar{\pi}_{k},\bar{p}_{k})\|^{2}\Big)\\
 & \ \ +4\kappa_{3}^{2}\|\bar{\pi}_{k}-\bar{\pi}_{k+1}\|^{2}\\
 & \leq\frac{2\varpi_{1}^{2}}{\mu^{2}}\beta^{2}+8\ell_{p\pi}^{2}\sigma^{2}\kappa_{6}^{2}c^{2}\varpi_{2}\beta^{2}+8(\kappa_{8}+1)^{2}\sigma^{2}\varpi_{3}\beta^{2}+4\kappa_{3}^{2}\varpi_{3}\beta^{3}\leq\varpi_{5}\beta^{2},
\end{align*}
where $\varpi_{2}:=\frac{256\omega_{0}}{\tau^{2}}\varpi_{1}$, $\varpi_{3}:=\frac{120r_{1}\omega_{0}}{r_{2}\sigma^{2}}\varpi_{1}$,
$\varpi_{4}:=40\omega_{0}\varpi_{1}$ and $\varpi_{5}:=\frac{2\varpi_{1}^{2}}{\mu^{2}}+8\ell_{p\pi}^{2}\sigma^{2}\kappa_{6}^{2}\tau^{2}\varpi_{2}+8(\kappa_{8}+1)^{2}\sigma^{2}\varpi_{3}+4\kappa_{3}^{2}\varpi_{3}$.
According to~\cref{lem:lem8_un}, there exists a $\varpi>0$
such that $(\pi_{k+1},p_{k+1})$ is a $\varpi\varepsilon$-game stationary point, where
$\varepsilon=\max\{\sqrt{\varpi_{2}}\beta,\sqrt{\varpi_{3}}\beta,\sqrt{\varpi_{4}}\beta,\sqrt{\varpi_{5}}\beta^{\frac{1}{2}}\}=\mcal O(\beta^{1/2})$.
Next, we show that $\bar{\pi}_{k+1}$ is an $\mcal O(\beta^{1/2})$-optimization stationary point.
Note that
\begin{equation}
\begin{split} & \|\bar{\pi}_{k+1}-\pi^{*}(\bar{\pi}_{k+1})\|\\
\leq\  & \|\bar{\pi}_{k+1}-\bar{\pi}_{k}\|+\|\bar{\pi}_{k}-\pi_{k+1}\|+\|\pi_{k+1}-\pi(p_{k},\bar{\pi}_{k},\bar{p}_{k})\|+\|\pi(p_{k},\bar{\pi}_{k},\bar{p}_{k})-\pi(p(\bar{\pi}_{k},\bar{p}_{k}),\bar{\pi}_{k},\bar{p}_{k})\|\\
  &\ \ +\|\pi(\bar{\pi}_{k},\bar{p}_{k})-\pi(\bar{\pi}_{k+1},\bar{p}_{k})\|+\|\pi(\bar{\pi}_{k+1},\bar{p}_{k})-\pi^{*}(\bar{\pi}_{k+1})\|\\
\aleq\  & (1+\kappa_{2})\|\bar{\pi}_{k+1}-\bar{\pi}_{k}\|+\frac{\|\bar{\pi}_{k}-\bar{\pi}_{k+1}\|}{\beta}+\kappa_{6}\|\pi_{k}-\pi_{k+1}\|+\kappa_{1}\kappa_{8}\|p_{k}-p_{k}^{+}(\bar{\pi}_{k},\bar{p}_{k})\|\\
  &\ \ +\omega_{1}\|\bar{p}_{k}-p(\bar{\pi}_{k+1},\bar{p}_{k})\|^{\frac{1}{2}}\\
\bleq\  & (1+\kappa_{2})\sqrt{\varpi_{4}}\mcal O(\beta^{3/2})+\sqrt{\varpi_{4}}\mcal O(\beta^{1/2})+\left(\kappa_{6}\sqrt{\varpi_{2}}+\kappa_{1}\kappa_{8}\sqrt{\varpi_{3}}\right)\mcal O(\beta)+\omega_{1}\varpi_{1}\mcal O(\beta^{1/2})=\mcal O(\beta^{1/2}),
\end{split}
\label{eq:appen_os_trans}
\end{equation}
where $(a)$ holds by Item~\ref{enu:3} Item~\ref{enu:1} in~\cref{lem:lem2_un},
Item~\ref{item:a)} and Item~\ref{item:c)} in~\cref{lem:lem4_un}
and ~\cref{prop:appen_:prop2}, and $(b)$ holds by $(\pi_{k+1},p_{k+1})$
is a $\varpi\varepsilon$ game stationary point.

Case2: For any $k\in[K-1]$, we have 
\begin{equation*}
\frac{1}{2}\iota\geq4r_{1}\beta\|\pi(\bar{\pi}_{k+1},\bar{p}(\bar{\pi}_{k+1}))-\pi(\bar{\pi}_{k+1},\bar{p}_{k}^{+}(\bar{\pi}_{k+1}))\|^{2}.
\end{equation*}
Since 
\[
\begin{aligned} & \Phi(\pi_{k},p_{k},\bar{\pi}_{k},\bar{p}_{k})-\Phi(\pi_{k+1},p_{k+1},\bar{\pi}_{k+1},\bar{p}_{k+1})\\
\geq\  & \frac{r_{1}}{64}\|\pi_{k+1}-\pi_{k}\|^{2}+\frac{r_{2}}{30}\|p_{k}-p_{k}^{+}(\bar{\pi}_{k},\bar{p}_{k})\|^{2}+\frac{r_{1}}{10\beta}\|\bar{\pi}_{k}-\bar{\pi}_{k+1}\|^{2}+\frac{r_{2}}{8\mu}\|\bar{p}_{k}^{+}(\bar{\pi}_{k+1})-\bar{p}_{k}\|^{2}
\end{aligned}
\]
holds for $k\in[K-1]$ and $\Phi(\pi, p,\bar{\pi},\bar{p})\geq\underline{\chi}$,
we know that 
\[
\begin{aligned} & \Phi(\pi_{0},p_{0},\bar{\pi}_{0},\bar{p}_{0})-\underline{\chi}\\
  & \geq \sum_{k=0}^{K-1}\frac{r_{1}}{64}\|\pi_{k+1}-\pi_{k}\|^{2}+\frac{r_{2}}{30}\|p_{k}-p_{k}^{+}(\bar{\pi}_{k},\bar{p}_{k})\|^{2}+\frac{r_{1}}{10\beta}\|\bar{\pi}_{k}-\bar{\pi}_{k+1}\|^{2}+\frac{r_{2}}{8\mu}\|\bar{p}_{k}^{+}(\bar{\pi}_{k+1})-\bar{p}_{k}\|^{2}\\
&\geq  K\min\left\{ \frac{r_{1}\tau^{2}}{64},\frac{r_{2}\sigma^{2}}{30},\frac{r_{1}}{10},\frac{r_{2}\mu}{8}\right\} \left(\frac{\|\pi_{k+1}-\pi_{k}\|^{2}}{\tau^{2}}+\frac{\|p_{k}-p_{k}^{+}(\bar{\pi}_{k},\bar{p}_{k})\|^{2}}{\sigma^{2}}\right)\\
&\ \ + K\min\left\{ \frac{r_{1}\tau^{2}}{64},\frac{r_{2}\sigma^{2}}{30},\frac{r_{1}}{10},\frac{r_{2}\mu}{8}\right\} \left(\frac{\|\bar{\pi}_{k}-\bar{\pi}_{k+1}\|^{2}}{\beta}+\frac{\|\bar{p}_{k}^{+}(\bar{\pi}_{k+1})-\bar{p}_{k}\|^{2}}{\mu^{2}}\right).
\end{aligned}
\]
Since $\Phi(\pi, p,\bar{\pi},\bar{p})\geq\underline{\chi}$, there exists
a $k\in[K-1]$ such that 
\begin{equation*}
\begin{split} & \max\left\{ \frac{\|\pi_{k}-\pi_{k+1}\|^{2}}{\tau^{2}},\frac{\|p_{k}-p_{k}^{+}(\bar{\pi}_{k},\bar{p}_{k})\|^{2}}{\sigma^{2}},\frac{\|\bar{\pi}_{k}-\bar{\pi}_{k+1}\|^{2}}{\beta},\frac{\|\bar{p}_{k}-\bar{p}_{k+1}^{+}(\bar{\pi}_{k+1})\|^{2}}{\mu^{2}}\right\} \\
 & \leq\frac{\Phi(\pi_{0},p_{0},\bar{\pi}_{0},\bar{p}_{0})-\underline{\chi}}{K\min\left\{ \frac{r_{1}\tau^{2}}{64},\frac{r_{2}\sigma^{2}}{30},\frac{r_{1}}{10},\frac{r_{2}\mu}{8}\right\} }\aleq\frac{\eta}{K},
\end{split}
\end{equation*}
where $(a)$ holds by taking 
\begin{equation*}
\begin{aligned}
    \eta&:=\Brbra{\min\left\{ \frac{r_{1}\tau^{2}}{64},\frac{r_{2}\sigma^{2}}{30},\frac{r_{1}}{10},\frac{r_{2}\mu}{8}\right\}}^{-1}\brbra{L_{\pi}\norm{\pi_{0}-\pi^{*}}+L_{p}\norm{p_{0}-p^{*}}+L_{\pi}\norm{\pi(\bar{\pi}_{0})-\pi(p_{0},\bar{\pi}_{0},\bar{p}_{0})}+L_{p}\norm{p(\bar{\pi}_{0})-p_{0}}}\\
    &=\mcal O(D_{\Pi} + D_{\mcal P})
    \end{aligned}
\end{equation*} and~\eqref{eq:upper_bound_Lya}.
Note that 
\[
\begin{aligned}\frac{\|\bar{p}_{k+1}-\bar{p}_{k}\|^{2}}{\mu^{2}} & \leq\frac{2\|\bar{p}_{k}^{+}(\bar{\pi}_{k+1})-\bar{p}_{k}\|^{2}}{\mu^{2}}+4\kappa_{3}^{2}\|\bar{\pi}_{k}-\bar{\pi}_{k+1}\|^{2}\\
 & \ \ +2\Brbra{4\ell_{p\pi}^{2}\sigma^{2}\kappa_{6}^{2}\|\pi_{k}-\pi_{k+1}\|^{2}+4(\kappa_{8}+1)^{2}\|p_{k}-p_{k}^{+}(\bar{\pi}_{k},\bar{p}_{k})\|^{2}}\\
 & \leq\frac{\eta\left(2+8\ell_{p\pi}^{2}\sigma^{2}\kappa_{6}^{2}\tau^{2}+8(\kappa_{8}+1)^{2}\sigma^{2}+4\kappa_{3}^{2}\beta\right)}{K}\aleq\frac{\varpi_{6}}{K},
\end{aligned}
\]
where $\varpi_{6}=\eta\left(2+8\ell_{p\pi}^{2}\sigma^{2}\kappa_{6}^{2}\tau^{2}+8(\kappa_{8}+1)^{2}\sigma^{2}+4\kappa_{3}^{2} \right)$ and $(a)$ holds by $\beta<1$.
Thus, we know there exists a $\varpi>0$ such that $(\pi_{k+1},p_{k+1})$
is a $\varpi\varepsilon$ game stationary point, where $\varepsilon=\sqrt{\frac{\varpi_{6}}{K\beta}}$.
Similar argument as~\eqref{eq:appen_os_trans} shows that $\bar{\pi}^{k+1}$
is an $\mcal O(\omega_6^{1/2}K^{-1/2}\beta^{-1/2})$ optimal stationary point. 

Combining the two cases, we have the following conclusion: if we choose
$\beta=\mcal O(K^{-1/2})$, then optimal stationary point and game stationary point are coincide as the same
rate $\mcal O(\omega_6^{1/2}K^{-1/4})$. 
\end{proof}

\subsection*{Proof of~\cref{thm:global_opt} }
\begin{theorem}\label{thm:appen_global_opt}
    Under all settings of~\cref{thm:appen_stationary_thm}. Then for the sequence $\cbra{\bar{\pi}_k}_{k=1}^{K}$ generated by {\srpg}, there exists a $k<K$ such that $\bar{\pi}_{k+1}$ is $\mcal O\brbra{(D_{\Pi}^{1/2} + D_{\mcal P}^{1/2})/K^{1/4}}$ global optimal solution.
\end{theorem}
\begin{proof}
    It follows from~\cref{thm:appen_stationary_thm} that there exists a $\bar{\pi}_{k+1}$ such that $\norm{\prox_{\phi,r_1}(\bar{\pi}_{k+1})-\bar{\pi}_{k+1}}=\mcal O\brbra{(D_{\Pi}^{1/2} + D_{\mcal P}^{1/2})/K^{1/4}}$. Recall the definition of $\nabla \phi_{2\ell_{\pi}}(\pi) = \ell_{\pi} (\prox_{\phi,2\ell_{\pi}}(\pi) - \pi)$, then combining the result and~\cref{thm:Moreau_smooth_grad_dom}, we have there exists a $k<K$ such that $\phi(\bar{\pi}_{k+1})-\phi^* = \mcal O\brbra{(D_{\Pi}^{1/2} + D_{\mcal P}^{1/2})/K^{1/4}}$. 
\end{proof}

\subsection{Proof of Results in Section~\ref{sec:reg_exp}}
\begin{lemma}
For function $\phi(\pi)$ and the corresponding Moreau Envelope function
$\phi_{1/\ell_{\pi}}(\pi)$, we have 

\begin{enumerate}[label=\Alph*)]

\item \label{item:A)}$\ensuremath{\phi(\pi)-\phi_{2\ell_{\pi}}(x)\geq\ell_{\pi}\norm{\pi-\tilde{\pi}(\pi)}^{2}/2}$,
where $\tilde{\pi}(\pi)\in\argmin_{\pi'\in\Pi}\{\phi(\pi')+\ell_{\pi}\norm{\pi'-\pi}^{2}\}$.

\item \label{item:B)} $\Pi^{*}=\Pi_{\ell_{\pi}}^{*}$ where $\Pi^{*}:=\argmin_{\pi\in\Pi}\phi(\pi)$
and $\Pi_{\ell_{\pi}}^{*}:=\argmin_{\pi\in\Pi}\phi_{2\ell_{\pi}}(\pi)$

\item \label{item:C)}$\phi(\pi^{*})=\phi_{2\ell_{\pi}}(\pi^{*})$
holds for any $\pi^{*}\in\Pi^{*}$.

\end{enumerate}
\end{lemma}
\begin{proof}
Item~\ref{item:A)}: Note that:
\begin{equation}
\begin{split} & \phi(\pi)-\phi_{2\ell_{\pi}}(\pi)\\
 & =\phi(\pi)+\ell_{\pi}\norm{\pi-\pi}^{2}-(\phi(\tilde{\pi}(\pi))+\ell_{\pi}\norm{\tilde{\pi}(\pi)-\pi}^{2})\\
 & \ageq\frac{\ell_{\pi}}{2}\norm{\tilde{\pi}(\pi)-x}^{2},
\end{split}
\label{eq:appen_smoothing_g_dom}
\end{equation}
where $(a)$ holds by $\phi(\pi')+\ell_{\pi}\norm{\pi'-\pi}^{2}$
is $\ell_{\pi}$-strongly convex with respect to $\pi'$.

Item~\ref{item:B)}: we consider $\Pi_{\ell_{\pi}}^{*}\subseteq\Pi^{*}$
and $\Pi^{*}\subseteq\Pi_{\ell_{\pi}}^{*}$, separately. $\Pi_{\ell_{\pi}}^{*}\subseteq\Pi^{*}$:
for any $\pi^{*}\in\argmin_{\pi\in\Pi}\phi_{2\ell_{\pi}}(\pi)$,
we have 
\begin{equation*}
\norm{\nabla\phi_{2\ell_{\pi}}(\pi^{*})}=4\ell_{\pi}^{2}\norm{\tilde{\pi}(\pi^{*})-\pi^{*}}^{2}=0,
\end{equation*}
 i.e., $\tilde{\pi}(\pi^{*})=\pi^{*}$, where $\tilde{\pi}(\pi^{*})\in\argmin_{\pi'\in\Pi}\{\phi(\pi')+\ell_{\pi}\norm{\pi'-\pi^{*}}^{2}\}$,
which implies that 
\begin{equation*}
\tilde{\pi}(\pi^{*})\in\argmin_{\pi'\in\Pi}\{\phi(\pi')\}.
\end{equation*}
 Hence $\argmin_{\pi\in\Pi}\phi_{2\ell_{\pi}}(\pi)\subseteq\argmin_{\pi\in\Pi}\phi(\pi)$. 

$\Pi^{*}\subseteq\Pi_{\ell_{\pi}}^{*}$: For any $\pi^{*}\in\argmin_{\pi\in\Pi}\phi(\pi)$,
we have 
\begin{equation}
\phi(\pi)\geq\phi(\pi^{*})-\ell_{\pi}\norm{\pi-\pi^{*}}/2,\label{eq:appen_temp1-1}
\end{equation}
Taking $\pi=\tilde{\pi}(\pi^{*})$, where $\tilde{\pi}(\pi^{*})\in\argmin_{\pi'\in\Pi}\{\phi(\pi')+L\norm{\pi'-\pi^{*}}^{2}\}$,
then we have
\begin{equation}
\phi(\tilde{\pi}(\pi^{*}))\geq\phi(\pi^{*})-\ell_{\pi}\norm{\tilde{\pi}(\pi^{*})-\pi^{*}}/2.\label{eq:appen_temp2}
\end{equation}
 It follows from the definition of $\tilde{\pi}(\pi^{*})$ that $\phi(\pi^{*})\geq\phi(\tilde{\pi}(\pi^{*}))+\ell_{\pi}\norm{\tilde{\pi}(\pi^{*})-\pi^{*}}^{2}$.
Putting~\eqref{eq:appen_temp1-1} and~\eqref{eq:appen_temp2} together, we have
\begin{equation*}
\phi(\pi^{*})-\ell_{\pi}\norm{\tilde{\pi}(\pi^{*})-\pi^{*}}/2\leq\phi(\tilde{\pi}(\pi^{*}))\leq\phi(\pi^{*})-\ell_{\pi}\norm{\tilde{\pi}(\pi^{*})-\pi^{*}}^{2}.
\end{equation*}
 Hence we have $\tilde{\pi}(\pi^{*})=\pi^{*}$, which implies that
$\norm{\nabla\phi_{2\ell_{\pi}}(\tilde{\pi}(\pi^{*}))}=\norm{\nabla\phi_{2\ell_{\pi}}(\pi^{*})}=4L^{2}\norm{\tilde{\pi}(\pi^{*})-\pi^{*}}^{2}=0$.
Note that
\begin{equation}
\phi_{2\ell_{\pi}}(\pi^{*})=\phi(\tilde{\pi}(\pi^{*}))+\ell_{\pi}\norm{\tilde{\pi}(\pi^{*})-\pi^{*}}^{2}=\phi(\pi^{*}),\label{eq:appen_equal}
\end{equation}
 and it follows from Lemma~\ref{thm:Moreau_smooth_grad_dom} that
\begin{equation}
\phi(\pi)-\phi(\pi^{*})\leq C_{\ell\pi}\norm{\nabla\phi_{2\ell_{\pi}}(\pi)},\ \ \forall\pi\in\Pi.\label{eq:appen_phi_gradient_dom}
\end{equation}
 Combining~\eqref{eq:appen_equal}, $\phi(\pi^{*})\leq\phi(\tilde{\pi}(\pi)),\forall\pi\in\Pi$
and the definition of $\phi_{2\ell_{\pi}}(\pi)$ yields
\begin{equation*}
0\leq\phi(\tilde{\pi}(\pi))+\ell_{\pi}\norm{\tilde{\pi}(\pi)-\pi}^{2}-\phi(\pi^{*})=\phi_{2\ell_{\pi}}(\pi)-\phi_{2\ell_{\pi}}(\pi^{*}),\ \ \forall\pi\in\Pi,
\end{equation*}
 which implies that $\pi^{*}\in\argmin_{\pi\in\Pi}\phi_{2\ell_{\pi}}(\pi)$.

Item~\ref{item:C)}: It follows from~\eqref{eq:appen_equal} that Item~\ref{item:C)}
holds.
\end{proof}
\subsection*{Proof of~\cref{lem:regional_exp_grow}}
\begin{lemma}
[Regional Exponential Growth]\label{lem:appen_regional_exp_grow}Suppose all assumptions of~\cref{thm:Moreau_smooth_grad_dom}
holds. Then for $\ell_{\pi}$-weakly convex function $\phi(\pi)$,
we have
\begin{equation*}
\phi(\pi)-\phi^{*}\geq\zeta\exp\Brbra{\dist(\pi,\Pi_{\ell_{\pi}}^{*}(\zeta))/C_{\ell\pi}},\ \ \pi\in\Pi\setminus\Pi^*_{\ell_{\pi}}(\zeta)
\end{equation*}
where $\Pi_{\ell_{\pi}}^{*}(\zeta)=\{x\mid\phi_{2\ell_{\pi}}(\pi)-\phi^{*}<\zeta\}$
and $\zeta>0$ is a constant.
\end{lemma}
\begin{proof}
Our proof is similar to~\citet{bolte2017error,karimi2016linear}. 
Define the function $g(\pi)=\ln\Brbra{\phi_{2\ell_{\pi}}(\pi)-\phi^{*}},\pi\in\Pi\setminus\Pi_{\ell_{\pi}}^{*}(\zeta)$,
we have 
\begin{equation}
\norm{\nabla g(\pi)}=\Bnorm{\frac{\nabla\phi_{2\ell_{\pi}}(\pi)}{\phi_{2\ell_{\pi}}(\pi)-\phi^{*}}}\ageq\Bnorm{\frac{\nabla\phi_{2\ell_{\pi}}(\pi)}{\phi(\pi)-\phi^{*}}}\bgeq\frac{1}{C_{\ell\pi}},\label{eq:appen_g_below}
\end{equation}
where $(a)$ holds by Item~\ref{item:A)} and $(b)$ follows from~\cref{thm:Moreau_smooth_grad_dom}.
For any point $\pi\in\Pi\setminus\Pi^*_{\ell_{\pi}}(\zeta)$, consider
the following differential equation: 
\begin{align*}
\frac{d\pi(t)}{dt} & =-\nabla g(\pi(t)),\\
x(t=0) & =\pi_{0},\ \ \pi(t)\in\Pi\setminus\Pi_{\ell_{\pi}}^{*}(\zeta)
\end{align*}
By~\eqref{eq:appen_g_below}, $\norm{\nabla g}$ is bounded from below,
and as $g$ is also bounded from below $\ln\zeta$, since $\pi\in\Pi\setminus\Pi_{\ell_{\pi}}^{*}(\zeta)$.
Thus, by moving along the path defined by above, we are sufficiently
reducing and will eventually reach the set $\Pi_{\ell_{\pi}}^{*}(\zeta)$.
Hence, there exists a $T$ such that $\pi(t)\in\Pi_{\ell_{\pi}}^{*}(\zeta),\forall t \leq T$. Now,
we can show this by using the steps 
\begin{equation*}
\begin{split}g(\pi_{0})-g(\pi_{T}) & =\int_{\pi_{T}}^{\pi_{0}}\inner{\nabla g(x)}{dx}=-\int_{\pi_{0}}^{\pi_{T}}\inner{\nabla g(x)}{dx}\\
 & =-\int_{_{0}}^{T}\inner{\nabla g(\pi(t))}{\frac{d\pi(t)}{dt}}dt=\int_{0}^{T}\norm{\nabla g(\pi(t))}^{2}dt
\end{split}
\end{equation*}
The length of the orbit $\pi(t)$ starting at $\pi_{0}$, which will
be denoted by $\mcal L(\pi_{0})$ is given by 
\begin{equation*}
\mcal L(\pi_{0})=\int_{0}^{T}\norm{d\pi(t)/dt}dt=\int_{0}^{T}\norm{\nabla g(\pi(t))}dt\geq\dist(\pi_{0},\Pi_{\ell_{\pi}}^{*}(\zeta)).
\end{equation*}
Hence, we have
\begin{equation*}
\begin{split}g(\pi_{0})-g(\pi_{T}) = \int_{0}^{T}\norm{\nabla g(\pi(t))}^{2}dt
 \geq\frac{1}{C_{\ell\pi}}\int_{0}^{T}\norm{\nabla g(\pi(t))}dt
 \geq\dist(\pi_{0},\Pi_{\ell_{\pi}}^{*}(\zeta))/C_{\ell\pi}.
\end{split}
\end{equation*}
 Since $g(\pi_{T})=\ln\zeta$, this yields 
\begin{equation*}
\ln\Brbra{\phi_{2\ell_{\pi}}(\pi_{0})-\phi^{*}}=g(\pi_{0})\geq\ln\zeta+\dist(\pi_{0},\Pi_{\ell_{\pi}}^{*}(\zeta))/C_{\ell\pi}.
\end{equation*}
 The above result implies that for any $\pi\in\Pi\setminus\Pi_{\ell_{\pi}}^{*}(\zeta)$,
we have 
\begin{equation*}
\phi(\pi)-\phi^{*}\geq\phi_{2\ell_{\pi}}(\pi)-\phi^{*}\geq\zeta\exp\Brbra{\dist(\pi,\Pi_{\ell_{\pi}}^{*}(\zeta))/C_{\ell\pi}}.
\end{equation*}
\end{proof}

\subsection*{Proof of~\cref{prop:region_growth}}
\begin{proposition}
    Suppose all the assumptions of Lemma~\ref{thm:Moreau_smooth_grad_dom} holds. Then for the sequence $\cbra{\bar{\pi}_k}$ generated by {\srpg}, we have 
    \begin{equation*}
        \dist(\bar{\pi}_k, \Pi^*_{\ell_{\pi}}(\vep)) = \mcal O(\log \rbra{(D_{\pi}^{1/2} + D_{\mcal P}^{1/2})k^{-0.25}\vep^{-1})},\forall \bar{\pi}_k\in \Pi\setminus \Pi^*_{\ell_{\pi}}(\vep). 
    \end{equation*}
\end{proposition}

\begin{proof}
    Take $\zeta=\vep$ in~\cref{lem:appen_regional_exp_grow}, and combine it with the result in~\cref{thm:appen_global_opt}, we have there exist a constant $\mathfrak{C}$ such that
    \begin{equation}
        \vep\cdot \exp\Brbra{\dist(\bar{\pi}_k,\Pi_{\ell_\pi}^*(\vep))/C_{\ell_{\pi}}}\leq \frac{\mathfrak{C}\cdot (D_{\Pi}^{1/2} + D_{\mcal P}^{1/2})}{k^{0.25}},
    \end{equation}
    which implies the finally result $\dist(\bar{\pi}_k,\Pi_{\ell_\pi}^*(\vep))=\mcal O\Brbra{ \log\brbra{(D_{\Pi}^{1/2} + D_{\mcal P}^{1/2}) k^{-0.25}\vep^{-1}}}$ for $\bar{\pi}_k \in \Pi\setminus \Pi^*_{\ell_{\pi}}(\vep)$.
\end{proof}

%% file: icml_experiment_appendix.tex
\section{Experiment Details\label{sec:appendix_num}}

\subsection{Inventory management problem}
We parameterize nominal transition kernel by  $\theta\in \mbb R^{m}$ and $\lambda\in \mbb R^{n}_{++}$. The $L_1$-norm ambiguity set is defined by
{\footnotesize
\begin{equation*}
\Xi=\left\{ (\theta,\lambda)\vert \begin{array}{c}
\theta\in\mbb R^{m},\lambda\in\mbb R_{++}^{n},\\
\norm{\theta-\theta_{c}}_{1}\leq\kappa_{\theta},\norm{\lambda-\lambda_{c}}_{1}\leq\kappa_{\lambda}
\end{array}\right\}, 
\end{equation*}
}
where $\theta_c,\lambda_c,\kappa_{\theta}$ and $\kappa_{\lambda}$ are pre-specified parameters.
For any $(s, a, s')$, we parameterize transition kernel with:
$$p_{sas'}^{\xi}:=\frac{\bar{p}_{sas'}\cdot\exp(\theta^{\top}\phi_{\theta}(s')/(\lambda^{\top}\phi_{\lambda}(s,a)))}{\sum_{s'\in\mcal S}\bar{p}_{sas'}\cdot\exp(\theta^{\top}\phi_{\theta}(s')/(\lambda^{\top}\phi_{\lambda}(s,a)))},$$
where $\xi:=(\theta,\lambda)$, $\theta\in\mbb R^{m},\phi_{\theta}(s)\in\mbb R^{m},\lambda\in\mbb R^{n}_{++},\phi_{\lambda}(s,a)\in\mbb R^{n}$, $m$ and $n$ are hyperparameters. We term function $\phi_{\theta}(\cdot)$
as state featuring function and $\phi_{\lambda}(\cdot)$ as state-action featuring function. We use radial-type features~\cite{sutton2018reinforcement}:
\begin{equation*}
\begin{aligned}
        [\phi_{\theta}(s)]_i &=\frac{1}{\sqrt{2\pi}\sigma_{\theta}}\exp{\Bcbra{-\frac{\norm{s-[c_\theta]_i}^2}{2\sigma_{\theta}^2}}},\forall i\in [m],\\
    [\phi_{\lambda}(s,a)]_i &=\frac{1}{\sqrt{2\pi}\sigma_{\lambda}}\exp{\Bcbra{-\frac{\norm{s-[c_{\lambda,s}]_i}^2 + \norm{a - [c_{\lambda,a}]_i}^2}{2\sigma_{\lambda}^2}}},\forall i \in [n].
\end{aligned}
\end{equation*}

Based on the above model, we give the specific parameter setting in our experiments: $S = 8, A = 3, b = 5, \gamma = 0.95, \theta_c = [0.4,0.9]^\top\in \mbb R^2, \lambda_c = [0.7,0.6]^\top\in \mbb R^2$, every state is a two dimension vector (storing and selling), we set $\mcal S:=\bcbra{(0.25, 1.3), (0.5, -2.1), (0.75, 3.4), (1,-1), (0.25, 2.5), (0.5, 0.5), (0.75,1.8), (1,-0.8)}$. Every action (ordering) is one dimension, we set $\mcal A:=\cbra{-3,-1,5}$. We set $\kappa_{\lambda} = \kappa_\theta = 1, [c_\theta]_1 = [-1, 2]^{\top}, [c_\theta]_2 = [0.3, -0.6]^\top, [c_{\lambda,s}]_1 = [1.3,2.1]^\top, [c_{\lambda,s}]_2 = [-0.7,1.5]^\top, [c_{\lambda,a}]_1 = 1, [c_{\lambda,a}]_2 = 0.5, \sigma_\theta = 1, \sigma_\lambda = 2$.

Now, we give more details on the gradient.
Let $\lambda_{sa}=(\lambda^{\top}\phi_{\lambda}(s,a)),\partial\lambda_{sa}/\partial\lambda_{i}=[\phi_{\lambda}(s,a)]_{i}$, then the partial gradient are
\begin{equation}
    \begin{aligned}\label{eq:theta_lambda_mid}
        \frac{\partial\log p_{sas'}^{\xi}}{\partial\theta_{i}}&=	\Brbra{\frac{\phi_{\theta i}(s')}{\lambda_{sa}}-\sum_{k}p_{sak}^{\xi}\cdot\frac{\phi_{\theta i}(k)}{\lambda_{sa}}}\\
\frac{\partial\log p_{sas'}^{\xi}}{\partial\lambda_{i}}&=	\Brbra{\Brbra{\sum_{k}p_{sak}^{\xi}\cdot\frac{\theta^{\top}\phi_{\theta}(k)}{\lambda_{sa}^{2}}}-\frac{\theta^{\top}\phi(s')}{\lambda_{sa}^{2}}}\cdot[\phi_{\lambda}(s,a)]_{i}
    \end{aligned}
\end{equation}

Moreover, we have
\begin{equation}\label{eq:J_partial_xi}
\frac{J_{\rho}(\pi,\xi)}{\partial\xi}=\frac{1}{1-\gamma}\mbb E_{s\sim d_{\rho}^{\pi,\xi},
a\sim\pi_{s},
s'\sim p_{sas'}}\Bsbra{\frac{\partial\log p_{sas'}^{\xi}}{\partial\xi}\brbra{c_{sas'}+\gamma v_{s'}^{\pi,\xi}}},
\end{equation}
where $d_{\rho}^{\pi,\xi}=(1-\gamma)\sum_{s\in\mcal S}\sum_{t=0}^{\infty}\gamma^{t}\rho(s)p_{ss'}^{\pi}(t)$ and $\xi:=(\theta, \lambda)$.
Combining~\eqref{eq:theta_lambda_mid} and~\eqref{eq:J_partial_xi} yields
\begin{equation*}
    \begin{aligned}
        \frac{\partial J_{\rho}(\pi,\xi)}{\partial \theta_i} &= \frac{1}{1-\gamma}\sum_{s\in \mcal S} d^{\pi,\xi}_s\sum_{a}\pi_{sa}\sum_{s'\in\mcal S}p_{sas'}^{\xi}\Bsbra{\Brbra{\frac{\phi_{\theta i}(s')}{\lambda_{sa}}- \sum_{k\in \mcal S} p_{sak}^{\xi} \cdot \frac{\phi_{\theta i}(k)}{\lambda_{sa}}}\cdot\brbra{c_{sas'}+\gamma v_{s'}^{\pi,\xi}}}\\
        \frac{\partial J_{\rho}(\pi,\xi)}{\partial \lambda_{sa}} &=  \frac{1}{1-\gamma}\sum_{s\in \mcal S} d^{\pi,\xi}_s\sum_{a}\pi_{sa}\sum_{s'\in\mcal S}p_{sas'}^{\xi} \Bsbra{\Brbra{\Brbra{\sum_{k\in \mcal S}p_{sak}^{\xi}\cdot\frac{\theta^{\top}\phi_{\theta}(k)}{\lambda_{sa}^{2}}}-\frac{\theta^{\top}\phi(s')}{\lambda_{sa}^{2}}}\cdot \phi_{\lambda i}(sa)\cdot \brbra{c_{sas'}+\gamma v_{s'}^{\pi,\xi}}}
    \end{aligned}
\end{equation*}

On the other hand, we have $\frac{\partial J_{\rho}(\pi, p)}{\partial \pi_{sa}} = \frac{1}{1-\gamma}\cdot d_{\rho}^{\pi, p}(s)\cdot q_{sa}^{\pi, p}$, where $d_\rho^{\pi,p}$ and $q^{\pi,p}$ are calculated by the following:
\begin{equation}
    \begin{aligned}
        d_\rho^{\pi,p} &= (1-\gamma)\cdot (I - \gamma P_{\pi})^{-1} \rho, \text{ $P_{\pi}$ is the state transition matrix under policy $\pi$}\\
        q_{sa}^{\pi, p} &= \mbb E_{\pi, p} \Bsbra{\frac{1}{1- \gamma}\cdot c_{sas}\mid s_0 = s, a_0 = a}.
    \end{aligned}
\end{equation}

Furthermore, for parameterized policy experiments, we consider the following softmax parameterized method:
\begin{equation}
    \pi_{sa}=\frac{\exp(w^\top \phi_{\lambda}(s,a))}{\sum_{a'\in \mcal A}\exp(w^\top \phi_{\lambda}(s,a'))}.
\end{equation}